\documentclass[10pt]{article}
\usepackage{etex}
\usepackage{mathtools}
\usepackage{enumitem}

\usepackage{amsmath,amsthm}
\usepackage{csquotes}
\usepackage{todonotes}

\makeatletter
\let\orgdescriptionlabel\descriptionlabel
\renewcommand*{\descriptionlabel}[1]{
	\let\orglabel\label
	\let\label\@gobble
	\phantomsection
	\edef\@currentlabel{#1} 	
	\let\label\orglabel
	\orgdescriptionlabel{#1}
}
 											
\def\th@plain{
	\thm@notefont{}
	\itshape
}
\def\th@definition{
	\thm@notefont{}
	\normalfont
}

\g@addto@macro\th@remark{\thm@headpunct{}}
\g@addto@macro\th@definition{\thm@headpunct{}}
\g@addto@macro\th@plain{\thm@headpunct{}}
\makeatother

\usepackage{amssymb}

\usepackage{graphicx}

\usepackage{subfig}
\usepackage[final]{showkeys}

\usepackage{etoolbox}

\usepackage{wasysym}

\usepackage{mathrsfs}

\usepackage{mathpazo}

\usepackage[titletoc]{appendix}
\usepackage[doc]{optional}

\usepackage{soul}

\usepackage{colortbl,booktabs,sectsty,multirow}
\usepackage{xcolor}

\usepackage{cancel}

\usepackage{empheq}

\definecolor{myblue}{rgb}{.8, .8, 1}
  \newcommand*\mybluebox[1]{
    \colorbox{myblue}{\hspace{1em}#1\hspace{1em}}}

\usepackage[obeyspaces,hyphens,spaces]{url}

\usepackage{hyperref}
\hypersetup{
    colorlinks=true,
    linkcolor=blue,
    citecolor=blue,
    filecolor=magenta,
    urlcolor=cyan
}

\usepackage[
  open,
  openlevel=2,
  atend,
  numbered
]{bookmark}

\usepackage[capitalize, nameinlink, noabbrev]{cleveref}
\crefname{equation}{}{}
\crefname{chapter}{Chapter}{Chapters}
\crefname{item}{item}{items}
\crefname{figure}{Figure}{Figures}
\crefname{theorem}{Theorem}{Theorems}
\crefname{lemma}{Lemma}{Lemmas}
\crefname{proposition}{Proposition}{Propositions}
\crefname{corollary}{Corollary}{Corollarys}
\crefname{definition}{Definition}{Definitions}
\crefname{fact}{Fact}{Facts}
\crefname{example}{Example}{Examples}
\crefname{algorithm}{Algorithm}{Algorithms}
\crefname{remark}{Remark}{Remarks}
\crefname{note}{Note}{Notes}
\crefname{notation}{Notation}{Notations}
\crefname{case}{Case}{Cases}
\crefname{exercise}{Exercise}{Exercises}
\crefname{question}{Question}{Questions}
\crefname{claim}{Claim}{Claims}
\crefname{enumi}{}{}

\usepackage[top= 2cm, bottom = 2 cm, left = 2.2 cm, right= 2.2 cm]{geometry}

\usepackage{float}

\usepackage{pgf}

\usepackage{array}
\usepackage{tabu}

\allowdisplaybreaks

\numberwithin{equation}{section}

\theoremstyle{plain}
\newtheorem{theorem}{Theorem}[section]

\newtheorem{corollary}[theorem]{Corollary}
\newtheorem{fact}[theorem]{Fact}
\newtheorem{lemma}[theorem]{Lemma}
\newtheorem{proposition}[theorem]{Proposition}

\theoremstyle{definition}
\newtheorem{definition}[theorem]{Definition}
\newtheorem{example}[theorem]{Example}

\newtheorem{remark}[theorem]{Remark}

\newcommand{\aff}{\ensuremath{\operatorname{aff} \,}}

\newcommand{\Fix}{\ensuremath{\operatorname{Fix}}}
\newcommand{\Id}{\ensuremath{\operatorname{Id}}}

\newcommand{\dist}{\ensuremath{\operatorname{d}}}

\newcommand{\Pro}{\ensuremath{\operatorname{P}}}
\newcommand{\R}{\ensuremath{\operatorname{R}}}
\newcommand{\I}{\ensuremath{\operatorname{I}}}

\newcommand{\pa}{\ensuremath{\operatorname{par}}}

\newcommand{\CCO}[1]{CC{#1}}

\newcommand{\CC}[1]{CC_{#1}}

\providecommand{\norm}[1]{\lVert#1\rVert}
\providecommand{\Norm}[1]{{\Big\lVert}#1{\Big\rVert}}

\providecommand{\innp}[1]{\langle#1\rangle}
\providecommand{\Innp}[1]{\Big\langle#1\Big\rangle}

\begin{document}

\title{ \sffamily  Best approximation mappings  in Hilbert spaces}

\author{
         Heinz H.\ Bauschke\thanks{
                 Mathematics, University of British Columbia, Kelowna, B.C.\ V1V~1V7, Canada.
                 E-mail: \href{mailto:heinz.bauschke@ubc.ca}{\texttt{heinz.bauschke@ubc.ca}}.},~
         Hui\ Ouyang\thanks{
                 Mathematics, University of British Columbia, Kelowna, B.C.\ V1V~1V7, Canada.
                 E-mail: \href{mailto:hui.ouyang@alumni.ubc.ca}{\texttt{hui.ouyang@alumni.ubc.ca}}.},~
         and Xianfu\ Wang\thanks{
                 Mathematics, University of British Columbia, Kelowna, B.C.\ V1V~1V7, Canada.
                 E-mail: \href{mailto:shawn.wang@ubc.ca}{\texttt{shawn.wang@ubc.ca}}.}
                 }

\date{June 2, 2020}

\maketitle

\begin{abstract}
\noindent
The notion of  best approximation mapping (BAM) with respect to  a closed affine subspace in finite-dimensional space was introduced by Behling, Bello Cruz and Santos to show the linear convergence of the block-wise circumcentered-reflection method. The best approximation mapping possesses  two critical properties of the circumcenter mapping for linear convergence.

Because the iteration sequence of BAM linearly converges, the BAM is interesting in its own right. In this paper, we  naturally extend the definition of BAM from closed affine subspace to nonempty closed convex set and from $\mathbb{R}^{n}$ to general Hilbert space.  We discover that the convex set associated with the BAM must be the fixed point set of the BAM. Hence, the iteration sequence generated by a BAM  linearly converges to the nearest fixed point of the BAM. Connections between BAMs and other mappings generating convergent iteration sequences are considered.
Behling et al.\ proved that the finite composition of BAMs associated with closed affine subspaces is  still a BAM in $\mathbb{R}^{n}$. We generalize their result from $\mathbb{R}^{n}$ to   general Hilbert space and also construct a new constant associated with the composition of BAMs. This provides a new proof of  the linear convergence of the method of alternating projections.
Moreover, compositions of BAMs associated with general convex sets are investigated.
In addition, we  show that convex combinations of BAMs associated with affine subspaces are BAMs.
Last but not least, we connect BAM with circumcenter mapping in Hilbert spaces.
\end{abstract}

{\small

}

%%%%%%%%%%%%%%%%%%%%%%%%%%%%%%%%%%%%%%%%%%%%%%%%%%%%%%%%%%%%%%%%%%%
%%%%%%%%%%%%%%%%\section{Introduction} %%%%%%%%%%%%%%%%%%%%%%%%%%%%%%%%%%%
%%%%%%%%%%%%%%%%%%%%%%%%%%%%%%%%%%%%%%%%%%%%%%%%%%%%%%%%%%%%%%%%%%%
\section{Introduction} \label{sec:Introduction}

Throughout this paper, we shall assume that
\begin{empheq}[box = \mybluebox]{equation*}
\text{$\mathcal{H}$ is a real Hilbert space},
\end{empheq}
with inner product $\innp{\cdot,\cdot}$ and induced norm $\|\cdot\|$, 
$\mathbb{N}=\{0,1,2,\ldots\}$ and $m \in \mathbb{N} \smallsetminus \{0\}$.

In \cite{BCS2017}, Behling, Bello Cruz and Santos introduced the circumcentered Douglas-Rachford method, which is a special instance of the circumcentered-reflection method (C-RM) and the first circumcentered isometry method in the literature. Then the same authors contributed \cite{BCS2018}, \cite{BCS2019} and
\cite{BCS2020ConvexFeasibility} on C-RMs.
In \cite{BCS2019},
in order to prove the linear convergence of the block-wise C-RM that is the sequence of iterations of finite composition of circumcentered-reflection operators, they introduced the best approximation mapping (BAM) and proved that the finite composition of BAMs is still a BAM.   Our paper is inspired by \cite{BCS2019}, and we provide the following main results:

\begin{enumerate}
	\item[\textbf{R1:}]  \cref{prop:BAM:Properties} states that the sequence of iterations of BAM solves the best approximation problem associated with the fixed point set of the BAM.
	
	\item[\textbf{R2:}] \cref{thm:BAM:COMPO} generalizes  \cite[Theorem~1]{BCS2019} and  shows that  the finite composition of BAMs associated with closed affine subspaces in Hilbert space is  a BAM. It also provides a new constant associated with the composition of BAMs. In fact, we provide examples showing that our new constant is independent with the one constructed in \cite[Lemma~1]{BCS2019}.
	In particular, as a corollary of the \cref{thm:BAM:COMPO},   in \cref{cor:BAM:COMPO:Pro}\cref{cor:BAM:COMPO:Pro:BAM}  we show the linear convergence of the method of alternating projections (MAP).
	
	\item[\textbf{R3:}] \Cref{theorem:convex:comb:BAM:2,theor:averacomb:BAM} use two different methods to show that the convex combination of finitely many BAMs associated with  affine subspaces is a BAM.
		
	\item[\textbf{R4:}] \Cref{theom:CCS1:compose:comb,theom:CCS2:compose:comb,theorem:Hilbert:BAM} show linear convergence of the iteration sequences generated from composition and convex combination of   circumcenter mappings in Hilbert spaces.
\end{enumerate}

The paper is organized as follows.  In \cref{sec:Preliminaries}, we present some auxiliary results to be used in the sequel.
\cref{sec:BAM} includes definition and properties of the BAM  in Hilbert spaces.  In particular, the comparisons: BAM vs convergent mapping,  BAM vs Banach contraction, and BAM vs linear regular operator are provided.
In \cref{sec:Composition:BAM}, we generalize results shown in \cite[Section~2]{BCS2019} from $\mathbb{R}^{n}$ to the general Hilbert space and show that the finite composition of BAMs with closed and affine fixed point sets in Hilbert space  is still a BAM.  In addition, compositions of BAMs associated with general convex sets are considered in \cref{sec:Composition:BAM} as well.
In \cref{sec:combinationBAM}, we use two methods to show that the convex combination of finitely many BAMs with closed and affine fixed point sets  is a BAM.
In \cref{sec:BAMandCIM}, we review definitions and facts on circumcenter mapping and circumcentered isometry methods. We also  provide sufficient conditions for the circumcenter mapping to be a BAM in Hilbert spaces. Moreover, we show linear convergence of sequences generated from composition and convex combination of   circumcenter mappings as BAMs in Hilbert spaces.

We now turn to the notation used in this paper. Let $C$ be a nonempty subset
of $\mathcal{H}$.
The \emph{orthogonal complement} of $C$ is the set $ C^{\perp} :=\{x \in \mathcal{H}~|~ \innp{x,y}=0
~\text{for all}~y \in C\}. $
$C$ is an \emph{affine subspace} of
$\mathcal{H}$ if $C \neq \varnothing$ and $(\forall \rho\in\mathbb{R})$ $\rho
C + (1-\rho)C = C$. The smallest affine subspace of $\mathcal{H}$ containing $C$ is
denoted by $\aff C$ and called the \emph{affine hull} of $C$.
An affine subspace $C$ is said to be \emph{parallel} to an affine subspace $
M $ if $C = M +a $ for some $ a \in \mathcal{H}$.
Suppose that  $C$ is a nonempty closed convex subset of $\mathcal{H}$. The \emph{projector} (or \emph{projection operator}) onto $C$ is the operator, denoted by
$\Pro_{C}$,  that maps every point in $\mathcal{H}$ to its unique projection onto $C$. $\R_{C} :=2 \Pro_{C} -\Id$ is the \emph{reflector} associated with $C$.
Moreover, $(\forall x \in \mathcal{H})$ $\dist_{C} (x) :=\min_{c \in C} \norm{x-c} =\norm{x -\Pro_{C}x}$. Let $x \in \mathcal{H}$ and $\rho \in \mathbb{R}_{++}$. Denote the ball centered at $x$ with radius $\rho$ as $\mathbf{B}[x;\rho ]$.

Let $T: \mathcal{H} \rightarrow \mathcal{H}$ be an operator.
The \emph{fixed
point set} of the operator $T$ is denoted by $\Fix T$, i.e., $\Fix T := \{x \in
\mathcal{H} ~|~ Tx=x\}$.
Denote by $\mathcal{B} (\mathcal{H}) := \{ T: \mathcal{H} \rightarrow \mathcal{H} ~:~ T ~\text{is bounded and linear} \}$. For every $T \in \mathcal{B} (\mathcal{H})$, the operator norm $\norm{T}$ of $T$ is defined by $\norm{T} := \sup_{\norm{x} \leq  1} \norm{Tx}$.

For other notation not explicitly defined here, we refer the reader to \cite{BC2017}.

%%%%%%%%%%%%%%%%%%%%%%%%%%%%%%%%%%%%%%%%%%%%%%%%%%%%%%%%%%%%%%%%%%%%%%%%%%%%%%%%%%%%%%%%%\section{Auxiliary results}%%%%%%%%%%%%%%%%%%%%%%%%%%%%%
%%%%%%%%%%%%%%%%%%%%%%%%%%%%%%%%%%%%%%%%%%%%%%%%%%%%%%%%%%%%%%%%%%%
\section{Preliminaries} \label{sec:Preliminaries}

\subsection*{Projections and Friedrichs angle}
\begin{fact} {\rm \cite[Proposition~3.19]{BC2017}} \label{fac:SetChangeProje}
	Let $C$ be a nonempty closed  convex subset of the Hilbert space $\mathcal{H}$ and let $x \in \mathcal{H}$. Set $D:=z+C$, where $z \in \mathcal{H}$. Then $\Pro_{D}x=z+\Pro_{C}(x-z)$.
\end{fact}

\begin{fact} {\rm \cite[Theorems~5.8]{D2012}} \label{MetrProSubs8}
	Let $M$ be a closed linear subspace of $\mathcal{H}$. Then $\Id =\Pro_{M}+\Pro_{M^{\perp}}$.
\end{fact}

Note that the case in which $M$ and $N$ are linear subspaces in the following result has already been shown in \cite[Lemma~9.2]{D2012}.
\begin{lemma} \label{lem:ExchangeProj}
	Let $M$ and $N$ be closed affine subspaces of $\mathcal{H}$ with $M \cap N \neq \varnothing$. Assume  $M \subseteq N$ or $N \subseteq M$. Then $\Pro_{M}\Pro_{N}=\Pro_{N}\Pro_{M} =\Pro_{M \cap N} $.
\end{lemma}	
\begin{proof}
	Let $z \in M \cap N$. By \cite[Theorem~1.2]{R1970}, the parallel linear subspaces of $M$ and $N$ are $\pa M =M -z$ and $\pa N = N-z $ respectively. By assumption, $M \subseteq N$ or $N \subseteq M$, we know, $\pa M \subseteq \pa N$ or $\pa N \subseteq \pa M$.
	
	Then by \cref{fac:SetChangeProje} and \cite[Lemma~9.2]{D2012}, for every $x \in \mathcal{H}$, $\Pro_{M}\Pro_{N}x =z+\Pro_{\pa M}( \Pro_{N}(x)-z) = z+\Pro_{\pa M}( z + \Pro_{\pa N}(x-z)-z) = z+\Pro_{\pa M} \Pro_{\pa N}(x-z) = z+\Pro_{\pa M \cap \pa N}(x-z) =\Pro_{M \cap N} x$,
	which implies that $\Pro_{M}\Pro_{N} = \Pro_{M \cap N}$.  The proof of $\Pro_{N}\Pro_{M} = \Pro_{M \cap N}$ is similar.
\end{proof}

\begin{definition} {\rm \cite[Definition~9.4]{D2012}} \label{defn:FredrichAngleClassical}
	The  \emph{Friedrichs angle} between two linear subspaces $U$ and $V$ is the angle $\alpha(U,V)$ between $0$ and $\frac{\pi}{2}$ whose cosine, $c(U,V) :=\cos \alpha(U,V)$, is defined by the expression
	\begin{align*}
	c(U,V)
	:=  \sup \{ |\innp{u,v}| ~:~ u \in U \cap (U \cap V)^{\perp}, v \in V \cap (U \cap V)^{\perp}, \norm{u} \leq 1, \norm{v} \leq 1 \}.
	\end{align*}
\end{definition}

\begin{fact} {\rm \cite[Theorem~9.35]{D2012}} \label{fac:cFLess1}
	Let $U$ and $V$ be closed linear subspaces of $\mathcal{H}$. Then the following statements are equivalent.
	\begin{enumerate}
		\item $c(U,V) <1$.
		\item $U +V$ is closed.
		\item $U^{\perp} +V^{\perp}$ is closed.
	\end{enumerate}
\end{fact}

\subsection*{Nonexpansive operators}

\begin{definition} \label{defn:Nonexpansive} {\rm \cite[Definition~4.1]{BC2017}}
	Let $D$ be a nonempty subset of $\mathcal{H}$ and let $T:D \rightarrow \mathcal{H}$. Then $T$ is	\begin{enumerate}
		\item \label{Nonex} \emph{nonexpansive} if it is Lipschitz continuous with constant 1, i.e., $(\forall x  \in D)$ 	$(\forall y \in D)$ $\norm{Tx -Ty} \leq \norm{x-y}$;
%		\begin{align*}
%		(\forall x  \in D) 	(\forall y \in D) \quad \norm{Tx -Ty} \leq \norm{x-y};
%		\end{align*}
		\item \label{QuasiNonex} \emph{quasinonexpansive} if $(\forall x \in D)$ $(\forall y \in \Fix T)$ $\norm{Tx-y} \leq \norm{x-y}$;
%		\begin{align*}
%		(\forall x \in D) (\forall y \in \Fix T) \quad \norm{Tx-y} \leq \norm{x-y};
%		\end{align*}
		\item \label{StrickQuasiNonex} and \emph{strictly quasinonexpansive} if $(\forall x \in D \smallsetminus \Fix T)$ $(\forall y \in \Fix T)$ $\norm{Tx-y} < \norm{x-y}$.
%		\begin{align*}
%		(\forall x \in D \smallsetminus \Fix T)  (\forall y \in \Fix T) \quad \norm{Tx-y} < \norm{x-y}.
%		\end{align*}
	\end{enumerate}
\end{definition}

\begin{definition} {\rm \cite[Definition~4.33]{BC2017}} \label{defn:AlphaAverage}
	Let $D$ be a nonempty subset of $\mathcal{H}$, let $T: D \rightarrow \mathcal{H}$ be nonexpansive, and let $\alpha \in \, \left]0,1\right[\,$. Then $T$ is \emph{averaged with constant $\alpha$}, or \emph{$\alpha$-averaged} for short, if there exists a nonexpansive operator $R: D \rightarrow \mathcal{H}$ such that $T=(1- \alpha) \Id +\alpha R$.
\end{definition}

\begin{lemma} \label{lemma:quasi:to:nonexpansive}
	Let $T: \mathcal{H} \to \mathcal{H}$ be affine operator with $\Fix T \neq \varnothing$. Then $T$ is quasinonexpansive if and only if $T$ is nonexpansive.
\end{lemma}

\begin{proof}
By \cref{defn:Nonexpansive}, $T$ is nonexpansive implies that $T$ is quasinonexpansive. Suppose that $T$ is quasinonexpansive.   Because $\Fix T \neq \varnothing$, take $z \in \Fix T$.	Define
\begin{align} \label{eq:lemma:quasi:to:nonexpansive:F}
(\forall x \in \mathcal{H}) \quad F(x) := T(x+z)-z.
\end{align}
Then by \cite[Lemma~3.8]{BOyW2019LinearConvergence}, $F$ is linear.  Because $T$ is quasinonexpansive,
\begin{align*}
(\forall x \in \mathcal{H}) \quad \norm{Fx} =\norm{T(x+z)-z } \leq \norm{(x+z)-z}=\norm{x},
\end{align*}
which, by the linearity of $F$, implies that
\begin{align} \label{eq:lemma:quasi:to:nonexpansive:ineq}
(\forall x \in \mathcal{H}) (\forall y \in \mathcal{H}) \quad \norm{Fx-Fy}   \leq \norm{x-y}.
\end{align}
Now, for every $x \in \mathcal{H}$ and for every $y \in \mathcal{H}$,
\begin{align*}
\norm{Tx -Ty} \stackrel{\cref{eq:lemma:quasi:to:nonexpansive:F}}{=} \norm{ z+ F(x-z) - ( z+ F(y-z))} =\norm{F(x-z) -F(y-z) } \stackrel{\cref{eq:lemma:quasi:to:nonexpansive:ineq}}{\leq} \norm{x-y},
\end{align*}
which means that $T$ is nonexpansive.
\end{proof}

%%%%%%%%%%%%%%%%%%%%%%%%%%%%%%%%%%%%%%%%%%%%%%%%%%%%%%%%%%%%%%%%%%%%%%%%%%%%\section{Best approximation mapping w.r.t. nonempty closed and convex sets}%%%%%
%%%%%%%%%%%%%%%%%%%%%%%%%%%%%%%%%%%%%%%%%%%%%%%%%%%%%%%%%%%%%%%%%%%
\section{Best approximation mapping} \label{sec:BAM}
 The best approximation mapping with respect to a closed affine subspaces in $\mathbb{R}^{n}$ was introduced by Behling, Bello-Cruz and Santos in \cite{BCS2019}. In this section, we  extend the definition of BAM from closed affine subspace to nonempty closed convex set, and from $\mathbb{R}^{n}$ to general Hilbert space.
 Moreover, we provide some examples and properties of the generalized version of BAM.

\subsection*{Definition of BAM}

\begin{definition} \label{def:BAM}
	Let $G: \mathcal{H} \to \mathcal{H}$, and let $\gamma \in \left[0,1\right[\,$.
	Then $G$ is a \emph{best approximation mapping with constant $\gamma$} (for short $\gamma$-BAM), if
	\begin{enumerate}
		\item \label{def:BAM:Fix}  $\Fix G$ is a nonempty closed  convex subset of $\mathcal{H}$,
		\item  \label{def:BAM:eq} $\Pro_{\Fix G}G=\Pro_{\Fix G}$, and
		\item  \label{def:BAM:Ineq}
		$(\forall x \in \mathcal{H})$ $\norm{Gx -\Pro_{\Fix G}x} \leq \gamma \norm{x - \Pro_{\Fix G}x}$.
	\end{enumerate}	
In particular, if $\gamma$ is unknown or not necessary to point out, we just say that $G$ is a BAM.
\end{definition}

The following \cref{lemma:CFixG}\cref{lemma:CFixG:Fix} illustrates that in \cite[Definition~2]{BCS2019}, the set $C$ is uniquely determined by the operator $G$, and that, moreover, $C=\Fix G$. Hence,   our \cref{def:BAM}  is indeed a natural generalization of \cite[Definition~2]{BCS2019}.
\begin{lemma} \label{lemma:CFixG}
	Let $G: \mathcal{H} \to \mathcal{H}$, let  $C$ be a nonempty closed  convex subset of $\mathcal{H}$, and let $\gamma \in \left[0,1\right[\,$.  Suppose that $\Pro_{C}G=\Pro_{C}$ and that
	\begin{align} \label{eq:lemma:CFixG}
	(\forall x \in \mathcal{H}) \quad \norm{Gx -\Pro_{C}x} \leq \gamma \norm{x - \Pro_{C}x}.
	\end{align}
	Then the following hold:
	\begin{enumerate}
		\item \label{lemma:CFixG:GC} $G\Pro_{C}=\Pro_{C}$.
		\item \label{lemma:CFixG:Fix} $\Fix G =C$.
		\item \label{lemma:CFixG:BAM} $G$ is a $\gamma$-BAM.
	\end{enumerate}
\end{lemma}

\begin{proof}
	\cref{lemma:CFixG:GC}: For every $y\in \mathcal{H}$, use the idempotent property of $\Pro_{C}$ and apply \cref{eq:lemma:CFixG}  with $x=\Pro_{C} y$ to obtain that
	\begin{align*}
	\norm{G\Pro_{C} y -\Pro_{C}y}=	\norm{G\Pro_{C} y -\Pro_{C}\Pro_{C} y} \leq \gamma \norm{\Pro_{C} y-\Pro_{C} \Pro_{C} y} =0,
	\end{align*}
	which implies that $(\forall y \in \mathcal{H})$ $G\Pro_{C} y =\Pro_{C}y$, that is, $G\Pro_{C}=\Pro_{C}$.

	\cref{lemma:CFixG:Fix}:	Let $x \in \mathcal{H}$.   On the one hand, by \cref{lemma:CFixG:GC}, $x \in C \Rightarrow x=\Pro_{C}x =G\Pro_{C}x =Gx \Rightarrow x \in \Fix G$.  On the other hand,   $x \in \Fix G \Rightarrow x=Gx \Rightarrow \norm{x -\Pro_{C}x}=\norm{Gx - \Pro_{C}x} \leq \gamma \norm{x -\Pro_{C}x} \Rightarrow x-\Pro_{C}x=0 \Rightarrow x \in C$,
	where the second and third implications are from \cref{eq:lemma:CFixG}, and $\gamma <1$ respectively. Altogether, \cref{lemma:CFixG:Fix} is true.
	
	\cref{lemma:CFixG:BAM}: This is directly from \cref{def:BAM}.
\end{proof}

\begin{proposition}  \label{prop:BAM:Fix}
	Let $\gamma \in \left[0,1\right[\,$. Suppose that $G$ is a $\gamma$-BAM.  Then $\dist_{\Fix G} \circ G \leq \gamma \dist_{\Fix G} $.
\end{proposition}	

\begin{proof}
	Let $x \in \mathcal{H}$.
	By \cref{def:BAM}\cref{def:BAM:Fix}, $\Fix G $ is a nonempty closed convex set, so $\dist_{\Fix G}$ is well defined. Moreover, by \cref{def:BAM}\cref{def:BAM:eq}$\&$\cref{def:BAM:Ineq},
	\begin{align*}
	\dist_{\Fix G} (Gx)
	= \norm{Gx - \Pro_{\Fix G}Gx}
	= \norm{Gx - \Pro_{\Fix G}x}
	\leq \gamma \norm{x - \Pro_{\Fix G}x}
	=  \gamma \dist_{\Fix G} x.
	\end{align*}
\end{proof}

\begin{example} \label{examp:BAM:Pro}
	Let $C$ be a nonempty closed  convex subset  of $\mathcal{H}$. Then for every $\gamma \in \left[0,1\right[\,$, $  (1-\gamma) \Pro_{C}+\gamma  \Id$ is a $\gamma$-BAM with $\Fix G=C$. Moreover, $\Id$ is a $0$-BAM with $\Fix \Id =\mathcal{H}$.
\end{example}

\begin{proof}
	Let $\gamma \in \left[0,1\right[\,$. Then  by \cite[Proposition~3.21]{BC2017}, $\Pro_{C} \left(   (1-\gamma) \Pro_{C}+\gamma  \Id \right) = \Pro_{C}$.
	In addition, $\norm{ (1-\gamma) \Pro_{C}x +\gamma  x - \Pro_{C}x } =\gamma \norm{x - \Pro_{C}x }$. The last assertion is clear from definitions.
\end{proof}	

\begin{remark} \label{remark:chara:Id:P}
		Let $C$ be a nonempty closed  convex subset  of $\mathcal{H}$ and let  $\gamma \in \mathbb{R}$.
	\begin{enumerate}
		\item \label{remark:chara:Id:P:i} Because $(\forall x \in \mathcal{H})$ $\norm{   (1-\gamma) \Pro_{C}x +\gamma  x- \Pro_{C}x } =|\gamma| \norm{x - \Pro_{C}x }$, and $|\gamma|< 1 \Leftrightarrow \gamma \in \left]-1, 1\right[\,$, by \cref{def:BAM}\cref{def:BAM:Ineq},  we know that $(1-\gamma) \Pro_{C}+\gamma  \Id$ is a BAM implies that $\gamma \in \left]-1, 1\right[\,$.
		\item \label{remark:chara:Id:P:ii} Let $\epsilon \in \mathbb{R}_{++}$. Suppose that $\mathcal{H}=\mathbb{R}^{2}$,   $C:=\mathbf{B}[0;1]$ and $\gamma :=-\epsilon$. Let $x:=(1+\epsilon, 0)$. Then
		\begin{align*}
		\Pro_{C}( (1-\gamma) \Pro_{C}+\gamma  \Id )x =\begin{cases}
		(1-\epsilon^{2},0) \quad & \text{if } \epsilon \leq \sqrt{2},\\
		(-1,0) \quad & \text{if } \epsilon >\sqrt{2},
		\end{cases}
		\end{align*}
	which implies that  $\Pro_{C}( (1-\gamma) \Pro_{C}+\gamma  \Id )x \neq (1,0) =\Pro_{C} x$, which yields that $(1-\gamma) \Pro_{C}+\gamma  \Id$ is not a BAM.
	\end{enumerate}
Hence, using the two items above, we conclude that if $(1-\gamma) \Pro_{C}+\gamma  \Id$ is a BAM, then $ \gamma  \in \left]-1, 1\right[\,$ and that generally if $\gamma \in \left]-1, 0\right]$, then $(1-\gamma) \Pro_{C}+\gamma  \Id$ is not a BAM. Therefore, the assumption in \cref{examp:BAM:Pro} is tight.
\end{remark}

\begin{example}\label{exam:LinearConver:BAM}
	Suppose that $\mathcal{H} = \mathbb{R}^{n}$.	Let $T: \mathcal{H} \to \mathcal{H}$ be $\alpha$-averaged with $\alpha \in \,\left]0,1\right[$ and let $T$ be  linear.  Then $\norm{T\Pro_{(\Fix T)^{\perp}}} \in \left[0,1\right[\,$ and $T$ is a $\norm{T\Pro_{(\Fix T)^{\perp}}}$-BAM.
\end{example}	

\begin{proof}
	The  items \cref{def:BAM:Fix}, \cref{def:BAM:eq} and \cref{def:BAM:Ineq} in \cref{def:BAM} follow from \cite[Lemmas~3.12 and 3.14]{BDHP2003}  and  \cite[Proposition~2.22]{BOyW2019LinearConvergence} respectively.
\end{proof}	

It is easy to see that $-\Id$  is linear and nonexpansive but   not a BAM. Hence, the condition \enquote{$T$ is $\alpha$-averaged} in \cref{exam:LinearConver:BAM} can not be replaced by \enquote{$T$ is nonexpansive}.

\begin{proposition} \label{prop:contraction:BAM}
	Let $T : \mathcal{H} \to \mathcal{H}$ be a Banach contraction on $\mathcal{H}$, say, there exists $\gamma \in \left[0,1\right[$ such that
	\begin{align} \label{eq:exam:contraction:BAM}
	(\forall x \in \mathcal{H}) (\forall y \in \mathcal{H}) \quad \norm{Tx-Ty} \leq \gamma \norm{x -y}.
	\end{align}
	 Then $T$ is a $\gamma$-BAM.
\end{proposition}

\begin{proof}
	By \cite[Theorem~1.50(i)]{BC2017}, $\Fix T$ is a singleton, say  $\Fix T =\{z\}$ for some $z \in \mathcal{H}$.
Let $x \in \mathcal{H}$. Then $\Pro_{\Fix T} Tx=z=\Pro_{\Fix T} x$, which implies that $\Pro_{\Fix T} T=\Pro_{\Fix T} $.
	 Moreover, $\norm{Tx - \Pro_{\Fix T}x}=\norm{Tx - z} =\norm{Tx - Tz} \stackrel{\cref{eq:exam:contraction:BAM}}{\leq}  \gamma \norm{x -z}= \gamma \norm{x -\Pro_{\Fix T}x}$.
	Altogether,  $T$ is a $\gamma$-BAM.
\end{proof}

\begin{remark}
	\begin{enumerate}
		\item \cref{prop:contraction:BAM} illustrates that every Banach contraction is a BAM.
		\item Note that a  contraction must be continuous. By \cref{exam:counterexam} below, a BAM (even with fixed point set being singleton) is generally not continuous. Hence, we know that a BAM is generally not a  contraction and that the converse of \cref{prop:contraction:BAM} fails.
	\end{enumerate}
\end{remark}

\begin{proposition} \label{prop:ABAM}
	Let $A \in \mathbb{R}^{n\times n}$ be a normal matrix. Denote by $\rho (A)$ the spectral radius of $A$, i.e., 
$$\rho (A):=\max \{ |\lambda| ~:~ \lambda \text{ is an eigenvalue of } A  \}.$$
	\begin{enumerate}
		\item \label{prop:ABAM:rho} Suppose one of the following holds:
		\begin{enumerate}
			\item \label{prop:ABAM:rho:<} $\rho (A) <1$.
			\item \label{prop:ABAM:rho:=} $\rho (A) =1$, where $\lambda=1$ is the only eigenvalue on the unit circle
and semisimple.
		\end{enumerate}
		Then $A$ is a BAM.
		\item \label{prop:ABAM:equa} The following are equivalent:
		\begin{enumerate}
		\item \label{prop:ABAM:equa:a} $\lim_{k \to \infty} A^{k}$ exists.
		\item \label{prop:ABAM:equa:b} $\lim_{k \to \infty} A^{k} =\Pro_{\Fix A}$.
		\item \label{prop:ABAM:equa:c} $A$ is a BAM.
		\end{enumerate}
\end{enumerate}
\end{proposition}

\begin{proof}
	\cref{prop:ABAM:rho}:
	If $\rho (A) <1$, then by \cite[Example~2.19]{BC2017}, $A$ is a Banach contraction. Hence, by \cref{prop:contraction:BAM}, $A$ is a BAM.
	
	Suppose that $\rho (A) =1$ and $\lambda=1$ is the only eigenvalue of $A$ on the unit circle and semisimple. 
Then by the Spectral Theorem for Diagonalizable Matrices \cite[page~517]{Meyer2000} and Properties of Normal Matrices  \cite[page~548]{Meyer2000},
	\begin{align*}
	A=\Pro_{U_{1}} +\lambda_{2}\Pro_{U_{2}}  +\cdots +\lambda_{k}\Pro_{U_{k}},
	\end{align*}
	where $\sigma(A) =\{\lambda_{1}, \lambda_{2}, \ldots, \lambda_{k}\}$ with $\lambda_{1}=1$ is the spectrum of $A$ and $(\forall i \in \{1, \ldots, k\})$ $U_{i} :=\ker (A-\lambda_{i}\Id)$.
	Then clearly $\Fix A=\ker(A-\Id) =U_{1}$. Moreover, by the Spectral Theorem for Diagonalizable Matrices \cite[page~517]{Meyer2000} again, it is easy to see that
	\begin{align*}
	&\Pro_{\Fix A}A=\Pro_{\Fix A},\\
	(\forall x \in \mathbb{R}^{n}) \quad &\norm{ Ax- \Pro_{\Fix A}x} \leq |\lambda_{2}| \norm{x -\Pro_{\Fix A}x},
	\end{align*}
	where $|\lambda_{2}| <1$. Therefore, $A$ is a BAM.
	
	\cref{prop:ABAM:equa}: By the Theorem of Limits of Powers \cite[Page~630]{Meyer2000}, 
$\lim_{k \to \infty} A^{k}$ exists if and only if $\rho (A) <1$ or $\rho (A) =1$ with $\lambda=1$ being the only eigenvalue of $A$ on the unit circle and semisimple, which implies that $\lim_{k \to \infty} A^{k}=\Pro_{\Fix A}$. Moreover, by
\cref{def:BAM}, $A$ being a BAM implies that  $\lim_{k \to \infty} A^{k} =\Pro_{\Fix A}$. Combine these results with \cref{prop:ABAM:rho} to obtain \cref{prop:ABAM:equa}.
\end{proof}

\subsection*{Properties of BAM}

The following \cref{prop:BAM:Properties}\cref{prop:BAM:Properties:compo:Ineq} states that any sequence of iterates of a BAM must linearly converge to the best approximation onto the fixed point set of the BAM. Therefore, we see the importance of the study of BAMs.
The following \cref{prop:BAM:Properties} reduces to \cite[Proposition~1]{BCS2019} when $\mathcal{H} =\mathbb{R}^{n}$ and $\Fix G$ is an affine subspace of $\mathbb{R}^{n}$. In fact, there is little difficulty  to extend the space from $\mathbb{R}^{n}$ to $\mathcal{H}$ and the related set from closed affine subspace to nonempty closed convex set.

\begin{proposition}  \label{prop:BAM:Properties}
Let $\gamma \in \left[0,1\right[\,$ and let $G: \mathcal{H} \to \mathcal{H}$. Suppose that $G$ is a $\gamma$-BAM. Then for every $k \in \mathbb{N}$,
	\begin{enumerate}
		\item \label{prop:BAM:Properties:compo:ProG} $\Pro_{\Fix G} G^{k}=\Pro_{\Fix G}$, and
		\item \label{prop:BAM:Properties:compo:Ineq} $(\forall x \in \mathcal{H})$  $\norm{G^{k}x -\Pro_{\Fix G}x} \leq \gamma^{k} \norm{x- \Pro_{\Fix G}x}$.
	\end{enumerate}	
	Consequently, for every $x \in \mathcal{H}$, $(G^{k}x)_{k\in \mathbb{N}}$ converges to $\Pro_{\Fix G}x$ with a linear rate $\gamma$.
\end{proposition}	
\begin{proof}
	Because  $G$ is a $\gamma$-BAM, by \cref{def:BAM}, we have that $\Fix G$ is a nonempty closed and convex subset of $\mathcal{H}$, and that
	\begin{subequations}
		\begin{align}
		&\Pro_{\Fix G} G=\Pro_{\Fix G}, \label{prop:BAM:Properties:PCG}\\
			(\forall y \in \mathcal{H}) \quad & \norm{Gy -\Pro_{\Fix G}y} \leq \gamma \norm{y- \Pro_{\Fix G}y}. \label{prop:BAM:Properties:GPCineq}
		\end{align}
	\end{subequations}
	We argue by induction on $k$. It is trivial that \cref{prop:BAM:Properties:compo:ProG} and \cref{prop:BAM:Properties:compo:Ineq} hold for $k=0$. Assume \cref{prop:BAM:Properties:compo:ProG} and \cref{prop:BAM:Properties:compo:Ineq}  are true for some $k \in \mathbb{N}$, that is,
	\begin{subequations}
			\begin{align}
		&\Pro_{\Fix G} G^{k}=\Pro_{\Fix G}, \label{prop:BAM:Properties:PCGm}\\
		(\forall y \in \mathcal{H}) \quad &\norm{G^{k}y -\Pro_{\Fix G}y} \leq \gamma^{k} \norm{y- \Pro_{\Fix G}y}.\label{prop:BAM:Properties:GPCineq:k}
		\end{align}
	\end{subequations}
	Let $x \in \mathcal{H}$. Now
	\begin{align*}
	\Pro_{\Fix G} G^{k+1}x=\Pro_{\Fix G} G (G^{k} x) \stackrel{\cref{prop:BAM:Properties:PCG}}{=}\Pro_{\Fix G}(G^{k} x) \stackrel{\cref{prop:BAM:Properties:PCGm}}{=}\Pro_{\Fix G}.
	\end{align*}
	Moreover,
$
	\norm{G^{k+1}x -\Pro_{\Fix G}x} \stackrel{\cref{prop:BAM:Properties:PCGm}}{=} \norm{G(G^{k}x) -\Pro_{\Fix G}(G^{k}x)} \stackrel{\cref{prop:BAM:Properties:GPCineq}}{\leq} \gamma \norm{G^{k}x- \Pro_{\Fix G}(G^{k}x)}\stackrel{\cref{prop:BAM:Properties:PCGm}}{=}\gamma \norm{G^{k}x- \Pro_{\Fix G}x} \stackrel{\cref{prop:BAM:Properties:GPCineq:k}}{\leq} \gamma^{k+1} \norm{x- \Pro_{\Fix G}x}.
	$
	
	Hence, the proof is complete by the principle of mathematical induction.
\end{proof}

\begin{proposition} \label{prop:T:QNEP}
	Let $T : \mathcal{H} \to \mathcal{H}$ be quasinonexpansive with $\Fix T  $ being a closed affine subspace of $\mathcal{H}$. Let $\gamma \in \left[0, 1\right[\,$. Suppose that $(\forall x \in \mathcal{H}) $ $\norm{Tx -\Pro_{\Fix T}x} \leq \gamma \norm{x - \Pro_{\Fix T}x} $.
	Then $T$ is a $\gamma$-BAM.
\end{proposition}

\begin{proof}
	By   assumptions and \cref{def:BAM}, it remains to prove $\Pro_{\Fix T}T=\Pro_{\Fix T} $.

	Let $x \in \mathcal{H}$. By \cite[Example~5.3]{BC2017}, $T$ is quasinonexpansive and $\Fix T \neq \varnothing$ imply that $(T^{k}x)_{k \in \mathbb{N}}$ is Fej\'er monotone with respect to $\Fix T$. This, the assumption  that $\Fix T$ is  a closed affine subspace, and  \cite[Proposition~5.9(i)]{BC2017} imply that
	\begin{align*}
	(\forall k \in \mathbb{N})  \quad  \Pro_{\Fix T}T^{k}x= \Pro_{\Fix T}x,
	\end{align*}
	which yields $\Pro_{\Fix T}T=\Pro_{\Fix T} $ when $k=1$.
\end{proof}

The following result shows further connection between BAMs and linear convergent mappings.
\begin{corollary}
	Let $T : \mathcal{H} \to \mathcal{H}$ be quasinonexpansive with $\Fix T  $ being a closed affine subspace of $\mathcal{H}$. Let $\gamma \in \left[0, 1\right[\,$. Then $T$ is a $\gamma$-BAM if and only if $ (\forall k \in \mathbb{N})$ $ (\forall x \in \mathcal{H})$ $ \norm{T^{k}x -\Pro_{\Fix T}x} \leq \gamma^{k} \norm{x - \Pro_{\Fix T}x}$.
\end{corollary}

\begin{proof}
	\enquote{$\Rightarrow$}: This is  clearly from \cref{prop:BAM:Properties}.
	
	\enquote{$\Leftarrow$}: This comes from the assumptions and \cref{prop:T:QNEP}.
\end{proof}

The following result states that BAM with closed affine fixed point set is strictly quasinonexpansive. In particular, the inequality shown in \cref{prop:BAM:StrictlyQuasinonexpansive}\cref{prop:BAM:StrictlyQuasinonexpansive:INEQ} is interesting on its own.
\begin{proposition} \label{prop:BAM:StrictlyQuasinonexpansive}
	Let $G: \mathcal{H} \to \mathcal{H}$ with $\Fix G$ being a  closed affine subspace of $\mathcal{H}$. Let $\gamma \in \left[0,1\right[\,$.  Suppose that $G$ is a $\gamma $-BAM.  The the following hold:
\begin{enumerate}
	\item \label{prop:BAM:StrictlyQuasinonexpansive:INEQ}	$(\forall x \in \mathcal{H}) $ $(\forall y \in \Fix G)$ $\norm{Gx -y}^{2} + ( 1-\gamma^{2}) \norm{x - \Pro_{\Fix G}(x) }^{2}   \leq \norm{x -y}^{2}$.
	\item \label{prop:BAM:StrictlyQuasinonexpansive:STRICQUA} $G$ is strictly quasinonexpansive.
\end{enumerate}
\end{proposition}

\begin{proof}
	\cref{prop:BAM:StrictlyQuasinonexpansive:INEQ}: Because $G$ is a $\gamma $-BAM, by \cref{def:BAM},
	\begin{subequations}
			\begin{align}
		&\Pro_{\Fix G}G =\Pro_{\Fix G},\label{prop:BAM:StrictlyQuasinonexpansive:eq}\\
		(\forall x \in \mathcal{H}) \quad & \norm{Gx -\Pro_{\Fix G}x} \leq \gamma \norm{x - \Pro_{\Fix G}x}.\label{prop:BAM:StrictlyQuasinonexpansive:ineq}
		\end{align}
	\end{subequations}
Because $\Fix G$ is a closed affine subspace of $\mathcal{H}$,	by \cite[Proposition~2.10]{BOyW2018Proper}, for every $x \in \mathcal{H}$ and $y \in \Fix G$,
\begin{subequations} \label{prop:BAM:StrictlyQuasinonexpansive:Gx-y}
	\begin{align}
	\norm{Gx -y}^{2} &~~=~ \norm{Gx - \Pro_{\Fix G}(Gx) }^{2} +\norm{\Pro_{\Fix G}(Gx) -y}^{2}\\
	& \stackrel{\cref{prop:BAM:StrictlyQuasinonexpansive:eq}}{=} \norm{Gx - \Pro_{\Fix G}(x) }^{2} +\norm{\Pro_{\Fix G}(x) -y}^{2}\\
	& \stackrel{\cref{prop:BAM:StrictlyQuasinonexpansive:ineq}}{\leq } \gamma^{2} \norm{x - \Pro_{\Fix G}(x) }^{2} +\norm{\Pro_{\Fix G}(x) -y}^{2}
	\end{align}
\end{subequations}
	and, by \cite[Proposition~2.10]{BOyW2018Proper} again,
	\begin{align} \label{prop:BAM:StrictlyQuasinonexpansive:x-y}
	\norm{x -y}^{2} = \norm{x - \Pro_{\Fix G}(x) }^{2} +\norm{\Pro_{\Fix G}(x) -y}^{2}.
	\end{align}
	Combine \cref{prop:BAM:StrictlyQuasinonexpansive:Gx-y} with \cref{prop:BAM:StrictlyQuasinonexpansive:x-y} to see that
	\begin{align} \label{eq:prop:BAM:StrictlyQuasinonexpansive:Gxy}
	\norm{Gx -y}^{2} - \norm{x -y}^{2}  \leq (\gamma^{2} -1) \norm{x - \Pro_{\Fix G}(x) }^{2},
	\end{align}
	which yields \cref{prop:BAM:StrictlyQuasinonexpansive:INEQ}.
	
	\cref{prop:BAM:StrictlyQuasinonexpansive:STRICQUA}:
	Because   $(\forall x \in \mathcal{H} \smallsetminus  \Fix G )$, $\norm{x - \Pro_{\Fix G}x } >0$ and $\gamma \in \left[0,1\right[\,$, by \cref{eq:prop:BAM:StrictlyQuasinonexpansive:Gxy},
	\begin{align*}
	( \forall x \in \mathcal{H} \smallsetminus  \Fix G ) ( \forall y \in \Fix G) \quad \norm{Gx -y}< \norm{x -y}.
	\end{align*}
	Hence, by \cref{defn:Nonexpansive}\cref{StrickQuasiNonex}, we obtain that $G$ is strictly quasinonexpansive.
\end{proof}

\begin{corollary} \label{cor:affine:BAM}
	Let $G: \mathcal{H} \to \mathcal{H}$ be an affine  BAM.  Then $G$ is  nonexpansive.
\end{corollary}

\begin{proof}
By \cref{def:BAM}\cref{def:BAM:Fix},  $G$ is a BAM yields that $\Fix G$ is a nonempty closed and convex subset of $\mathcal{H}$. Moreover, because $G$ is affine,
\begin{align*}
(\forall x \in \Fix G ) (\forall y \in \Fix G ) (\forall \alpha \in \mathbb{R}) \quad G ( \alpha x + (1-\alpha) y) =\alpha G(x) + (1-\alpha) G( y)= \alpha x + (1-\alpha) y,
\end{align*}
which implies that  $\Fix G$ is an affine subspace.
Hence, by	\cref{prop:BAM:StrictlyQuasinonexpansive}\cref{prop:BAM:StrictlyQuasinonexpansive:STRICQUA}, $G$ is strictly quasinonexpansive.  Therefore, by \cref{lemma:quasi:to:nonexpansive}, $G$ is  nonexpansive.
\end{proof}

Let $T :\mathcal{H} \to \mathcal{H}$ with $\Fix T \neq \varnothing$ and let $\kappa \in \mathbb{R}_{+}$. We say $T$ is \emph{linear regular} with constant $\kappa$ if
\begin{align*}
(\forall x \in \mathcal{H}) \quad \dist_{\Fix T} (x) \leq \kappa \norm{x -Tx}.
\end{align*}
By the following two results, we know that every BAM is linearly regular, but generally linearly regular operator is not a BAM.
\begin{proposition} \label{prop:BAM:linearRegular}
	Let $G: \mathcal{H} \to \mathcal{H}$ and let $\gamma \in \left[0,1\right[\,$. Suppose that $G$ is a $\gamma$-BAM. Then $G$ is linearly regular with constant $\frac{1}{1 -\gamma}$.
\end{proposition}

\begin{proof}
	Because $G$ is a $\gamma$-BAM, by \cref{def:BAM}, $\Fix G$ is a  nonempty closed convex subset of $\mathcal{H}$ and
	\begin{align} \label{eq:prop:BAM:linearRegular}
	(\forall x \in \mathcal{H}) \quad \norm{Gx -\Pro_{\Fix G}x} \leq \gamma \norm{x - \Pro_{\Fix G}x}.
	\end{align}
	Let $x \in \mathcal{H}$.  By the triangle inequality and \cref{eq:prop:BAM:linearRegular},
	\begin{align*}
	&\norm{x - \Pro_{\Fix G}x }  \leq \norm{x - Gx} +\norm{Gx- \Pro_{\Fix G}x }
	\leq \norm{x - Gx} +\gamma \norm{x- \Pro_{\Fix G}x },\\
	\Rightarrow~	&	(1-\gamma ) \norm{x - \Pro_{\Fix G}x }  \leq \norm{x - Gx} \\
	\Leftrightarrow~ &  \norm{x - \Pro_{\Fix G}x }  \leq  \frac{1}{1 -\gamma} \norm{x - Gx}.
	\end{align*}
	Hence, $(\forall x \in \mathcal{H})$ $\dist_{\Fix T} (x)   \leq \frac{1}{1 -\gamma} \norm{x -Gx}$, that is, $G$ is linearly regular with constant $\frac{1}{1 -\gamma}$.
\end{proof}

\begin{example} \label{exam:notlinearregular}
	Suppose that $\mathcal{H} =\mathbb{R}^{2}$. Let $C=\mathbf{B}[0;1]$ and $G=\R_{C}$.  Let $x=(2,0)$. $\Pro_{C}\R_{C}x=(0,0) \neq (1,0) =\Pro_{C}x$, which, by \cref{def:BAM}, yields that $\R_{C}$ is not a BAM. On the other hand, apply \cite[Example~2.2]{BNP2015} with $\lambda =2$ to obtain that $\R_{C}= (1-2) \Id +2 \Pro_{C}$ is linearly regular with constant $\frac{1}{2}$.
\end{example}

\begin{proposition} \label{prop:CompositionBAM}
	Let $\I:=\{1, \ldots, m\}$. Let $(\forall i \in \I)$ $G_{i}: \mathcal{H} \to \mathcal{H}$ be operators with $\Fix G_{i} $ being a  closed affine subspace of $\mathcal{H}$  and  $(\forall i \in \I)$ $\gamma_{i} \in \left[0,1\right[\,$. Suppose that $(\forall i \in \I)$ $G_{i}$ is a $\gamma_{i}$-BAM and that $\cap_{j \in \I} \Fix G_{j} \neq \varnothing$.
	The  following hold:
	\begin{enumerate}
		\item \label{prop:CompositionBAM:SQN} $G_{m} \cdots G_{1} $ is strictly quasinonexpansive.
		\item \label{prop:CompositionBAM:comp}  $\Fix G_{m} \cdots G_{1} = \Fix \cap_{i \in \I} \Fix G_{i}$.
		\item \label{prop:CompositionBAM:comb} Let $(\omega_{i})_{i \in \I}$ be real numbers in $\left]0,1\right] $ such that $\sum_{i \in \I} \omega_{i}=1$. Then $\Fix \sum_{i \in \I} \omega_{i} G_{i} = \Fix \cap_{i \in \I} \Fix G_{i}$.
	\end{enumerate}
\end{proposition}

\begin{proof}
 Because $(\forall i \in \I)$ $G_{i}$ is a $\gamma_{i}$-BAM  with $\Fix G_{i} $ being a  closed affine subspace of $\mathcal{H}$,  by \cref{prop:BAM:StrictlyQuasinonexpansive},  $(\forall i \in \I)$  $G_{i}$ is strictly quasinonexpansive. Moreover, by assumption, $\cap_{i \in \I} \Fix G_{i} \neq \varnothing$.

\cref{prop:CompositionBAM:SQN}$\&$\cref{prop:CompositionBAM:comp}:  These are from  \cite[Corollary~4.50]{BC2017}.

 \cref{prop:CompositionBAM:comb}: This comes from \cite[Proposition~4.47]{BC2017}.
\end{proof}		

\begin{proposition} \label{prop:BAMG:0}
	Let $G: \mathcal{H} \to \mathcal{H}$  with $\Fix G$ being a  nonempty closed convex subset of $\mathcal{H}$. Then $G$ is a  $0$-BAM if and only if $G=\Pro_{\Fix G}$.
\end{proposition}

\begin{proof}
	\enquote{$\Rightarrow$}: Assume that $G$ is a $0$-BAM.
	By
	\cref{def:BAM},
	$(\forall x \in \mathcal{H})$ $\norm{Gx -\Pro_{\Fix G}x} \leq 0 \norm{x - \Pro_{\Fix G}x}=0$. Hence, $G=\Pro_{\Fix G}$.
	
	\enquote{$\Leftarrow$}: Assume that  $G=\Pro_{\Fix G}$. Then by \cref{examp:BAM:Pro}, $G$ is a BAM with constant $0$.
\end{proof}

\begin{corollary} \label{cor:G2G1:Constant0}
	Let $(\forall i \in \{1,2\})$ $G_{i}: \mathcal{H} \to \mathcal{H}$ be such that $\Fix G_{i} $ is a  closed affine subspace of $\mathcal{H}$. Suppose that  $(\forall i \in \{1,2\})$ $G_{i}$ is a BAM and  that  $\Fix G_{1} \cap \Fix G_{2} \neq \varnothing $. Then $G_{2}G_{1}$ is a $0$-BAM if and only if $G_{2}G_{1}=\Pro_{\Fix G_{1} \cap \Fix G_{2}}$.
\end{corollary}

\begin{proof}
	Because $\Fix G_{1}$ and $\Fix G_{2}$ are  closed affine subspaces and $\Fix G_{1} \cap \Fix G_{2} \neq \varnothing $, $\Fix G_{1} \cap \Fix G_{2} $ is a closed affine subspace.
	
	\enquote{$\Rightarrow$}:  By \cref{prop:CompositionBAM}\cref{prop:CompositionBAM:comp}, $\Fix G_{2}G_{1} = \Fix G_{1} \cap \Fix G_{2}$ is a closed affine subspace. Hence, by \cref{prop:BAMG:0}, $G_{2}G_{1}=\Pro_{\Fix G_{2}G_{1} }=\Pro_{\Fix G_{1} \cap \Fix G_{2}}$.
	
	\enquote{$\Leftarrow$}:  By \cref{examp:BAM:Pro}, $G_{2}G_{1}=\Pro_{\Fix G_{1} \cap \Fix G_{2}}$ is a $0$-BAM.
\end{proof}
According to the following \cref{exam:T2T1} and \cref{examp:CCS1CCS2} below, we know that   the composition of BAMs is a projector is not sufficient to deduce that  the individual BAMs are projectors. Hence, the condition  \enquote{$G_{i}$ is a BAM} in the \cref{cor:G2G1:Constant0} above is more general than \enquote{$G_{i}$ is a projector}.
\begin{example} \label{exam:T2T1}
	Let $U_{1}:=\mathbb{R} (1,0)$ and $U_{2}:=\mathbb{R} (0,1) $. Set $T_{1}:=\frac{1}{2} \Pro_{U_{1}}$ and  $T_{2}:=\frac{1}{2} \Pro_{U_{2}}$. Then neither $T_{1}$ nor $T_{2}$ is a projection. Moreover, $T_{2}T_{1} =\Pro_{\{(0,0) \}}$.
\end{example}

\begin{corollary} \label{cor:BAMProPro}
	Let $C_{1}$ and $C_{2}$ be  closed   convex subsets of $\mathcal{H}$ with $C_{1} \cap C_{2} \neq \varnothing$. Then  $\Pro_{C_{2}}\Pro_{C_{1}}$ is a $0$-BAM   if and only if $\Pro_{C_{2}}\Pro_{C_{1}} = \Pro_{C_{1} \cap C_{2}} $.
\end{corollary}

\begin{proof}
	Because $C_{1} \cap C_{2} \neq \varnothing$, by \cite[Corollary~4.5.2]{Cegielski}, $\Fix \Pro_{C_{2}}\Pro_{C_{1}} = C_{1} \cap C_{2} $ is nonempty, closed, and convex. Therefore,
the desired result follows  from  \cref{prop:BAMG:0}.
\end{proof}

\begin{proposition} \label{prop:BAMAffinLinea}
	Let $z \in \mathcal{H}$.	Let $(\forall i \in \{1,2\})$ $G_{i} :\mathcal{H} \rightarrow \mathcal{H}$  satisfy
	\begin{align} \label{eq:prop:BAMAffinLinea}
	(\forall x \in \mathcal{H})  \quad G_{1}x =z+ G_{2} (x -z).
	\end{align}
	Then the following assertions hold:
	\begin{enumerate}
		\item \label{prop:BAMAffinLinea:Fix} $\Fix G_{2} = \Fix G_{1} -z$.
		\item \label{prop:BAMAffinLinea:BAM}  Suppose that $\Fix G_{1}$ or $\Fix G_{2}$ is a  nonempty closed convex subset of $\mathcal{H}$. Let $\gamma \in \left[0,1\right[\,$.  Then  $G_{1}$ is a  $\gamma$-BAM  if and only if $G_{2}$ is a $\gamma$-BAM.
	\end{enumerate}
\end{proposition}	
\begin{proof}
	\cref{prop:BAMAffinLinea:Fix}: Let $x \in \mathcal{H}$. Then,
	\begin{align*}
	x \in \Fix G_{2} \Leftrightarrow x =G_{2}x =G_{1}(x+z) -z \Leftrightarrow  x +z =G_{1}(x+z)  \Leftrightarrow x+z \in \Fix G_{1} \Leftrightarrow x \in \Fix G_{1} -z.
	\end{align*}
	\cref{prop:BAMAffinLinea:BAM}:
	Clearly, by \cref{prop:BAMAffinLinea:Fix}, $\Fix G_{1} $ is a nonempty closed convex subset of $\mathcal{H}$  if and only if  $\Fix G_{2} $ is a nonempty closed  convex subset of $\mathcal{H}$.
	
Note that
	\begin{align*}
	\Pro_{\Fix G_{1}}G_{1}=\Pro_{\Fix G_{1}} & ~\Leftrightarrow~ (\forall x \in \mathcal{H})  ~\Pro_{\Fix G_{1}}G_{1}x= \Pro_{\Fix G_{1}}x\\
	& ~\Leftrightarrow~ (\forall x \in \mathcal{H})  ~\Pro_{z+\Fix G_{2}}G_{1}x= \Pro_{z+\Fix G_{2}}x \quad (\text{by \cref{prop:BAMAffinLinea:Fix}})\\
	&~\Leftrightarrow~ (\forall x \in \mathcal{H}) ~ z + \Pro_{\Fix G_{2}}( G_{1}x -z)= z + \Pro_{\Fix G_{2}} (x -z) \quad (\text{by \cref{fac:SetChangeProje}})\\
	&\stackrel{\cref{eq:prop:BAMAffinLinea} }{\Leftrightarrow} (\forall x \in \mathcal{H})   \Pro_{\Fix G_{2}}( G_{2} (x -z) )=  \Pro_{\Fix G_{2}} (x -z)\\
	&~\Leftrightarrow~ (\forall x \in \mathcal{H}) ~ \Pro_{\Fix G_{2}} (G_{2}x) = \Pro_{\Fix G_{2}}x\\
	&~\Leftrightarrow~ \Pro_{\Fix G_{2}} G_{2} = \Pro_{\Fix G_{2}},
	\end{align*}	
	and that
	\begin{align*}
	&	(\forall x \in \mathcal{H}) ~ \norm{G_{1}x -\Pro_{\Fix G_{1}}x} \leq \gamma \norm{x - \Pro_{\Fix G_{1}}x} \\
	 \Leftrightarrow~ &	(\forall x \in \mathcal{H}) ~ \norm{G_{1}x -\Pro_{z+\Fix G_{2}}x} \leq \gamma \norm{x - \Pro_{z+\Fix G_{2}}x} \quad (\text{by \cref{prop:BAMAffinLinea:Fix}})\\
	\Leftrightarrow~ &	(\forall x \in \mathcal{H}) ~ \norm{G_{1}x -\left(z+\Pro_{\Fix G_{2}}(x-z) \right)} \leq \gamma \norm{x - \left(z+\Pro_{\Fix G_{2}}(x-z) \right)} \quad (\text{by \cref{fac:SetChangeProje}})\\
	\stackrel{\cref{eq:prop:BAMAffinLinea} }{\Leftrightarrow}& (\forall x \in \mathcal{H}) ~
	\norm{G_{2}(x-z) -\Pro_{\Fix G_{2}}(x-z) } \leq \gamma \norm{(x-z) - \Pro_{\Fix G_{2}}(x-z) } \\
 \Leftrightarrow~ &	(\forall x \in \mathcal{H})  ~
	\norm{G_{2}x -\Pro_{\Fix G_{2}}x } \leq \gamma \norm{x - \Pro_{\Fix G_{2}}x }.
	\end{align*}
	Altogether, by \cref{def:BAM}, \cref{prop:BAMAffinLinea:BAM} above is true.
\end{proof}	

\begin{lemma} \label{lemma:shift}
Set $\I:=\{1, \ldots, m\}$.	Let $(\forall i \in \I)$ $F_{i}: \mathcal{H} \to \mathcal{H} $. Define
	\begin{align} \label{eq:lemma:shift}
	(\forall i \in \I)	(\forall x \in \mathcal{H})  \quad T_{i}x :=z+ F_{i} (x -z).
	\end{align}
	Let $ \gamma \in  \left[0,1\right[\,$. Then the following hold:
	\begin{enumerate}
		\item \label{lemma:shift:comp} Suppose that $\Fix F_{m}\cdots F_{1}$  or $\Fix T_{m}\cdots T_{1}$ is a  nonempty closed and convex subset of $\mathcal{H}$. Then $F_{m}\cdots F_{1}$ is a $\gamma$-BAM if and only if $T_{m}\cdots T_{1}$ is a $\gamma$-BAM.
		\item \label{lemma:shift:combi} Let $(\omega_{i})_{i \in \I}$ be in $\mathbb{R}$ such that $\sum_{i \in \I} \omega_{i} =1$. Suppose that $\Fix \sum_{i \in \I} \omega_{i} F_{i}$ or $\Fix \sum_{i \in \I} \omega_{i} T_{i}$ is a  nonempty closed and convex subset of $\mathcal{H}$.
		Then $\sum_{i \in \I} \omega_{i} F_{i}$ is a $\gamma$-BAM if and only if $\sum_{i \in \I} \omega_{i} T_{i}$ is a $\gamma$-BAM.
	\end{enumerate}
\end{lemma}

\begin{proof}
	Let $x \in \mathcal{H}$. By \cref{eq:lemma:shift}, it is easy to see that
	\begin{align*}
	&T_{m}\cdots T_{2}T_{1}x= 	T_{m}\cdots T_{2}(z+ F_{1} (x -z))=\cdots=z+F_{m}\cdots F_{2}F_{1} (x -z)\\
	&\sum_{i \in \I} \omega_{i} T_{i}x= \sum_{i \in \I} \omega_{i} \big(z+ F_{i} (x -z)\big) =z +\sum_{i \in \I} \omega_{i} F_{i}(x-z).
	\end{align*}
	Therefore, both \cref{lemma:shift:comp} and 	\cref{lemma:shift:combi} follow from \cref{prop:BAMAffinLinea}\ref{prop:BAMAffinLinea:BAM}.
\end{proof}

%%%%%%%%%%%%%%%%%%%%%%%%%%%%%%%%%%%%%%%%%%%%%%%%%%%%%%%%%%%%%%%%%%%%
%%%%%%%%%%%%%%%%%%%\section{Compositions of BAMs}%%%%%%%%%%%%%%%%%%%%%%%%%%
%%%%%%%%%%%%%%%%%%%%%%%%%%%%%%%%%%%%%%%%%%%%%%%%%%%%%%%%%%%%%%%%%%%%

\section{Compositions of BAMs} \label{sec:Composition:BAM}
In this section, we  study compositions of BAMs and determine whether the composition of BAMs is still a BAM or not.

\subsection*{Compositions of BAMs with closed and affine fixed point sets}
In this subsection, we consider   compositions of BAMs  with with closed and affine fixed point sets.

The following result is essential to the proof of  \cref{thm:ComposiBAM:Linear} below.
\begin{lemma} \label{lem:ComposiBAM:Linear}
	Set $\I:=\{1,2\}$. Let $(\forall i \in \I)$ $G_{i} :\mathcal{H} \to \mathcal{H}$, and let $\gamma_{i} \in \left[0,1\right[\,$. Set $(\forall i \in \I)$  $U_{i}:=\Fix G_{i}$. Suppose that $(\forall i \in \I)$ $ G_{i}$ is a $\gamma_{i} $-BAM and that $U_{i}$ is a closed linear subspace of $\mathcal{H}$.
	Denote the cosine $c( U_{1}, U_{2})$ of the Friedrichs angle between $ U_{1}$ and $U_{2}$ by $c_{F}$.
	Let $x \in \mathcal{H}$, and let
$x - \Pro_{U_{1}\cap U_{2}}x \neq 0$ and $G_{1}x - \Pro_{U_{1}\cap U_{2}}x \neq 0$.     Set
\begin{align} \label{eq:beta1:beta2}
\beta_{1} :=\frac{\norm{ \Pro_{U_{2}}G_{1}x - \Pro_{U_{1}\cap U_{2}}x}}{\norm{G_{1}x-\Pro_{U_{1}\cap U_{2}}x}} \quad \text{and} \quad \beta_{2}:= \frac{\norm{\Pro_{U_{1}}x - \Pro_{U_{1}\cap U_{2}}x}}{\norm{x-\Pro_{U_{1}\cap U_{2}}x}}.
\end{align}
	Then the following statements hold:
	\begin{enumerate}
		\item \label{lem:ComposiBAM:Linear:ineq} $\norm{ G_{2}G_{1}x - \Pro_{U_{1}\cap U_{2}}x}^{2} \leq 	\left(\gamma^{2}_{2} +(1 - \gamma^{2}_{2}) \beta^{2}_{1}  \right)  \left(\gamma^{2}_{1} +(1 - \gamma^{2}_{1} ) \beta^{2}_{2}  \right) \norm{x -\Pro_{U_{1}\cap U_{2}}x}^{2}$.
\item \label{lem:ComposiBAM:Linear:beta1beta2} $\beta_{1} \in [0,1]$ and $\beta_{2} \in [0,1]$.

		\item \label{lem:ComposiBAM:Linear:constant} Suppose that   $\Pro_{U_{1}}x - \Pro_{U_{1}\cap U_{2}}x \neq 0$ and $\Pro_{ U_{2}}G_{1}x - \Pro_{U_{1}\cap U_{2}}x  \neq 0$.
		Set
		\begin{align*}
		u:= \frac{G_{1}x- \Pro_{U_{1}\cap U_{2}}x}{\norm{G_{1}x- \Pro_{U_{1}\cap U_{2}}x}},\quad  v:=\frac{\Pro_{U_{1}}x-\Pro_{U_{1}\cap U_{2}}x}{\norm{\Pro_{U_{1}}x -\Pro_{U_{1}\cap U_{2}}x}}, \quad \text{and} \quad  w:=\frac{\Pro_{U_{2}}G_{1}x - \Pro_{U_{1}\cap U_{2}}x }{\norm{\Pro_{U_{2}}G_{1}x - \Pro_{U_{1}\cap U_{2}}x }} .
		\end{align*}
		Then
		\begin{subequations}
			\begin{align}
			&\innp{v,w} \leq c_{F},\label{lem:ComposiBAM:Linear:constant:cF} \\
			& \beta_{1} =\innp{u,w} \quad \text{and} \quad \beta_{2} \leq \innp{u,v},\label{lem:ComposiBAM:Linear:constant:beta} \\
			&\beta_{1}\beta_{2} \leq \frac{1+c_{F}}{2},\label{lem:ComposiBAM:Linear:constant:betacF}\\
			& \min\{ \beta_{1} , \beta_{2}  \} \leq  \sqrt{\frac{1+c_{F}}{2}}. \label{lem:ComposiBAM:Linear:sumbeta}
			\end{align}
		\end{subequations}
	\end{enumerate}
\end{lemma}

\begin{proof}
	Because $G_{1}$ is a $\gamma_{1} $-BAM  and  $G_{2}$ is a $ \gamma_{2} $-BAM,  by  \cref{def:BAM}  and  \cref{lemma:CFixG}\cref{lemma:CFixG:GC},  we have that
	\begin{align} \label{eq:prop:ComposiBAM:ProG}
	\Pro_{U_{1} }G_{1} = \Pro_{U_{1}} = G_{1}\Pro_{U_{1}}   \quad \text{and} \quad \Pro_{U_{2}} G_{2} = \Pro_{U_{2}}=G_{2}\Pro_{U_{2}},
	\end{align}
	and that
	\begin{align} \label{eq:prop:ComposiBAM:Ineq}
	(\forall y \in \mathcal{H}) \quad \norm{G_{1}y -\Pro_{U_{1}}y} \leq  \gamma_{1}   \norm{y - \Pro_{U_{1}}y} \quad \text{and} \quad \norm{G_{2}y -\Pro_{U_{2}}y} \leq \gamma_{2} \norm{y - \Pro_{U_{2}}y}.
	\end{align}
Note that by \cref{eq:prop:ComposiBAM:ProG} and \cref{MetrProSubs8}, we have that
	\begin{subequations}
		\begin{align}
		&G_{2}G_{1}x - \Pro_{U_{2}}G_{1}x \stackrel{\cref{eq:prop:ComposiBAM:ProG}}{=}G_{2}G_{1}x - \Pro_{ U_{2}}G_{2}G_{1}x =(\Id - \Pro_{ U_{2}})G_{2}G_{1}x= \Pro_{ U^{\perp}_{2}}G_{2}G_{1}x \in U^{\perp}_{2}, \label{eq:GWGW:perp:GWGV}\\
		& G_{1}x -\Pro_{U_{2}}G_{1}x =(\Id - \Pro_{ U_{2}}) G_{1}x   =\Pro_{ U^{\perp}_{2}} G_{1}x  \in U^{\perp}_{2}, \label{eq:GWGW:perp:GVPW}\\
		&G_{1}x-\Pro_{U_{1}}x \stackrel{\cref{eq:prop:ComposiBAM:ProG}}{=} G_{1}x-\Pro_{U_{1}} G_{1}x =(\Id - \Pro_{ U_{1}}) G_{1}x =\Pro_{ U^{\perp}_{1}}G_{1}x  \in U^{\perp}_{1}, \label{eq:GWGW:perp:GV}\\
		& x-\Pro_{U_{1}}x=(\Id - \Pro_{ U_{1}}) x=\Pro_{ U^{\perp}_{1}}x \in U^{\perp}_{1}. \label{eq:GWGW:perp:x}
		\end{align}
	\end{subequations}
	Hence, by the Pythagorean theorem, we obtain
	\begin{align} \label{eq:thm:ComposiBAM:GWGV}
	\norm{ \underbrace{G_{2}G_{1}x - \Pro_{U_{2}}G_{1}x}_{\in U^{\perp}_{2}}}^{2} +  \norm{ \underbrace{ \Pro_{U_{2}}G_{1} x - \Pro_{U_{1}\cap U_{2}}x}_{\in U_{2}}}^{2}= \norm{ G_{2}G_{1}x - \Pro_{U_{1}\cap U_{2}}x}^{2},
	\end{align}
	\begin{align} \label{eq:thm:ComposiBAM:GV:PWGV}
	\norm{\underbrace{G_{1}x -\Pro_{U_{2}}G_{1}x}_{\in U^{\perp}_{2}}}^{2}+\norm{\underbrace{\Pro_{U_{2}}G_{1}x - \Pro_{U_{1}\cap U_{2}}x}_{\in U_{2}} }^{2} =\norm{ G_{1}x -\Pro_{U_{1}\cap U_{2}}x }^{2},
	\end{align}
	\begin{align} \label{eq:thm:ComposiBAM:GV:PV}
	\norm{\underbrace{ G_{1}x-\Pro_{U_{1}}x}_{\in U^{\perp}_{1}}}^{2} +\norm{\underbrace{\Pro_{U_{1}}x-\Pro_{U_{1}\cap U_{2}}x}_{\in U_{1}} }^{2} =\norm{ G_{1}x - \Pro_{U_{1}\cap U_{2}}x }^{2},
	\end{align}
	\begin{align} \label{eq:thm:ComposiBAM:xPV}
	\norm{\underbrace{x-\Pro_{U_{1}}x}_{\in U^{\perp}_{1}} }^{2} + \norm{\underbrace{\Pro_{U_{1}}x-\Pro_{U_{1}\cap U_{2}}x}_{\in U_{1}} }^{2} = \norm{x-\Pro_{U_{1}\cap U_{2}}x }^{2}.
	\end{align}
	\cref{lem:ComposiBAM:Linear:ineq}:
Note that
\begin{subequations} \label{eq:prop:ComposiBAM:subeq}
	\begin{align}
	&\norm{ G_{2}G_{1}x - \Pro_{U_{1}\cap U_{2}}x}^{2} \nonumber\\
	\stackrel{\cref{eq:thm:ComposiBAM:GWGV}}{=} &   \norm{ G_{2}G_{1}x - \Pro_{U_{2}}G_{1}x}^{2} +  \norm{ \Pro_{U_{2}}G_{1}x - \Pro_{U_{1}\cap U_{2}}x}^{2} \nonumber\\
	\stackrel{\cref{eq:prop:ComposiBAM:Ineq}}{\leq}  & \gamma^{2}_{2}  \norm{ G_{1}x - \Pro_{U_{2}}G_{1}x}^{2} +  \norm{ \Pro_{U_{2}}G_{1}x - \Pro_{U_{1}\cap U_{2}}x}^{2}
	\nonumber \\
	=\,& \gamma^{2}_{2}  \left( \norm{ G_{1}x - \Pro_{U_{2}}G_{1}x}^{2} +\norm{ \Pro_{U_{2}}G_{1}x - \Pro_{U_{1}\cap U_{2}}x}^{2} \right)+(1 -\gamma^{2}_{2} )  \norm{ \Pro_{U_{2}}G_{1}x - \Pro_{U_{1}\cap U_{2}}x}^{2}  \nonumber\\
	\stackrel{\cref{eq:thm:ComposiBAM:GV:PWGV}}{=} & \gamma^{2}_{2}  \norm{ G_{1}x -\Pro_{U_{1}\cap U_{2}}x}^{2}   +(1 - \gamma^{2}_{2} )  \norm{ \Pro_{U_{2}}G_{1}x - \Pro_{U_{1}\cap U_{2}}x}^{2}  \nonumber\\
	\stackrel{\cref{eq:beta1:beta2}}{=} & \gamma^{2}_{2}  \norm{ G_{1}x -\Pro_{U_{1}\cap U_{2}}x}^{2}   +(1 - \gamma^{2}_{2} ) \beta^{2}_{1} \norm{G_{1}x-\Pro_{U_{1}\cap U_{2}}x}^{2}
	\nonumber\\
	=\,& \left(\gamma^{2}_{2}  +(1 -\gamma^{2}_{2} ) \beta^{2}_{1}  \right)  \norm{ G_{1}x-\Pro_{U_{1}\cap U_{2}}x}^{2}
	\nonumber \\
	\stackrel{\cref{eq:thm:ComposiBAM:GV:PV}}{=} & \left(\gamma^{2}_{2}  +(1 - \gamma^{2}_{2} ) \beta^{2}_{1}  \right) \left( \norm{G_{1}x-\Pro_{U_{1}}x}^{2} +\norm{\Pro_{U_{1}}x-\Pro_{U_{1}\cap U_{2}}x}^{2}  \right)
	\nonumber \\
	\stackrel{\cref{eq:prop:ComposiBAM:Ineq}}{\leq}  &  \left(\gamma^{2}_{2} +(1 -\gamma^{2}_{2} ) \beta^{2}_{1}  \right) \left( \gamma^{2}_{1} \norm{x-\Pro_{U_{1}}x}^{2} +\norm{\Pro_{U_{1}}x-\Pro_{U_{1}\cap U_{2}}x}^{2}  \right)
	\nonumber\\
	=\,&  \left(\gamma^{2}_{2}  +(1 - \gamma^{2}_{2} ) \beta^{2}_{1}  \right) \Big( \gamma^{2}_{1}  \left( \norm{x-\Pro_{U_{1}}x}^{2} + \norm{\Pro_{U_{1}}x-\Pro_{U_{1}\cap U_{2}}x}^{2} \right) +(1 -\gamma^{2}_{1} ) \norm{\Pro_{U_{1}}x-\Pro_{U_{1}\cap U_{2}}x}^{2}  \Big)
	\nonumber\\
	\stackrel{\cref{eq:thm:ComposiBAM:xPV}}{=}   &  \left(\gamma^{2}_{2}  +(1 - \gamma^{2}_{2} ) \beta^{2}_{1}  \right) \Big( \gamma^{2}_{1} \norm{x-\Pro_{U_{1}\cap U_{2}}x}^{2}  +(1- \gamma^{2}_{1} ) \norm{\Pro_{U_{1}}x-\Pro_{U_{1}\cap U_{2}}x}^{2}  \Big)
	\nonumber\\
	\stackrel{\cref{eq:beta1:beta2}}{=}  &  \left(\gamma^{2}_{2}  +(1 -\gamma^{2}_{2} ) \beta^{2}_{1}  \right) \Big( \gamma^{2}_{1} \norm{x-\Pro_{U_{1}\cap U_{2}}x}^{2} +(1 -\gamma^{2}_{1} ) \beta^{2}_{2}\norm{x -\Pro_{U_{1}\cap U_{2}}x}^{2}  \Big)
	\nonumber\\
	= \,&   \left(\gamma^{2}_{2} +(1 - \gamma^{2}_{2} ) \beta^{2}_{1}  \right)  \left(\gamma^{2}_{1}  +(1 - \gamma^{2}_{1} ) \beta^{2}_{2}  \right) \norm{x -\Pro_{U_{1}\cap U_{2}}x}^{2}.  \nonumber
	\end{align}
\end{subequations}	

\cref{lem:ComposiBAM:Linear:beta1beta2}:
This comes from \cref{eq:beta1:beta2}, \cref{eq:thm:ComposiBAM:GV:PWGV} and \cref{eq:thm:ComposiBAM:xPV}.

	\cref{lem:ComposiBAM:Linear:constant}:
	By \cref{lem:ExchangeProj} and   \cref{MetrProSubs8},  we know that
	\begin{align*}
	\Pro_{U_{1}}x -\Pro_{U_{1}\cap U_{2}}x = \Pro_{U_{1}}x -\Pro_{U_{1}\cap U_{2}}\Pro_{U_{1}}x = (\Id - \Pro_{U_{1} \cap U_{2} })\Pro_{U_{1}}x =\Pro_{(U_{1}\cap U_{2})^{\perp}} \Pro_{U_{1}}(x)\in U_{1} \cap (U_{1} \cap U_{2})^{\perp}.
	\end{align*}
By \cref{lem:ExchangeProj}, \cref{eq:prop:ComposiBAM:ProG} and \cref{MetrProSubs8}, $\Pro_{U_{1}\cap U_{2}}x =\Pro_{U_{1}\cap U_{2}} \Pro_{U_{1}}x=\Pro_{U_{1}\cap U_{2}}\Pro_{U_{1}} G_{1}x =\Pro_{U_{1}\cap U_{2}}\Pro_{U_{2}} G_{1}x$, so by \cref{MetrProSubs8},
\begin{align*}
\Pro_{U_{2}} G_{1} x - \Pro_{U_{1}\cap U_{2}}x =  \Pro_{U_{2}}G_{1} x - \Pro_{U_{1}\cap U_{2}} \Pro_{U_{2}}G_{1} x
=\Pro_{(U_{1} \cap U_{2})^{\perp}} \Pro_{U_{2}}G_{1} x \in U_{2} \cap (U_{1} \cap U_{2})^{\perp}.
\end{align*}
Hence, using
	\cref{defn:FredrichAngleClassical}, we obtain that
	\begin{align*}
	\innp{v,w} =\Innp{ \frac{\Pro_{U_{1}}x -\Pro_{U_{1}\cap U_{2}}x}{\norm{\Pro_{U_{1}}x-\Pro_{U_{1}\cap U_{2}}x}} , \frac{ \Pro_{U_{2}}G_{1} x - \Pro_{U_{1}\cap U_{2}}x}{\norm{\Pro_{U_{2}}G_{1} x - \Pro_{U_{1}\cap U_{2}}x }} } \leq c_{F},
	\end{align*}
which yields \cref{lem:ComposiBAM:Linear:constant:cF}.

It is easy to see that
	$ \innp{ G_{1}x- \Pro_{U_{1}\cap U_{2}}x,\Pro_{U_{2}}G_{1}x- \Pro_{U_{1}\cap U_{2}}x  } = \innp{ G_{1}x- \Pro_{U_{2}}G_{1}x,  \Pro_{U_{2}}G_{1}x - \Pro_{U_{1}\cap U_{2}}x}+ \innp{\Pro_{U_{2}}G_{1}x - \Pro_{U_{1}\cap U_{2}}x,\Pro_{U_{2}}G_{1}x- \Pro_{U_{1}\cap U_{2}}x} \stackrel{\cref{eq:GWGW:perp:GVPW}}{=}  \norm{\Pro_{U_{2}}G_{1}x - \Pro_{U_{1}\cap U_{2}}x}^{2}$. Hence,
	\begin{align}  \label{lemma:eq:com:BAM:beta1}
	\beta_{1}=\frac{\norm{\Pro_{U_{2}}G_{1}x - \Pro_{U_{1}\cap U_{2}}x}}{\norm{G_{1}x-\Pro_{U_{1}\cap U_{2}}x}} =\Innp{ \frac{G_{1}x- \Pro_{U_{1}\cap U_{2}}x}{\norm{G_{1}x- \Pro_{U_{1}\cap U_{2}}x}},   \frac{\Pro_{U_{2}}G_{1}x-\Pro_{U_{1}\cap U_{2}}x }{\norm{\Pro_{U_{2}}G_{1}x - \Pro_{U_{1}\cap U_{2}}x }} }   =\innp{u,w}.
	\end{align}
Moreover, by \cref{eq:prop:ComposiBAM:Ineq},  $\norm{G_{1}x -\Pro_{U_{1}}x} \leq  \gamma_{1}   \norm{x -\Pro_{U_{1}}x}  \leq \norm{x -\Pro_{U_{1}}x} $, then using \cref{eq:thm:ComposiBAM:GV:PV} and \cref{eq:thm:ComposiBAM:xPV}, we know that  $\norm{G_{1}x-\Pro_{U_{1}\cap U_{2}}x} \leq \norm{x-\Pro_{U_{1}\cap U_{2}}x}$. Hence,
	\begin{align} \label{lemma:eq:com:BAM:beta2}
	\beta_{2}=  \frac{\norm{\Pro_{U_{1}}x - \Pro_{U_{1}\cap U_{2}}x}}{\norm{x-\Pro_{U_{1}\cap U_{2}}x}}
	\leq   \frac{\norm{\Pro_{U_{1}}x - \Pro_{U_{1}\cap U_{2}}x}}{\norm{G_{1}x-\Pro_{U_{1}\cap U_{2}}x}}
	=  \Innp{\frac{G_{1}x- \Pro_{U_{1}\cap U_{2}}x}{ \norm{G_{1}x-\Pro_{U_{1}\cap U_{2}}x}}, \frac{\Pro_{U_{1}}x -\Pro_{U_{1}\cap U_{2}}x}{\norm{\Pro_{U_{1}}x -\Pro_{U_{1}\cap U_{2}}x}} }=\innp{u,v},
	\end{align}
	where the second equality is from $\innp{G_{1}x- \Pro_{U_{1}\cap U_{2}}x, \Pro_{U_{1}}x -\Pro_{U_{1}\cap U_{2}}x}= \innp{G_{1}x- \Pro_{U_{1}}x, \Pro_{U_{1}}x -\Pro_{U_{1}\cap U_{2}}x} +\innp{\Pro_{U_{1}}x- \Pro_{U_{1}\cap U_{2}}x, \Pro_{U_{1}}x -\Pro_{U_{1}\cap U_{2}}x} \stackrel{\cref{eq:GWGW:perp:GV}}{=} 0+ \norm{\Pro_{U_{1}}x- \Pro_{U_{1}\cap U_{2}}x }^{2}$.
	
	Hence, \cref{lemma:eq:com:BAM:beta1} and \cref{lemma:eq:com:BAM:beta2} yield \cref{lem:ComposiBAM:Linear:constant:beta}.
	
	Note that by \cref{lem:ComposiBAM:Linear:constant:beta} and  $\norm{u}=\norm{v}=\norm{w}=1$,
	\begin{align*}
	\beta_{1} +\beta_{2} \leq \innp{u, v+w} \leq \norm{u}\norm{v+w} =\sqrt{\norm{v}^{2} +2\innp{v,w} +\norm{w}^{2} } = \sqrt{2(1+\innp{v,w}) } \stackrel{\cref{lem:ComposiBAM:Linear:constant:cF}}{\leq} \sqrt{2(1+c_{F}) },
	\end{align*}
so
	\begin{align*}
	\beta_{1}\beta_{2} \leq \frac{(\beta_{1} +\beta_{2})^{2}}{4} \leq \frac{1+c_{F} }{2},
	\end{align*}
	which shows \cref{lem:ComposiBAM:Linear:constant:betacF}.
	
	We now turn to \cref{lem:ComposiBAM:Linear:sumbeta}. Suppose to the contrary that $\min\{ \beta_{1} , \beta_{2}  \} >  \sqrt{\frac{1+c_{F}}{2}}$. Then $\beta_{1} \beta_{2} > \frac{1+c_{F} }{2}$, which contradicts \cref{lem:ComposiBAM:Linear:constant:betacF}.
	
Altogether, the proof is complete.
\end{proof}

In the following  result, we  extend \cite[Lemma~1]{BCS2019} from $\mathbb{R}^{n}$ to $\mathcal{H}$ and also provide a new constant associated with the composition of BAMs. Although the following proof is shorter than the proof of \cite[Lemma~1]{BCS2019}, the main idea of the following proof is from the proof of \cite[Lemma~1]{BCS2019}.

\begin{theorem} \label{thm:ComposiBAM:Linear}
	Set $\I:=\{1,2\}$. Let $(\forall i \in \I)$ $G_{i} :\mathcal{H} \to \mathcal{H}$, and let $\gamma_{i} \in \left[0,1\right[\,$. Set $(\forall i \in \I)$  $U_{i}:=\Fix G_{i}$. Suppose that $(\forall i \in \I)$ $ G_{i}$ is a $\gamma_{i} $-BAM and that $U_{i}$ is a closed linear subspace of $\mathcal{H}$ such that $U_{1}+U_{2}$ is closed.
	Denote the cosine $c( U_{1}, U_{2})$ of the Friedrichs angle between $ U_{1}$ and $U_{2}$ by $c_{F}$.
	Then the following  hold:
	\begin{enumerate}
		\item \label{thm:ComposiBAM:PVcapW} $\Fix (G_{2}\circ G_{1}) = \Fix G_{1} \cap \Fix G_{2}$ is a closed linear subspace of $\mathcal{H}$ and $\Pro_{U_{1}\cap U_{2}} G_{2}G_{1}=\Pro_{U_{1}\cap U_{2}} $.
		\item  \label{thm:ComposiBAM:Linear:eq0} Let $x \in \mathcal{H}$. If  $x - \Pro_{U_{1}\cap U_{2}}x = 0$ or $G_{1}x - \Pro_{U_{1}\cap U_{2}}x = 0$, then $\norm{ G_{2}G_{1}x - \Pro_{U_{1}\cap U_{2}}x}=0$.
		\item  \label{thm:ComposiBAM:rVW}  Set
		\begin{align} \label{eq:rVcapW}
		r:=\max \left\{ \sqrt{\gamma^{2}_{1}  + (1-\gamma^{2}_{1} )\frac{1+c_{F}}{2}}, \sqrt{\gamma^{2}_{2}  + (1-\gamma^{2}_{2} )\frac{1+c_{F}}{2}}  \right\}.
		\end{align}
	Then	$r \in \big[\max \{  \gamma_{1}   , \gamma_{2} \},1 \big[\,$. Moreover,
	\begin{align} \label{eq:thm:ComposiBAM:ineq}
(\forall x \in \mathcal{H})\quad 	\norm{ G_{2}G_{1}x - \Pro_{U_{1}\cap U_{2}}x} \leq r \norm{x - \Pro_{U_{1}\cap U_{2}} x}.
	\end{align}
		\item \label{thm:ComposiBAM:Linear:lVW} Set
		\begin{align} \label{eq:lVcapW}
		s:=\sqrt{\gamma^{2}_{1} +\gamma^{2}_{2}   -\gamma^{2}_{1}  \gamma^{2}_{2}  +(1-\gamma^{2}_{1} ) (1-\gamma^{2}_{2} ) \frac{(1+c_{F})^{2}}{4}}.
		\end{align}
		Then $s \in \big[\max \{ \gamma_{1}   , \gamma_{2}, \frac{1}{2} \}, 1 \big[\,$. Moreover,
		\begin{align} \label{eq:thm:ComposiBAM:gamma}
	(\forall x \in \mathcal{H})\quad  	\norm{ G_{2}G_{1}x - \Pro_{U_{1}\cap U_{2}}x} \leq s \norm{x - \Pro_{U_{1}\cap U_{2}}x}.
		\end{align}
		
		\item \label{thm:ComposiBAM:BAM} $G_{2}\circ G_{1}$ is a $\min \{r ,s \}$-BAM.
	\end{enumerate}
\end{theorem}	

\begin{proof}
Because $U_{1} + U_{2}$ is closed, by \cref{fac:cFLess1}, we know that
	\begin{align}   \label{eq:prop:ComposiBAM:cF}
	c_{F} :=c(U_{1},U_{2}) \in \left[0,1\right[\,.
	\end{align}
	Because $G_{1}$ is a $ \gamma_{1}   $-BAM    and  $G_{2}$ is a $\gamma_{2}  $-BAM,  by  \cref{def:BAM}  and  \cref{lemma:CFixG}\cref{lemma:CFixG:GC},  we have that
	\begin{align} \label{eq:prop:ComposiBAM:ProG:TH}
	\Pro_{U_{1}} G_{1}= \Pro_{U_{1}} = G_{1}\Pro_{U_{1}}   \quad \text{and} \quad \Pro_{U_{2}} G_{2} = \Pro_{U_{2}}=G_{2}\Pro_{U_{2}},
	\end{align}
	and that
	\begin{align} \label{eq:prop:ComposiBAM:Ineq:TH}
	(\forall x \in \mathcal{H}) \quad  \norm{G_{2}x -\Pro_{U_{2}}x} \leq \gamma_{2} \norm{x - \Pro_{U_{2}}x}.
	\end{align}
	
		\cref{thm:ComposiBAM:PVcapW}: By assumptions and by \cref{prop:CompositionBAM}\cref{prop:CompositionBAM:comp},  $\Fix (G_{2}\circ G_{1}) =\Fix G_{1} \cap \Fix G_{2}= U_{1} \cap U_{2} $ is a closed linear subspace of $\mathcal{H}$.
		 Because $U_{1} \cap U_{2} \subseteq U_{1}$ and $U_{1} \cap U_{2} \subseteq U_{2}$, by \cref{lem:ExchangeProj}, we know that
	\begin{align} \label{eq:prop:ComposiBAM:ProVW}
	\Pro_{U_{1}\cap U_{2}} \Pro_{U_{2}}=\Pro_{U_{1}\cap U_{2}}= \Pro_{U_{2}}\Pro_{U_{1}\cap U_{2}} \quad \text{and} \quad\Pro_{U_{1}\cap U_{2}}\Pro_{U_{1}}=\Pro_{U_{1}\cap U_{2}} = \Pro_{U_{1}} \Pro_{U_{1}\cap U_{2}}.
	\end{align}
	Moreover,
	\begin{align*}
	\Pro_{U_{1}\cap U_{2}} G_{2}G_{1} \stackrel{\cref{eq:prop:ComposiBAM:ProVW} }{=}  \Pro_{U_{1}\cap U_{2}}\Pro_{U_{2}}G_{2}G_{1} \stackrel{\cref{eq:prop:ComposiBAM:ProG:TH} }{=}  \Pro_{U_{1}\cap U_{2}}\Pro_{U_{2}}G_{1}  \stackrel{\cref{eq:prop:ComposiBAM:ProVW} }{=} \Pro_{U_{1}\cap U_{2}}\Pro_{U_{1}} G_{1} \stackrel{\cref{eq:prop:ComposiBAM:ProG:TH} }{=}   \Pro_{U_{1}\cap U_{2}}\Pro_{U_{1}} \stackrel{\cref{eq:prop:ComposiBAM:ProVW} }{=}  \Pro_{U_{1}\cap U_{2}}.
	\end{align*}
	
	\cref{thm:ComposiBAM:Linear:eq0}: If $x- \Pro_{U_{1}\cap U_{2}}x=0$, then by \cref{eq:prop:ComposiBAM:ProVW} and \cref{eq:prop:ComposiBAM:ProG:TH},  then $G_{2}G_{1}x=G_{2}G_{1}\Pro_{U_{1}\cap U_{2}}x=G_{2}G_{1}\Pro_{U_{1}}\Pro_{U_{1}\cap U_{2}}x=G_{2}\Pro_{U_{1}}\Pro_{U_{1}\cap U_{2}}x=G_{2}\Pro_{U_{2}}\Pro_{U_{1}\cap U_{2}}x=\Pro_{U_{2}}\Pro_{U_{1}\cap U_{2}}x=\Pro_{U_{1}\cap U_{2}}x$. Hence,
	$
	\norm{ G_{2}G_{1}x - \Pro_{U_{1}\cap U_{2}}x}= 0
	$.
	
	If $G_{1}x - \Pro_{U_{1}\cap U_{2}}x=0$, then by \cref{eq:prop:ComposiBAM:ProVW} and \cref{eq:prop:ComposiBAM:Ineq:TH},
	$
	\norm{ G_{2}G_{1}x -\Pro_{U_{1}\cap U_{2}}x} = \norm{G_{2} \Pro_{U_{1}\cap U_{2}}x- \Pro_{ U_{2}}\Pro_{U_{1}\cap U_{2}}x}
	\leq \gamma_{2}   \norm{ \Pro_{U_{1}\cap U_{2}}x - \Pro_{ U_{2}}\Pro_{U_{1}\cap U_{2}}x }
	=0
	$, that is, $\norm{ G_{2}G_{1}x - \Pro_{U_{1}\cap U_{2}}x} =0$.

	\cref{thm:ComposiBAM:rVW}:
	Because $\gamma_{1} \in \left[0,1\right[$ and $\gamma_{2} \in \left[0,1\right[$, by \cref{eq:prop:ComposiBAM:cF}, $r \in \big[\max \{  \gamma_{1}   , \gamma_{2}  \},1 \big[\, .$
	
We shall prove \cref{eq:thm:ComposiBAM:ineq} next.  Let $x \in \mathcal{H}$. By \cref{thm:ComposiBAM:Linear:eq0}, we are able to assume
	$x - \Pro_{U_{1}\cap U_{2}}x \neq 0$ and $G_{1}x - \Pro_{U_{1}\cap U_{2}}x \neq 0$.
	We define $\beta_{1}$ and $\beta_{2}$ as  in \cref{lem:ComposiBAM:Linear}.
	
	Note that if $\Pro_{U_{1}}x - \Pro_{U_{1}\cap U_{2}}x=0$, then $\beta_{2} =0$. Moreover,  by \cref{lem:ComposiBAM:Linear}\cref{lem:ComposiBAM:Linear:ineq}$\&$\cref{lem:ComposiBAM:Linear:beta1beta2},
	$
	\norm{ G_{2}G_{1}x - \Pro_{U_{1}\cap U_{2}}x}^{2}  \leq \left(\gamma^{2}_{2}  +(1 - \gamma^{2}_{2} ) \beta^{2}_{1}  \right)\gamma^{2}_{1}  \norm{x -\Pro_{U_{1}\cap U_{2}}x}^{2} \leq
	\gamma^{2}_{1}   \norm{x -\Pro_{U_{1}\cap U_{2}}x}^{2}.
	$
	If $\Pro_{U_{2}}G_{1}x - \Pro_{U_{1}\cap U_{2}}x =0$, then $\beta_{1} =0$ and, by \cref{lem:ComposiBAM:Linear}\cref{lem:ComposiBAM:Linear:ineq}$\&$\cref{lem:ComposiBAM:Linear:beta1beta2},
	$
	\norm{ G_{2}G_{1}x - \Pro_{U_{1}\cap U_{2}}x}^{2}  \leq  \gamma^{2}_{2}   \left(\gamma^{2}_{1}  +(1 - \gamma^{2}_{1} ) \beta^{2}_{2}  \right)   \norm{x -\Pro_{U_{1}\cap U_{2}}x}^{2} \leq
	\gamma^{2}_{2}   \norm{x -\Pro_{U_{1}\cap U_{2}}x}^{2}$. Because $\max \{  \gamma_{1}  , \gamma_{2}  \} \leq r$,  we know that in these two cases, \cref{eq:thm:ComposiBAM:ineq} is true. So in the rest of the proof, we assume that $\Pro_{U_{1}}x - \Pro_{U_{1}\cap U_{2}}x \neq 0$ and $\Pro_{U_{2}}G_{1}x - \Pro_{U_{1}\cap U_{2}}x  \neq 0$.

Using \cref{lem:ComposiBAM:Linear:sumbeta} in \cref{lem:ComposiBAM:Linear} \cref{lem:ComposiBAM:Linear:constant}, we obtain that
 	\begin{align}  \label{eq:thm:ComposiBAM:rVW:sumbeta}
 \beta_{1}  \leq  \sqrt{\frac{1+c_{F}}{2}} \quad \text{or} \quad \beta_{2}  \leq  \sqrt{\frac{1+c_{F}}{2}}.
 \end{align}

Therefore, combine \cref{lem:ComposiBAM:Linear}\cref{lem:ComposiBAM:Linear:ineq}$\&$\cref{lem:ComposiBAM:Linear:beta1beta2} with \cref{eq:rVcapW} and \cref{eq:thm:ComposiBAM:rVW:sumbeta} to obtain \cref{eq:thm:ComposiBAM:ineq}. 	Altogether, \cref{thm:ComposiBAM:rVW} holds.

	\cref{thm:ComposiBAM:Linear:lVW}: $	s=\sqrt{\gamma^{2}_{1}  +\gamma^{2}_{2}  - \gamma^{2}_{1}  \gamma^{2}_{2}   +(1-\gamma^{2}_{1} ) (1-\gamma^{2}_{2} ) \frac{(1+c_{F})^{2}}{4}}= \sqrt{\gamma^{2}_{1}  +(1-\gamma^{2}_{1} ) \big( \gamma^{2}_{2} +  (1-\gamma^{2}_{2} ) \frac{(1+c_{F})^{2}}{4}} \big)$ and $\gamma_{1} \in \left[0,1\right[$ and $\gamma_{2} \in \left[0,1\right[$ are symmetric in the expression of $s$. So, by \cref{eq:prop:ComposiBAM:cF}, $s \in \big[\max \{ \gamma_{1}  , \gamma_{2}  \}, 1 \big[\,$.   In addition, some elementary algebraic manipulations yield $\gamma^{2}_{1}  +\gamma^{2}_{2}  - \gamma^{2}_{1}  \gamma^{2}_{2}   +(1-\gamma^{2}_{1} ) (1-\gamma^{2}_{2} ) \frac{(1+c_{F})^{2}}{4} \geq \gamma^{2}_{1}  +\gamma^{2}_{2}  - \gamma^{2}_{1}  \gamma^{2}_{2}   +(1-\gamma^{2}_{1} ) (1-\gamma^{2}_{2} ) \frac{1}{4} \geq \frac{1}{4}$.
	 Hence,  $s \in \big[\max \{ \gamma_{1}  , \gamma_{2} , \frac{1}{2} \}, 1 \big[\,$.
	
	We prove \cref{eq:thm:ComposiBAM:gamma} next. Let $x \in \mathcal{H}$. Because $s\geq  \max \{ \gamma_{1}  , \gamma_{2}  \}$,
	similarly to the proof of \cref{thm:ComposiBAM:rVW}, to show \cref{eq:thm:ComposiBAM:gamma}, we are able to assume $x - \Pro_{U_{1}\cap U_{2}}x \neq 0$ and $G_{1}x - \Pro_{U_{1}\cap U_{2}}x \neq 0$, $\Pro_{U_{1}}x - \Pro_{U_{1}\cap U_{2}}x \neq 0$ and $\Pro_{U_{2}}G_{1}x - \Pro_{U_{1}\cap U_{2}}x  \neq 0$.
	Define   $\beta_{1}$ and $\beta_{2}$ as in \cref{lem:ComposiBAM:Linear}.
	
	Use \cref{lem:ComposiBAM:Linear}\cref{lem:ComposiBAM:Linear:beta1beta2} and   \cref{lem:ComposiBAM:Linear:constant:betacF}  in \cref{lem:ComposiBAM:Linear}\cref	{lem:ComposiBAM:Linear:constant} respectively in the following two inequalities to obtain that
	\begin{subequations}
		\begin{align*}
		\left(\gamma^{2}_{2}  +(1 - \gamma^{2}_{2} ) \beta^{2}_{1}  \right)  \left(\gamma^{2}_{1}  +(1 - \gamma^{2}_{1} ) \beta^{2}_{2}  \right) &=
		\gamma^{2}_{2} \gamma^{2}_{1} +\gamma^{2}_{2}  (1 - \gamma^{2}_{1} ) \beta^{2}_{2} +\gamma^{2}_{1} (1 - \gamma^{2}_{2} ) \beta^{2}_{1}
		+(1 - \gamma^{2}_{2} )  (1 - \gamma^{2}_{1} )\beta^{2}_{1} \beta^{2}_{2}  \\
		&\leq
		\gamma^{2}_{2} \gamma^{2}_{1} +\gamma^{2}_{2}  (1 -\gamma^{2}_{1} ) +\gamma^{2}_{1} (1 - \gamma^{2}_{2} )
		+(1 - \gamma^{2}_{2} ) (1 - \gamma^{2}_{1} ) \beta^{2}_{1}\beta^{2}_{2}  \\
		& \leq
		\gamma^{2}_{1} +\gamma^{2}_{2}   - \gamma^{2}_{1}  \gamma^{2}_{2}   +(1-\gamma^{2}_{1} ) (1-\gamma^{2}_{2} ) \frac{(1+c_{F})^{2}}{4}=s^{2}.
		\end{align*}
	\end{subequations}
This and  \cref{lem:ComposiBAM:Linear}\cref{lem:ComposiBAM:Linear:ineq}  yield that
	\begin{align*}
	\norm{ G_{2}G_{1}x - \Pro_{U_{1}\cap U_{2}}x}^{2} &\leq 	\left(\gamma^{2}_{2} +(1 - \gamma^{2}_{2} ) \beta^{2}_{1}  \right)  \left(\gamma^{2}_{1} +(1 - \gamma^{2}_{1} ) \beta^{2}_{2}  \right) \norm{x -\Pro_{U_{1}\cap U_{2}}x}^{2}\\
	&\leq s^{2} \norm{x -\Pro_{U_{1}\cap U_{2}}x}^{2}.
	\end{align*}
	Hence, 	\cref{thm:ComposiBAM:Linear:lVW} holds.
	
	\cref{thm:ComposiBAM:BAM}:
	Combine \cref{def:BAM} with  \cref{thm:ComposiBAM:PVcapW}, \cref{thm:ComposiBAM:rVW} and \cref{thm:ComposiBAM:Linear:lVW}	  to obtain 	that  $G_{2}\circ G_{1}$ is a $\min \{r, s\}$-BAM.
\end{proof}	

\begin{lemma} \label{lemma:perp:equiv}
	Set $\I:=\{1, \ldots, m\}$.	Let $U_{1}, \ldots, U_{m}$ be closed linear subspaces of $\mathcal{H}$. The following hold:
	\begin{enumerate}
		\item \label{lemma:perp:equiv:i} Let $ i \in \I \smallsetminus\{m\}$. Then
			\begin{align*}
	  U_{i+1} + \cap^{i}_{j=1}   U_{j} \text{ is closed }   \Leftrightarrow  U^{\perp}_{i+1} + (\cap^{i}_{j=1}   U_{j})^{\perp} \text{ is closed }
		\Leftrightarrow  U^{\perp}_{i+1} +\overline{\sum^{i}_{j=1}   U^{\perp} _{j}} \text{ is closed}.
		\end{align*}
		\item \label{lemma:perp:equiv:EQ} $ (\forall i \in \I \smallsetminus\{m\}) $ $U_{i+1} + \cap^{i}_{j=1}   U_{j} $ is closed if and only if $ (\forall i \in \I )$ $\sum^{i}_{j=1}   U^{\perp}_{j}$  is closed
	\end{enumerate}
\end{lemma}

\begin{proof}
\cref{lemma:perp:equiv:i}:	The two equivalences follow by  \cref{fac:cFLess1} and \cite[Theorem~4.6(5)]{D2012} respectively.

\cref{lemma:perp:equiv:EQ}: Note that by \cite[Theorem~4.5(1)]{D2012},
$ U_{1}^{\perp}$ is a closed linear subspace of $\mathcal{H}$, that is, $U_{1}^{\perp} =\overline{U_{1}^{\perp}}$. Then the asserted result follows from \cref{lemma:perp:equiv:i} by the principle of strong mathematical induction on $m$.
\end{proof}

\begin{theorem} \label{thm:BAM:COMPO}
	Set $\I:= \{1, \ldots, m\}$.	
	Let $(\forall i \in \I)$ $\gamma_{i} \in \left[0,1\right[\,$ and let
	$G_{i} : \mathcal{H} \to \mathcal{H}$ be a $\gamma_{i}$-BAM such that $ U_{i}:=\Fix G_{i}$ is a closed affine subspaces of $\mathcal{H}$ with $\cap_{i \in \I} U_{i} \neq \varnothing$. Assume that $(\forall i \in \I )$ $\sum^{i}_{j=1}   (\pa U_{j} )^{\perp} $ is closed. Then the following statements hold:
	\begin{enumerate}
		\item \label{thm:BAM:COMPO:Fix} $(\forall k \in \{1, \ldots, m\})$ $\Fix (G_{k} \circ \cdots \circ G_{1}) = \cap^{k}_{i=1} \Fix G_{i} $ is a closed affine subspaces of $\mathcal{H}$.
		\item  \label{thm:BAM:COMPO:BAM} $G_{m} \circ \cdots \circ G_{2} \circ G_{1}$ is a  BAM. \item \label{thm:BAM:COMPO:BAM:2} Suppose that $m=2$. Denote the cosine $c( \pa U_{1}, \pa U_{2})$ of the Friedrichs angle between $ \pa U_{1}$ and $\pa U_{2}$ by $c_{F}$.   Set
		\begin{align*}
	r:=\max_{i\in \I}  \sqrt{\gamma^{2}_{i}  + (1-\gamma^{2}_{i} )\frac{1+c_{F}}{2}},
		\quad \text{and} \quad s:=\sqrt{\gamma^{2}_{1} +\gamma^{2}_{2}   - \gamma^{2}_{1}  \gamma^{2}_{2}  +(1-\gamma^{2}_{1} ) (1-\gamma^{2}_{2} ) \frac{(1+c_{F})^{2}}{4}}.
		\end{align*}
		Then $\min \{ r , s \} \in \left[0,1\right[\,$ and $G_{2}\circ G_{1}$ is a  $\min \{ r , s \}$-BAM.
		\item \label{thm:BAM:COMPO:LineaConve}  There exists $\gamma \in \left[0,1\right[$ such that
		\begin{align*}
		(\forall x \in \mathcal{H}) \quad \norm{ (G_{m} \circ \cdots \circ G_{2} \circ G_{1})^{k}x - \Pro_{ \cap^{m}_{i=1} U_{i}} x} \leq \gamma^{k} \norm{x - \Pro_{ \cap^{m}_{i=1} U_{i}} x}.
		\end{align*}
	\end{enumerate}
\end{theorem}	

\begin{proof}
	\cref{thm:BAM:COMPO:Fix}: This is from \cref{prop:CompositionBAM}\cref{prop:CompositionBAM:comp}.

	\cref{thm:BAM:COMPO:BAM}$\&$\cref{thm:BAM:COMPO:BAM:2}: Let $z \in \cap_{i \in \I} U_{i} $. Define $(\forall i \in \I)$ $F_{i} :\mathcal{H} \to \mathcal{H}$ by
	\begin{align} \label{eq:thm:BAM:COMPO:F}
	(\forall x \in \mathcal{H}) \quad 	F_{i} (x) := G_{i}(x +z) -z
	\end{align}
	By  the assumptions, \cref{eq:thm:BAM:COMPO:F} and \cref{prop:BAMAffinLinea}, $F_{i}$ is a $\gamma_{i} $-BAM with $\Fix F_{i} =\pa U_{i}$ being a closed linear subspace of $\mathcal{H}$.
	Hence,   by \cref{thm:BAM:COMPO:Fix}, \cref{eq:thm:BAM:COMPO:F} and \cref{lemma:shift}\cref{lemma:shift:comp},  we are able to assume that $U_{1}, \ldots, U_{m}$ are  closed linear subspaces of $\mathcal{H}$. Then  \cref{thm:BAM:COMPO:BAM:2} reduces to \cref{thm:ComposiBAM:Linear}\cref{thm:ComposiBAM:BAM}.
	
	We prove 	\cref{thm:BAM:COMPO:BAM} next.
	If $m=1$, then there is nothing to prove. Suppose  that $m \geq 2$. We prove it by induction on $k \in \{1, \ldots, m\}$. By assumption,  $G_{1}$ is a  BAM, so the base case is true.  Assume $G_{k} \circ \cdots   \circ G_{1}$ is a  BAM for some $k \in \{1,2,\ldots, m-1\}$.
	By the assumption, $(\forall i \in \I )$  $ \sum^{i}_{j=1}   U^{\perp}_{j} $ is closed,  and by \cref{thm:BAM:COMPO:Fix} and  \cref{lemma:perp:equiv}\cref{lemma:perp:equiv:EQ}, we know that
	$\Fix (G_{k} \circ \cdots   \circ G_{1})  + \Fix G_{k+1} = (\cap^{k}_{j=1}   U_{j})  +U_{k+1}$ is closed. Hence,
	apply \cref{thm:ComposiBAM:Linear}\cref{thm:ComposiBAM:BAM} with $G_{1}=G_{k} \circ \cdots   \circ G_{1}$ and $G_{2} = G_{k+1}$   to obtain that $G_{k+1} \circ G_{k} \circ \cdots   \circ G_{1}$ is a  BAM.
	Therefore,  \cref{thm:BAM:COMPO:BAM}  holds as well.

	\cref{thm:BAM:COMPO:LineaConve}: This comes from \cref{thm:BAM:COMPO:BAM} and \cref{prop:BAM:Properties}.
\end{proof}

The following \cref{rema:comp:BAM}\cref{rema:comp:BAM:PP} and \cref{remark:ComposiBAM:Linear}\cref{remark:ComposiBAM:Linear:constant}   exhibit a case  where the new constant $s$ associated with the composition of BAMs presented in \cref{thm:ComposiBAM:Linear}\cref{thm:ComposiBAM:BAM} is better than the constant $r $ from \cite{BCS2019}.
Moreover, \cref{rema:comp:BAM} illustrates that generally $\min\{r,s\}$ in \cref{thm:ComposiBAM:Linear} is not a sharp constant for the composition of BAMs.
\begin{remark} \label{rema:comp:BAM}
	Let $L_{1}$ and $L_{2}$ be closed linear subspaces of $\mathcal{H}$. Assume that $L_{1} + L_{2}$ is closed. Denote by $c_{F}:=c(L_{1}, L_{2})$ the Friedrichs angle between $L_{1}$ and $L_{2}$.  By \cite[Corollary~4.5.2]{Cegielski},   $\Fix \Pro_{L_{2}}\Pro_{L_{1}} =L_{1} \cap L_{2}$ is a closed linear subspace of $\mathcal{H}$. By \cref{examp:BAM:Pro}, both $ \Pro_{L_{1}}$ and  $ \Pro_{L_{2}}$ are  $0$-BAM.  Moreover,  the following hold:
	\begin{enumerate}
		\item \label{rema:comp:BAM:PP} Apply \cref{thm:ComposiBAM:Linear}\cref{thm:ComposiBAM:BAM} with $G_{1}=\Pro_{L_{1}}$, $G_{2}=\Pro_{L_{2}}$, $ \gamma_{1}  =0$, $ \gamma_{2} =0$ to obtain that $ \min \{ \sqrt{\frac{1+c_{F}}{2}}, \frac{1+c_{F}}{2}\}=\frac{1+c_{F}}{2}$ and $\Pro_{L_{2}}\Pro_{L_{1}}$ is a $\frac{1+c_{F}}{2}$-BAM.
		
		\item 	\label{rema:comp:BAM:ineq} By \cite[Lemma~9.5(7) and Theorem~9.8]{D2012},
		\begin{align*}
		(\forall x \in \mathcal{H}) \quad \norm{\Pro_{L_{2}} \Pro_{L_{1}}x - \Pro_{  L_{1} \cap L_{2} } x  }  \leq  c_{F}  \norm{x - \Pro_{  L_{1} \cap L_{2} }x},
		\end{align*}
	and $c_{F} $ is the smallest constant satisfying the inequality above.
		Hence, $\Pro_{L_{2}}\Pro_{L_{1}}$ is a BAM with sharp constant $c_{F} $.
		\end{enumerate}
	Recall that $c_{F} :=c(U_{1}, U_{2}) \in \left[0,1\right[\,$, so $c_{F} < \frac{1+c_{F}}{2}$.
	Hence, we know that
	generally the constant associated with the composition of  BAMs provided  by \cref{thm:ComposiBAM:Linear}\cref{thm:ComposiBAM:BAM}   is not sharp.
\end{remark}		
The following \cref{remark:ComposiBAM:Linear}\cref{rema:comp:BAM:l:r} presents examples showing that the constants $s  $ and $  r $  in \cref{thm:ComposiBAM:Linear} are  independent.
\begin{remark} \label{remark:ComposiBAM:Linear}
	Consider the constants $r,s$ in \cref{thm:ComposiBAM:Linear}\cref{thm:ComposiBAM:BAM} .
	\begin{enumerate}
		\item \label{remark:ComposiBAM:Linear:constant} Suppose that $\gamma_{1}=0$ or  $\gamma_{2}=0$, that is, $G_{1} =\Pro_{U_{1}}$ or $G_{2} =\Pro_{U_{2}}$. Without loss of generality, let $\gamma_{2}=0$. Then
		\begin{align*}
		r  =\sqrt{\gamma^{2}_{1}  + (1-\gamma^{2}_{1} )\frac{1+c_{F}}{2}}, \quad \text{and} \quad
		&s=\sqrt{\gamma^{2}_{1} +(1-\gamma^{2}_{1} ) \frac{(1+c_{F})^{2}}{4}}.
		\end{align*}
		Therefore,  $s  \leq  r $.
		\item \label{rema:comp:BAM:l:r} Suppose that  $\gamma:=\gamma_{1} =\gamma_{2} \in \left[0,1\right[\,$ and that $c_{F} =0$. Then
		\begin{align*}
		r =\sqrt{\gamma^{2}+ (1-\gamma^{2} )\frac{1}{2}} \quad \text{and} \quad
		s =\sqrt{2\gamma^{2}  - \gamma^{4}  +(1-\gamma^{2} )^{2} \frac{1}{4}}.
		\end{align*}
		Hence
		\begin{align*}
		s^{2} -r^{2}
		=\frac{  (1-\gamma^{2} )  }{4}  (  3\gamma^{2}  -1 ),
		\end{align*}
		which implies that
		\begin{align*}
		s  \geq  r  \Leftrightarrow  \gamma \in \left[\frac{\sqrt{3}}{3},1\right[\, \quad \text{and} \quad s  <  r  \Leftrightarrow \gamma \in \left[ 0, \frac{\sqrt{3}}{3} \right[\, .
		\end{align*}
	\end{enumerate}
\end{remark}

\subsection*{Compositions of BAMs with general convex fixed point sets}
In this subsection, we investigate compositions of BAMs with general closed and convex  fixed point sets.

By \cref{examp:BAM:Pro},   the projection onto a nonempty closed convex subset of $\mathcal{H}$ is the most common BAM.
The following results show that the order of the projections does matter to determine whether the composition of projections is a BAM or not.
The next result considers the composition of projections onto a cone and a ball.
\begin{proposition}  \label{prop:coneball}
	Let $K$ be a nonempty closed convex cone in $\mathcal{H}$, and let $\rho \in \mathbb{R}_{++}$. Denote by $B:=\mathbf{B}[0; \rho]$.
	\begin{enumerate}
		\item \label{item:prop:coneball:cone} $\Pro_{B}\Pro_{K}=\Pro_{K \cap B }$  is a $0$-BAM.
		\item \label{item:prop:coneball:ball} Suppose that $\mathcal{H} =\mathbb{R}^{2}$, $K=\mathbb{R}^{2}_{+}$ and $\rho=1$. Then $\Pro_{K}\Pro_{B}$ is not a BAM.
	\end{enumerate}
\end{proposition}

\begin{proof}
\cref{item:prop:coneball:cone}: By	\cite[Corollary~7.3]{BBW2018}, $\Pro_{B}\Pro_{K}=\Pro_{K \cap B }$, which, by \cref{cor:BAMProPro}, yields that $\Pro_{B}\Pro_{K}$  is a $0$-BAM.

\cref{item:prop:coneball:ball}: By \cite[Corollary~4.5.2]{Cegielski}, $\Fix \Pro_{K}\Pro_{B}= K \cap B $.
By	\cite[Example~7.5]{BBW2018}, we know that
\begin{align*}
\Pro_{K \cap B }\Pro_{K}\Pro_{B}(1,-1) =\left(\frac{1}{\sqrt{2}}, 0\right) \neq (1,0)=\Pro_{K \cap B}(1,-1),
\end{align*}
which implies that $\Pro_{K \cap B}\Pro_{K}\Pro_{B} \neq \Pro_{K \cap B }$. So, by \cref{def:BAM}, $\Pro_{K}\Pro_{B} $ is not a BAM.
\end{proof}

The following example considers projections onto an affine subspace and a cone.

\begin{example} \label{exam:LineCone}
	Suppose $\mathcal{H} =\mathbb{R}^{2}$. Let $U:= \{(x_{1},x_{2} ) \in \mathbb{R}^{2} ~:~ x_{2} =-x_{1} +1   \}$  and $K:=\mathbb{R}^{2}_{+}$. Then the  following hold (see also \cref{fig:examp:BAM:LineCone}):
	\begin{enumerate}
		\item \label{exam:LineCone:ConeLine} $\Pro_{U}\Pro_{K}$ is not a BAM.
		\item \label{exam:LineCone:LineCone} $\Pro_{K}\Pro_{U}$ is a $ \frac{\sqrt{2}}{2}  $-BAM.
	\end{enumerate}
\end{example}

\begin{proof}
	Define the lines $L_{1} :=\mathbb{R} \cdot (1,0)$, $L_{2} :=\mathbb{R} \cdot (0,1)$, $l_{1}:= \{(x_{1},x_{2} ) \in \mathbb{R}^{2} ~:~ x_{2} =x_{1} -1   \}$ and $l_{2}:= \{(x_{1},x_{2} ) \in \mathbb{R}^{2} ~:~ x_{2} =x_{1} +1  \}$.
	It is easy to see that for every $(x_{1},x_{2}) \in \mathbb{R}^{2}$,
	\begin{subequations} \label{eq:exam:LineCone}
		\begin{align}
	    \Pro_{U}(x_{1},x_{2}) &= \left(\frac{x_{1}-x_{2}+1}{2},\frac{-x_{1}+x_{2}+1}{2} \right), \label{eq:exam:LineCone:PU}\\
		\Pro_{K}(x_{1},x_{2})  &=\begin{cases}
		(x_{1},x_{2}), \quad & \text{if} ~x_{1}\geq 0 ~\text{and}~x_{2} \geq 0;\\
		(0,0), \quad & \text{if} ~x_{1}< 0 ~\text{and}~x_{2} < 0;\\
		(x_{1},0) =\Pro_{L_{1}}(x_{1},x_{2}), \quad & \text{if} ~x_{1}\geq 0 ~\text{and}~x_{2} < 0;\\
		(0,x_{2}) =\Pro_{L_{2}}(x_{1},x_{2}), \quad & \text{if} ~x_{1}< 0 ~\text{and}~x_{2} \geq 0.\\
		\end{cases} \label{eq:exam:LineCone:K}
		\end{align}
	\end{subequations}
	
	By \cite[Corollary~4.5.2]{Cegielski},
	\begin{align}  \label{exam:LineCone:Fix}
	\Fix \Pro_{U}\Pro_{K} =U \cap K=\Fix \Pro_{K}\Pro_{U}.
	\end{align}
	
	\cref{exam:LineCone:ConeLine}:  Let $(x_{1},x_{2}) \in \mathbb{R}^{2} \smallsetminus (K \cup \mathbb{R}^{2}_{--})$
	 such that $x_{1}-1 < x_{2} <x_{1} +1$, that is, $(x_{1},x_{2}) $ is above $l_{1}$ and below $l_{2}$ but neither in $K$ nor in the strictly negative orthant. Then by \cref{eq:exam:LineCone},
	\begin{align*}
	\Pro_{U \cap K}(x_{1},x_{2}) =\Pro_{U }(x_{1},x_{2})  \neq \Pro_{U \cap K}\Pro_{U}\Pro_{K}(x_{1},x_{2}),
	\end{align*}
	which, by \cref{def:BAM} and \cref{exam:LineCone:Fix}, implies  that $\Pro_{U}\Pro_{K}$ is not a BAM.	
	
	\cref{exam:LineCone:LineCone}: By \cref{defn:FredrichAngleClassical}, the cosine of Friedrichs angles between $\pa U$ and $L_{1}$ and between $\pa U$ and $L_{2}$ is
	\begin{align} \label{eq:exam:LineCone:cF}
	c(\pa U, L_{1})= \Innp{\frac{(1,-1)}{ \sqrt{2}} ,(1,0)} =\frac{\sqrt{2}}{2} = \Innp{\frac{(1,-1)}{ \sqrt{2}} ,(0,1)} = c(\pa U, L_{2})
	\end{align}
	Let $(x_{1},x_{2}) \in \mathbb{R}^{2}$. If $(x_{1},x_{2}) \in    \{ (y_{1},y_{2})  \in \mathbb{R}^{2} ~:~ y_{1}-1 \leq y_{2} \leq y_{1} +1 \} $, then $ \Pro_{U \cap K} (x_{1},x_{2}) =\Pro_{U} (x_{1},x_{2}) $, which yields that
	\begin{align*}
	\Pro_{K}  \Pro_{U} (x_{1},x_{2})  =\Pro_{U \cap K} (x_{1},x_{2}) \quad \text{and} \quad \Pro_{U \cap K} \Pro_{K}  \Pro_{U} (x_{1},x_{2})  =\Pro_{U \cap K} (x_{1},x_{2}).
	\end{align*}
	Assume that $x_{2} < x_{1} -1$. Then
	\begin{subequations}
		\begin{align} \label{eq:exam:LineCone:PPPPP}
		\Pro_{K}\Pro_{U}(x_{1},x_{2}) &=\Pro_{L_{1}}\Pro_{U}(x_{1},x_{2}), \\
		 \Pro_{U \cap K}\Pro_{K}\Pro_{U}(x_{1},x_{2}) &=(1,0) =\Pro_{U \cap K}(x_{1},x_{2})=\Pro_{U \cap L_{1}}(x_{1},x_{2}).
		\end{align}
	\end{subequations}
	
	Moreover, because $U$ and $L_{1}$ are closed affine subspaces with $ U \cap L_{1} \neq \varnothing$,
		\begin{align*}
	\norm{\Pro_{K}\Pro_{U}(x_{1},x_{2}) -  \Pro_{U \cap K}(x_{1},x_{2})}  &= \norm{\Pro_{L_{1}}\Pro_{U}(x_{1},x_{2}) -  \Pro_{U \cap L_{1}}(x_{1},x_{2})}  \quad (\text{by \cref{eq:exam:LineCone:PPPPP}})\\
	& \leq  \frac{\sqrt{2}}{2}  \norm{(x_{1},x_{2})-  \Pro_{U \cap L_{1}}(x_{1},x_{2})} \quad (\text{by \cref{rema:comp:BAM}\cref{rema:comp:BAM:ineq} and \cref{eq:exam:LineCone:cF}})\\
	&=  \frac{\sqrt{2}}{2}  \norm{(x_{1},x_{2})-  \Pro_{U \cap K}(x_{1},x_{2})}. \quad (\text{by \cref{eq:exam:LineCone:PPPPP}})
		\end{align*}

	Assume that $x_{2} > x_{1} +1$. Then similarly to the case that $x_{2} < x_{1} -1$, we also have  that $\norm{\Pro_{K}\Pro_{U}(x_{1},x_{2}) -  \Pro_{U \cap K}(x_{1},x_{2})}   \leq \frac{\sqrt{2}}{2}  \norm{(x_{1},x_{2})-  \Pro_{U \cap K}(x_{1},x_{2})}$.
	
	Altogether, for every $(x_{1},x_{2}) \in \mathbb{R}^{2}$, we have that
	\begin{subequations}
		\begin{align*}
		\Pro_{U \cap K}\Pro_{K}\Pro_{U}(x_{1},x_{2})  &= \Pro_{U \cap K}(x_{1},x_{2}),\\
		\norm{\Pro_{K}\Pro_{U}(x_{1},x_{2}) -  \Pro_{U \cap K}(x_{1},x_{2})}   &\leq \frac{\sqrt{2}}{2}  \norm{(x_{1},x_{2})-  \Pro_{U \cap K}(x_{1},x_{2})},
		\end{align*}
	\end{subequations}
	which combining with \cref{exam:LineCone:Fix} yield that $\Pro_{K}\Pro_{U}$ is a $ \frac{\sqrt{2}}{2}  $-BAM.
\end{proof}

\begin{figure}[H]
	\begin{center} \includegraphics[scale=3.0]{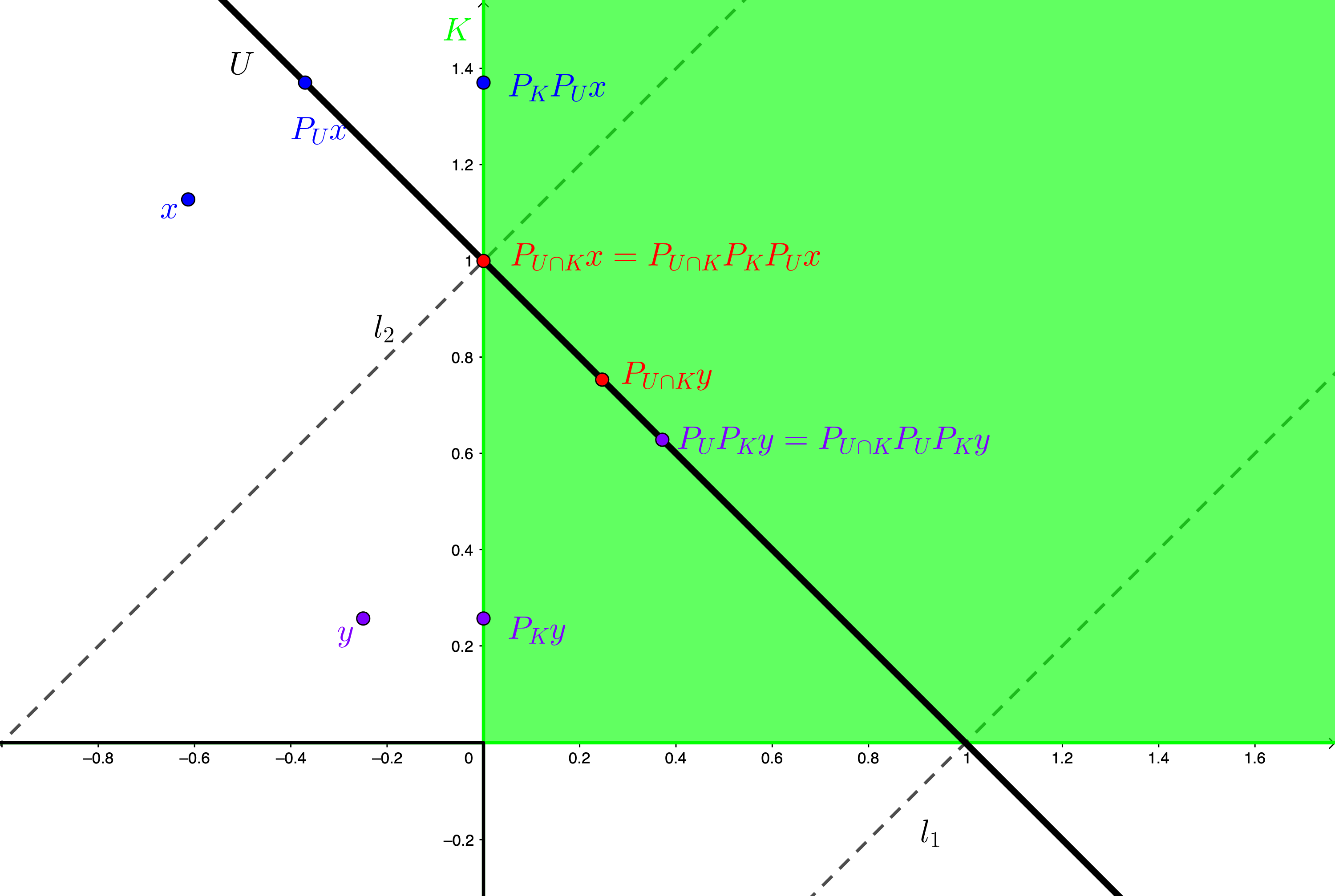}
	\end{center}
	\caption{Composition of projections onto line and cone} \label{fig:examp:BAM:LineCone}
\end{figure}

\begin{remark} \label{remark:HW}
By \cref{prop:coneball} and \cref{exam:LineCone}, we know that in \cref{thm:BAM:COMPO}, the assumption \enquote{$(\forall i \in \I)$ $\Fix G_{i}$ is closed affine subspaces} is not tight, and that the order of the operators matters.
\end{remark}

The following example examines the composition of projections onto balls and  states that generally the composition of BAMs is not a BAM again.
\begin{example} \label{examp:BAM:counterexamp:Balls}
	Suppose that $\mathcal{H} = \mathbb{R}^{2}$.  Consider the two closed balls $K_{1}= \{(x_{1},x_{2}) ~:~  (x_{1}+1)^{2} +x^{2}_{2} \leq 4 \}$ and let $K_{2} = \{(x_{1},x_{2}) ~:~ (x_{1}-1)^{2} +x^{2}_{2} \leq 4  \}$. Then the following statements hold (see also \cref{fig:examp:BAM:counterexamp:Balls}):
	\begin{enumerate}
		\item \label{examp:BAM:counterexamp:Balls:neq} For every $x \in \{ (x_{1},x_{2}) \in \mathbb{R}^{2}\smallsetminus (K_{1} \cup K_{2}) ~:~ x_{1} <0 ~\text{and}~ x_{2} \neq 0 \}$, $\Pro_{K_{1}\cap K_{2}}\Pro_{K_{2}}\Pro_{K_{1}}x = \Pro_{K_{2}}\Pro_{K_{1}}x \neq \Pro_{K_{1}\cap K_{2}}x$.
		\item \label{examp:BAM:counterexamp:Balls:notBAM} $\Pro_{K_{2}}\Pro_{K_{1}}$ is not a BAM.
	\end{enumerate}
	
\end{example}

\begin{proof}
	By \cref{examp:BAM:Pro} and  \cref{prop:CompositionBAM}\cref{prop:CompositionBAM:comp}, $\Fix \Pro_{K_{2}}\Pro_{K_{1}} =K_{1} \cap K_{2}$. The proof follows by \cref{def:BAM},  the formula shown in \cite[Example~3.18]{BC2017} and some elementary algebraic manipulations.
\end{proof}

\begin{figure}[H]
	\begin{center} \includegraphics[scale=1.0]{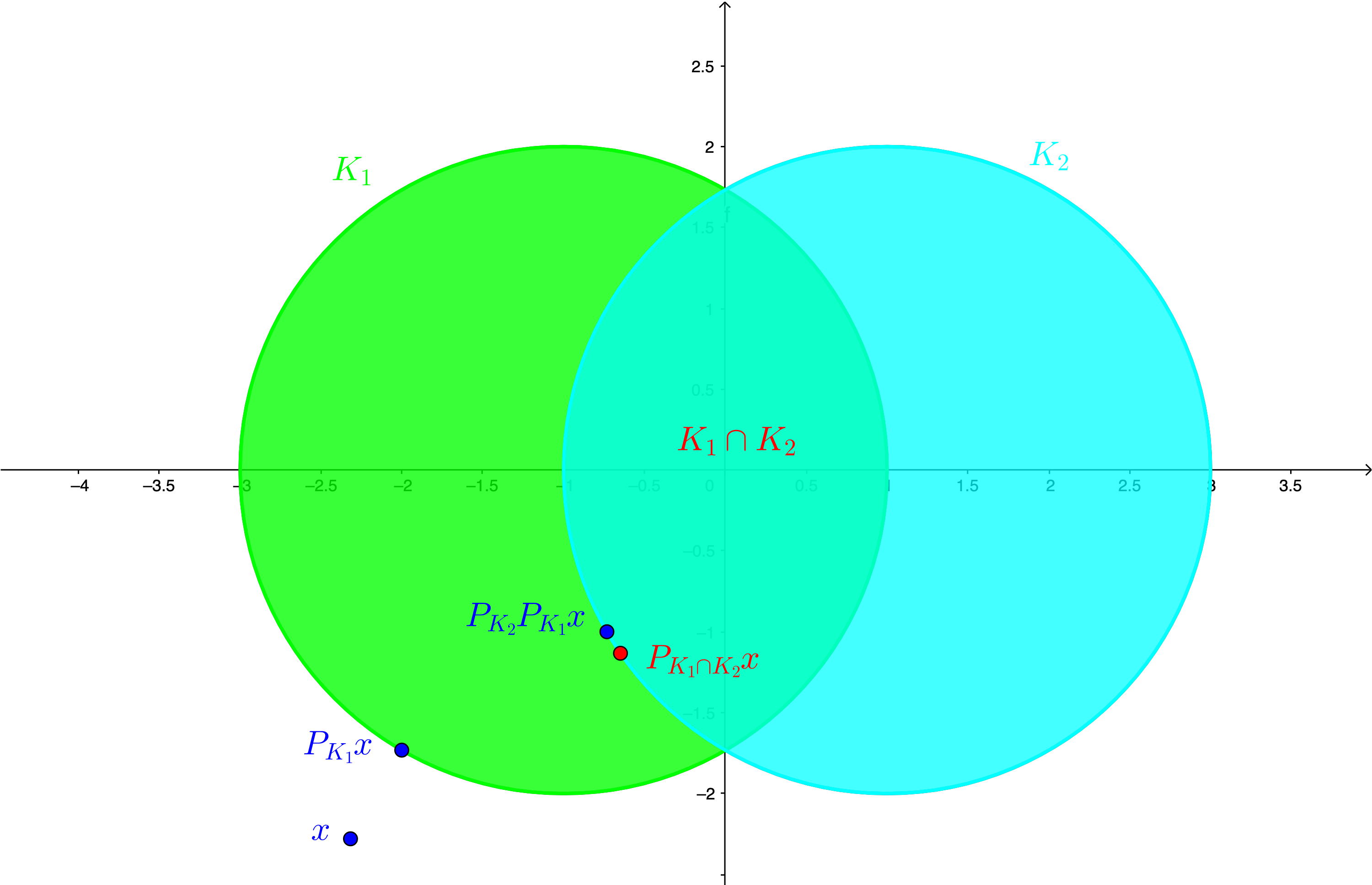}
	\end{center}
	\caption{Composition of BAMs may not be a BAM} \label{fig:examp:BAM:counterexamp:Balls}
\end{figure}

%%%%%%%%%%%%%%%%%%%%%%%%%%%%%%%%%%%%%%%%%%%%%%%%%%%%%%%%%%%%%%%%%%%%%%%%%%%%%%\section{Convex combinations of BAMs} %%%%%%%%%%%%%%%%%%%
%%%%%%%%%%%%%%%%%%%%%%%%%%%%%%%%%%%%%%%%%%%%%%%%%%%%%%%%%%%%%%%%%%%
\section{Combinations of BAMs} \label{sec:combinationBAM}
In this section, we consider combinations of finitely many BAMs.
In the following results, by reviewing  \cref{remark:chara:Id:P}, we obtain constraints for the coefficients constructing the combinations.
\begin{remark} \label{exam:affineconv:notaffinesets}
	Let $C$ be a nonempty closed  convex subset  of $\mathcal{H}$ and let  $\gamma \in \mathbb{R}$. Note that by \cref{examp:BAM:Pro}, $\Pro_{C}$ and $\Id=\Pro_{\mathcal{H}} $ are BAMs.
	\begin{enumerate}
		\item   Let $\gamma \in \mathbb{R}$. By \cref{remark:chara:Id:P}\cref{remark:chara:Id:P:i} and \cref{def:BAM}, if  $\gamma \Id +  (1 -  \gamma) \Pro_{C}$ is a BAM, then $\gamma \in \left]-1, 1\right[\,$.
		\item In addition, suppose that $\mathcal{H}=\mathbb{R}^{2}$ and $C:=\mathbf{B}[0;1]$. Then   by \cref{remark:chara:Id:P}\cref{remark:chara:Id:P:ii},  $\gamma \Id +  (1 -  \gamma ) \Pro_{C}$ is a BAM implies that $\gamma \in \left[0,1\right[\,$.
	\end{enumerate}
\end{remark}

The following results are similar to \cref{examp:BAM:counterexamp:Balls}.
	\begin{example} \label{examp:BAM:Balls}
		Suppose that $\mathcal{H} = \mathbb{R}^{2}$.  Consider the two closed balls $K_{1}:= \{(x_{1},x_{2}) ~:~  (x_{1}+1)^{2} +x^{2}_{2} \leq 4 \}$ and let $K_{2} := \{(x_{1},x_{2}) ~:~ (x_{1}-1)^{2} +x^{2}_{2} \leq 4  \}$. Let $\alpha \in \left]0, 1\right[\,$. Then the following hold:
		\begin{enumerate}
			\item \label{examp:BAM:Balls:neq} For every $x \in \{ (x_{1},x_{2}) \in \mathbb{R}^{2}\smallsetminus (K_{1} \cup K_{2}) ~:~ x_{1} <0 ~\text{and}~ x_{2} \neq 0 \}$, $\Pro_{K_{1}\cap K_{2}} ( \alpha \Pro_{K_{1}} +(1-\alpha )\Pro_{K_{2}}  )x  \neq \Pro_{K_{1}\cap K_{2}}x$.
			\item \label{examp:BAM:Balls:notBAM} $\Pro_{K_{2}}\Pro_{K_{1}}$ is not a BAM.
		\end{enumerate}
	\end{example}

\begin{proof}
	Note that by \cref{examp:BAM:Pro} and  \cref{prop:CompositionBAM}\cref{prop:CompositionBAM:comb}, $\Fix (\alpha \Pro_{K_{2}} +(1-\alpha )\Pro_{K_{2}} )=K_{1} \cap K_{2}$. The remaining part of the proof is similar to the proof of \cref{examp:BAM:counterexamp:Balls}.
\end{proof}

In the remaining part of this section, we consider  convex combinations of BAMs with closed and affine fixed point sets.

\begin{lemma} \label{lem:convcomb:BAM:eq}
	Set $\I :=\{1, \ldots,m\}$. Let  $(\forall i \in \I)$  $G_{i}$ be a BAM such that  $\Fix G_{i}$ is a closed affine subspace of $\mathcal{H}$, and $\cap_{i \in \I} \Fix G_{i} \neq \varnothing$. Let $(\forall i \in \I)$ $\omega_{i} \in \left]0, 1\right[\,$ such that $\sum_{i \in \I} \omega_{i} =1$. Set $ G:= \sum_{i \in \I} \omega_{i} G_{i} $.
	Then
	\begin{enumerate}
		\item \label{lem:convcomb:BAM:eq:Fix} $\Fix G= \cap_{i \in \I} \Fix G_{i}$ is a closed affine subspace of $\mathcal{H}$.
		\item \label{lem:convcomb:BAM:eq:EQ} $\Pro_{ \Fix G} G=\Pro_{ \Fix G}$.
	\end{enumerate}
\end{lemma}

\begin{proof}
 \cref{lem:convcomb:BAM:eq:Fix}: By \cref{prop:CompositionBAM}\cref{prop:CompositionBAM:comb} and the assumptions,   $\Fix G= \Fix (\sum_{i \in \I} \omega_{i} G_{i})= \cap_{i \in \I} \Fix G_{i}$ is a closed affine subspace of $\mathcal{H}$.

 \cref{lem:convcomb:BAM:eq:EQ}: By \cref{lem:convcomb:BAM:eq:Fix}  and  \cite[Proposition~29.14(i)]{BC2017}, we know that $\Pro_{ \Fix G} $ is affine. Note that $(\forall j \in \I)$ $\Fix G = \cap_{i \in \I} \Fix G_{i} \subseteq \Fix G_{j}$ and that  both $ \cap_{i \in \I} \Fix G_{i}$ and $\Fix G_{j}$ are closed affine subspaces. Moreover, by \cref{lem:ExchangeProj} and  \cref{def:BAM}\cref{def:BAM:Fix},  we have that
 \begin{align} \label{eq:lem:convcomb:BAM:eq:EQ}
 (\forall j \in \I) \quad \Pro_{ \Fix G}G_{j} =  \Pro_{ \cap_{i \in \I} \Fix G_{i} }G_{j} =\Pro_{ \cap_{i \in \I} \Fix G_{i} } \Pro_{  \Fix G_{j} } G_{j} = \Pro_{ \cap_{i \in \I} \Fix G_{i} } \Pro_{  \Fix G_{j} }= \Pro_{ \cap_{i \in \I} \Fix G_{i} } = \Pro_{ \Fix G}
 \end{align}
Therefore,
\begin{align*}
	\Pro_{ \Fix G} G =\Pro_{ \Fix G} \left(\sum_{i \in \I} \omega_{i} G_{i} \right) =\sum_{i \in \I} \omega_{i} \Pro_{ \Fix G}G_{i}  \stackrel{\cref{eq:lem:convcomb:BAM:eq:EQ}}{=} \sum_{i \in \I} \omega_{i} \Pro_{ \Fix G}= \Pro_{ \Fix G}.
\end{align*}
\end{proof}
\subsection*{Convex combination of BAMs with closed and affine fixed point sets}

\begin{theorem} \label{theorem:convex:comb:BAM:2}
	 Set $\I:=\{1,2\} $. Let $(\forall i \in \I)$ $\gamma_{i} \in  \left[ 0,1\right[\,$ and  let $G_{i}  : \mathcal{H} \to \mathcal{H}$ be a $\gamma_{i}$-BAM. Suppose that $(\forall i \in \I)$ $\Fix G_{i}$ is a closed linear subspace of $\mathcal{H}$. Suppose that $\Fix G_{1}+\Fix G_{2}$ is closed.  Let $\alpha \in \left] 0,1\right[\,$. Set  $c_{F}:=c( \Fix G_{1}, \Fix G_{2})$ and
	\begin{align} \label{eq:theorem:convex:comb:BAM:2:gamma}
	\gamma:=\max \left\{ \alpha \sqrt{\gamma^{2}_{1} + ( 1-\gamma^{2}_{1}) \frac{1+c_{F}}{2} } + (1-\alpha) , \alpha + (1-\alpha)\sqrt{\gamma^{2}_{2} + ( 1-\gamma^{2}_{2}) \frac{1+c_{F}}{2} } \right\}.
	\end{align}
	Then $\max \big\{  \alpha \sqrt{\frac{1}{2}( 1+\gamma^{2}_{1}) } +(1-\alpha),    \alpha+(1-\alpha)  \sqrt{\frac{1}{2}( 1+\gamma^{2}_{2}) } \big\} \leq \gamma  < 1$ and $\alpha G_{1} +(1- \alpha ) G_{2}  $ is a $\gamma$-BAM.
\end{theorem}

\begin{proof}
	Set $(\forall i \in \I)$ $U_{i}:=\Fix G_{i}$. Because $U_{1}+ U_{2}$ is closed,  by \cref{fac:cFLess1}, we know that $c_{F} :=c(U_{1},U_{2}) \in \left[0,1\right[\,$, which yields that  $\gamma <1$ and that $(\forall i \in \I)$ $\gamma^{2}_{i}  +( 1-\gamma^{2}_{i}) \frac{1+c_{F}}{2}  \geq  \frac{1}{2}( 1+\gamma^{2}_{i})  $. Hence, $\gamma \geq \big\{  \alpha \sqrt{\frac{1}{2}( 1+\gamma^{2}_{1}) } +(1-\alpha),    \alpha+(1-\alpha)  \sqrt{\frac{1}{2}( 1+\gamma^{2}_{2}) } \big\}$.

	Let $x \in \mathcal{H}$.
	By \cref{lem:convcomb:BAM:eq} and \cref{def:BAM}, it suffices to show that
	\begin{align} \label{eq:theorem:convex:comb:BAM:2}
	\norm{\alpha G_{1}x +(1- \alpha ) G_{2}x -\Pro_{U_{1}\cap U_{2}}x} \leq \gamma \norm{x - \Pro_{U_{1}\cap U_{2}}x}.
	\end{align}
	
	Because $(\forall i \in \I)$ $G_{i}$ is a $\gamma_{i}$-BAM,  by  \cref{def:BAM}  and  \cref{lemma:CFixG}\cref{lemma:CFixG:GC},  we have that
	\begin{align} \label{eq:theorem:convex:comb:BAM:2:ProG}
	(\forall i \in \I) \quad \Pro_{U_{i}} G_{i} = \Pro_{U_{i}} = G_{i}\Pro_{U_{i}},
	\end{align}
	and that
	\begin{align} \label{eq:theorem:convex:comb:BAM:2:Ineq}
	(\forall y \in \mathcal{H}) \quad \norm{G_{i}y -\Pro_{U_{i}} y} \leq \gamma_{i} \norm{y - \Pro_{U_{i}}  y}.
	\end{align}
	
	 If $x = \Pro_{U_{1}\cap U_{2}}x$, then $x \in U_{1} \cap U_{2}$ and $\alpha G_{1}x +(1- \alpha ) G_{2}x -\Pro_{U_{1}\cap U_{2}}x =\alpha G_{1} \Pro_{U_{1}} x +(1- \alpha ) G_{2}\Pro_{U_{2}} x -x \stackrel{ \cref{eq:theorem:convex:comb:BAM:2:ProG}}{=} \alpha \Pro_{U_{1}} x +(1- \alpha ) \Pro_{U_{2}} x -x =x-x=0$, from which we  deduce that \cref{eq:theorem:convex:comb:BAM:2} holds.
	 Therefore, in the rest of the proof, we assume that $x \neq \Pro_{U_{1}\cap U_{2}}x $. Set
	 \begin{align} \label{eq:theorem:convex:comb:BAM:2:beta}
	 \beta_{1}:= \frac{ \norm{\Pro_{U_{1}} x-  \Pro_{U_{1}\cap U_{2}}x } }{\norm{x -\Pro_{U_{1}\cap U_{2}}x}} \quad \text{and} \quad  \beta_{2}:= \frac{ \norm{\Pro_{U_{2}} x-  \Pro_{U_{1}\cap U_{2}}x } }{\norm{x -\Pro_{U_{1}\cap U_{2}}x}}.
	 \end{align}
	By the triangle inequality,
	\begin{subequations} \label{eq:theorem:convex:comb:BAM:2:p1}
			\begin{align}
		&\norm{\alpha G_{1}x +(1- \alpha ) G_{2}x -\Pro_{U_{1}\cap U_{2}}x}^{2} \\
		\leq & \alpha^{2} \norm{ G_{1}x- \Pro_{U_{1}\cap U_{2}}x}^{2} +(1-\alpha)^{2}  \norm{ G_{2}x- \Pro_{U_{1}\cap U_{2}}x}^{2} +2\alpha (1-\alpha) \norm{ G_{1}x- \Pro_{U_{1}\cap U_{2}}x}\norm{ G_{2}x- \Pro_{U_{1}\cap U_{2}}x}.
		\end{align}
	\end{subequations}
Note that $(\forall i \in \I)$,
\begin{subequations}
	\begin{align}
	&G_{i}x-\Pro_{U_{i}}x \stackrel{\cref{eq:theorem:convex:comb:BAM:2:ProG}}{=} G_{i}x-\Pro_{U_{i}}G_{i}x =\Pro_{U_{i}^{\perp}}G_{i}x \in U_{i}^{\perp}, \label{eq:theorem:convex:comb:BAM:2:GP}\\
	&\Pro_{U_{i}}x  -  \Pro_{U_{1}\cap U_{2}}x = \Pro_{U_{i}}x  -  \Pro_{U_{1}\cap U_{2}}\Pro_{U_{i}}x =\Pro_{(U_{1}\cap U_{2})^{\perp}}\Pro_{U_{i}}x \in U_{i} \cap (U_{1}\cap U_{2})^{\perp} \label{eq:theorem:convex:comb:BAM:2:PPu1u2}.
	\end{align}
\end{subequations}
Now, using \cref{eq:theorem:convex:comb:BAM:2:GP} and \cref{eq:theorem:convex:comb:BAM:2:PPu1u2} in the following \cref{eq:theorem:convex:comb:BAM:2GP} and \cref{eq:theorem:convex:comb:BAM:xP}, we know that $(\forall i \in \I)$,
\begin{subequations} \label{eq:theorem:convex:comb:BAM:2:p2}
	\begin{align}
	\norm{ G_{i}x- \Pro_{U_{1}\cap U_{2}}x}^{2} &~=\norm{G_{i}x-\Pro_{U_{i}}x }^{2} + \norm{\Pro_{U_{i}}x  -  \Pro_{U_{1}\cap U_{2}}x }^{2} \label{eq:theorem:convex:comb:BAM:2GP} \\
	& \stackrel{\cref{eq:theorem:convex:comb:BAM:2:Ineq}  }{\leq} \gamma^{2}_{i}  \norm{x-\Pro_{U_{i}}x }^{2} + \norm{\Pro_{U_{i}}x  -  \Pro_{U_{1}\cap U_{2}}x }^{2}\\
	&~= \gamma^{2}_{i}  \norm{x-\Pro_{U_{i}}x }^{2} +  \gamma^{2}_{i} \norm{\Pro_{U_{i}}x  -  \Pro_{U_{1}\cap U_{2}}x }^{2} +(1-\gamma^{2}_{i}) \norm{\Pro_{U_{i}}x  -  \Pro_{U_{1}\cap U_{2}}x }^{2}\\
	&~ = \gamma^{2}_{i}  \norm{x-\Pro_{U_{1}\cap U_{2}}x }^{2} +(1-\gamma^{2}_{i}) \norm{\Pro_{U_{i}}x  -  \Pro_{U_{1}\cap U_{2}}x }^{2} \label{eq:theorem:convex:comb:BAM:xP}\\
	&\stackrel{\cref{eq:theorem:convex:comb:BAM:2:beta}}{=} \gamma^{2}_{i}  \norm{x-\Pro_{U_{1}\cap U_{2}}x }^{2} +(1-\gamma^{2}_{i}) \beta^{2}_{i}\norm{x  -  \Pro_{U_{1}\cap U_{2}}x }^{2}\\
	&~= \big( \gamma^{2}_{i}  + (1-\gamma^{2}_{i}) \beta^{2}_{i} \big)\norm{x  -  \Pro_{U_{1}\cap U_{2}}x }^{2}.
	\end{align}
\end{subequations}
Set
\begin{align} \label{eq:theorem:convex:comb:BAM:2:eta}
(\forall i \in \I) \quad \eta_{i} :=\sqrt{ \gamma^{2}_{i}  + (1-\gamma^{2}_{i}) \beta^{2}_{i} }.
\end{align}
Combine \cref{eq:theorem:convex:comb:BAM:2:p1} with \cref{eq:theorem:convex:comb:BAM:2:p2} to obtain that
\begin{subequations} \label{eq:theorem:convex:comb:BAM:2:p3}
	\begin{align}
	\norm{\alpha G_{1}x +(1- \alpha ) G_{2}x -\Pro_{U_{1}\cap U_{2}}x}^{2}  &\leq \left(\alpha^{2}  \eta^{2}_{1} + (1-\alpha)^{2}\eta^{2}_{2}  +2\alpha (1-\alpha)  \eta_{1}\eta_{2}  \right) \norm{x  -  \Pro_{U_{1}\cap U_{2}}x }^{2}\\
	&=( \alpha \eta_{1} +(1-\alpha)  \eta_{2} )^{2}\norm{x  -  \Pro_{U_{1}\cap U_{2}}x }^{2}.
	\end{align}
\end{subequations}
Combining \cref{eq:theorem:convex:comb:BAM:2}, \cref{eq:theorem:convex:comb:BAM:2:gamma}, \cref{eq:theorem:convex:comb:BAM:2:eta} and \cref{eq:theorem:convex:comb:BAM:2:p3}, we know that it remains to show that
\begin{align} \label{eq:theorem:convex:comb:BAM:2:goal}
\min \{\beta_{1} ,  \beta_{2}  \} \leq \sqrt{\frac{1+c_{F}}{2} }.
\end{align}
Note that  by \cref{eq:theorem:convex:comb:BAM:2:beta}, if there exists $i \in \I$ such that $\Pro_{U_{i}}x-\Pro_{U_{1}\cap U_{2}}x =0$, then  $\beta_{i}=0$ and \cref{eq:theorem:convex:comb:BAM:2:goal} is true. Hence, we assume $(\forall i \in \I)$ $\Pro_{U_{i}}x-\Pro_{U_{1}\cap U_{2}}x \neq 0$ from now on.

Let $i \in \I$. Because $ \Pro_{U_{i}}x - \Pro_{U_{1}\cap U_{2}}x \in U_{i}$ and $x-\Pro_{U_{i}}x =\Pro_{U_{i}^{\perp}}x\in  U_{i}^{\perp}$, we have $\innp{ \Pro_{U_{i}}x - \Pro_{U_{1}\cap U_{2}}x, x-\Pro_{U_{i}}x }=0$. Hence
\begin{align*}
\innp{ \Pro_{U_{i}}x - \Pro_{U_{1}\cap U_{2}}x, x-\Pro_{U_{1}\cap U_{2}}x } &=  \innp{ \Pro_{U_{i}}x - \Pro_{U_{1}\cap U_{2}}x, x-\Pro_{U_{i}}x } +\innp{ \Pro_{U_{i}}x - \Pro_{U_{1}\cap U_{2}}x, \Pro_{U_{i}}x-\Pro_{U_{1}\cap U_{2}}x }\\
&= \norm{\Pro_{U_{i}}x-\Pro_{U_{1}\cap U_{2}}x }^{2}
\end{align*}
and thus
\begin{align} \label{eq:theorem:convex:comb:BAM:2:beta:again}
\beta_{i} =\frac{ \norm{\Pro_{U_{i}} x-  \Pro_{U_{1}\cap U_{2}}x } }{\norm{x -\Pro_{U_{1}\cap U_{2}}x}}
=\Innp{ \frac{\Pro_{U_{i}}x - \Pro_{U_{1}\cap U_{2}}x}{\norm{\Pro_{U_{i}}x - \Pro_{U_{1}\cap U_{2}}x}} ,  \frac{x-\Pro_{U_{1}\cap U_{2}}x   }{\norm{x-\Pro_{U_{1}\cap U_{2}}x}}}.
\end{align}
Set $u:=\frac{\Pro_{U_{1}}x - \Pro_{U_{1}\cap U_{2}}x}{\norm{\Pro_{U_{1}}x - \Pro_{U_{1}\cap U_{2}}x}} $, $v:=\frac{\Pro_{U_{2}}x - \Pro_{U_{1}\cap U_{2}}x}{\norm{\Pro_{U_{2}}x - \Pro_{U_{1}\cap U_{2}}x}} $ and $ w :=\frac{x-\Pro_{U_{1}\cap U_{2}}x   }{\norm{x-\Pro_{U_{1}\cap U_{2}}x} }$. By \cref{eq:theorem:convex:comb:BAM:2:PPu1u2},  $\Pro_{U_{1}} x-  \Pro_{U_{1}\cap U_{2}}x \in U_{1} \cap ( U_{1}\cap U_{2})^{\perp}$ and $\Pro_{U_{2}} x-  \Pro_{U_{1}\cap U_{2}}x \in U_{2} \cap ( U_{1}\cap U_{2})^{\perp}$. Hence,  by
	\cref{defn:FredrichAngleClassical},
	\begin{align}  \label{eq:theorem:convex:comb:BAM:2:uv}
	\innp{u,v} \leq c_{F}.
	\end{align}
Using \cref{eq:theorem:convex:comb:BAM:2:beta:again}, the Cauchy-Schwarz inequality,  and $\norm{u}=\norm{v}=\norm{w}=1$, we deduce that
	\begin{align} \label{eq:theorem:convex:comb:BAM:2:sum:beta}
\beta_{1} +\beta_{2} =\innp{u+v,w} \leq \norm{u+v}	=\sqrt{ \norm{u}^{2}+2\innp{u,v} +\norm{v}^{2}} =\sqrt{2 (1+\innp{u,v} )} \stackrel{\cref{eq:theorem:convex:comb:BAM:2:uv}}{\leq} \sqrt{2 (1+c_{F} )}.
	\end{align}
	Suppose to the contrary that \cref{eq:theorem:convex:comb:BAM:2:goal} is not true, that is, $ \beta_{1} > \sqrt{\frac{1+c_{F}}{2} } $ and $ \beta_{2} > \sqrt{\frac{1+c_{F}}{2} } $.  Then
	\begin{align*}
	\beta_{1} +\beta_{2} > 2\sqrt{\frac{1+c_{F}}{2} } = \sqrt{2 (1+c_{F} )},
	\end{align*}
	which contradicts with \cref{eq:theorem:convex:comb:BAM:2:sum:beta}. Altogether, the proof is complete.
\end{proof}

The following example illustrates that the constant associated with the convex combination of BAMs provided in \cref{theorem:convex:comb:BAM:2} is not sharp.
\begin{example} \label{exam:alpha:PU}
	Let $U$ be a closed linear subspace of $\mathcal{H}$. Let $\alpha \in \left]0, 1\right[\,$.
	Then the following hold:
	\begin{enumerate}
		\item \label{exam:alpha:PU:gama1} $\alpha \Pro_{U} + (1- \alpha) \Pro_{U^{\perp}}$ is a BAM with constant $\max \{ \alpha \frac{\sqrt{2}}{2 } +(1-\alpha), (1-\alpha)\frac{\sqrt{2}}{2 } +\alpha \}$, by \cref{theorem:convex:comb:BAM:2}.
		\item \label{exam:alpha:PU:gama2} $\alpha \Pro_{U} + (1- \alpha) \Pro_{U^{\perp}}$ is a BAM with sharp constant $\max\{ \alpha , 1-\alpha \}$.
		\item \label{exam:alpha:PU:gamm} $\max \{ \alpha \frac{\sqrt{2}}{2 } +(1-\alpha), (1-\alpha)\frac{\sqrt{2}}{2 } +\alpha \} > \max\{ \alpha , 1-\alpha \}$.
	\end{enumerate}
\end{example}

\begin{proof}	
	\cref{exam:alpha:PU:gama1}: By \cref{examp:BAM:Pro}, both $\Pro_{U}$ and $\Pro_{U^{\perp}}$ are $0$-BAMs. Moreover, by \cref{defn:FredrichAngleClassical}, $c_{F}=c( U, U^{\perp}) =0$. Hence, using \cref{theorem:convex:comb:BAM:2} directly, we obtain that $\alpha \Pro_{U} + (1- \alpha) \Pro_{U^{\perp}}$ is a BAM with constant $\max \{  \frac{\sqrt{2}}{2 } \alpha+(1-\alpha), \alpha+ \frac{\sqrt{2}}{2 }(1-\alpha)  \}$.
	
	\cref{exam:alpha:PU:gama2}:
	Denote by $G:= \alpha \Pro_{U} + (1- \alpha) \Pro_{U^{\perp}}$.
	Let $x \in \mathcal{H}$ and $\gamma \in \left[0,1\right[\,$.  Using  $\innp{\Pro_{U} x, \Pro_{U^{\perp}} x }=0$, \cref{MetrProSubs8}, and $ \Pro_{ \Fix G }x= \Pro_{ \{0\} }x=0$, we obtain that
	\begin{subequations}  \label{eq:exam:alpha:PU:neq}
\begin{align}
& \norm{ Gx - \Pro_{ \Fix G }x} \leq  \gamma \norm{x - \Pro_{ \Fix G }x}\\
\Leftrightarrow &~\norm{\alpha  \Pro_{U} x + (1-\alpha) \Pro_{U^{\perp}} x}^{2} \leq   \gamma^{2} \norm{x}^{2} \\
\Leftrightarrow &~ \alpha^{2}\norm{\Pro_{U} x}^{2} +(1-\alpha)^{2}\norm{\Pro_{U^{\perp}} x}^{2} \leq \gamma^{2} ( \norm{\Pro_{U} x}^{2} +\norm{\Pro_{U^{\perp}} x}^{2} ) \\
\Leftrightarrow &~ 0 \leq (\gamma^{2} -\alpha^{2}) \norm{\Pro_{U} x}^{2} +   (\gamma^{2} -(1-\alpha)^{2}) \norm{\Pro_{U^{\perp}} x}^{2},
\end{align}
	\end{subequations}
which implies that 	$\gamma \geq \max\{ \alpha , 1-\alpha \}$, since $x \in \mathcal{H}$ is arbitrary.
	Therefore, the required result follows from \cref{lem:convcomb:BAM:eq},  \cref{eq:exam:alpha:PU:neq} and \cref{def:BAM}.
	
		\cref{exam:alpha:PU:gamm}: This is trivial from $\alpha \in \left]0, 1\right[\,$ and $\frac{\sqrt{2}}{2 } \in \left]0, 1\right[\,$.
\end{proof}

\begin{theorem} \label{theor:averacomb:BAM:shift}
Set $\I:=\{1,\ldots, m \}$. Let $(\forall i \in \I)$ $\omega_{i}  \in  \left]0,1\right[\,$.	Suppose that $m\geq 2$ and that $(\forall i \in \I)$  $G_{i}$ is a BAM with  $U_{i}:=\Fix G_{i}$ being a closed  affine subspace of $\mathcal{H}$ such that $\cap_{i \in \I} \Fix G_{i} \neq \varnothing$. Suppose  that $(\forall i \in \I )$  $\sum^{i}_{j=1}   (\pa U_{j} )^{\perp}$ is closed.  Then $\sum_{i \in \I} \omega_{i} G_{i}$ is a BAM.
\end{theorem}

\begin{proof}
	Let $z \in \cap_{i \in \I} U_{i} $. Define $(\forall i \in \I)$ $F_{i} :\mathcal{H} \to \mathcal{H}$ by
	\begin{align} \label{eq:theor:averacomb:BAM:shift:F}
	(\forall x \in \mathcal{H}) \quad 	F_{i} (x) := G_{i}(x +z) -z
	\end{align}
	By  the assumptions, \cref{eq:theor:averacomb:BAM:shift:F} and \cref{prop:BAMAffinLinea}, $F_{i}$ is a BAM with $\Fix F_{i} =\pa U_{i}$ being a closed linear subspace of $\mathcal{H}$. By \cref{prop:CompositionBAM}\cref{prop:CompositionBAM:comb} and by assumptions,  $\Fix (\sum_{i \in \I} \omega_{i} G_{i})= \cap^{m}_{i=1} U_{i}$ is a  closed affine subspace. Hence,   by \cref{eq:theor:averacomb:BAM:shift:F} and \cref{lemma:shift}\cref{lemma:shift:combi},  to show $\sum_{i \in \I} \omega_{i} G_{i}$ is a BAM,  we are able to assume that $U_{1}, \ldots, U_{m}$ are closed linear subspaces of $\mathcal{H}$.
	
 We prove it by induction on $m$. By \cref{lemma:perp:equiv}\cref{lemma:perp:equiv:EQ} and \cref{theorem:convex:comb:BAM:2}, we know that the base case in which $m=2$ holds. Suppose  that $m \geq 3$ and that the required result holds for $m-1$, that is,
 for any $\{ \alpha_{1}, \ldots, \alpha_{m-1}\} \subseteq \left]0,1\right[\,$ we have that  if $(\forall i \in \{1, \ldots, m-1\})$  $\sum^{i}_{j=1}   U^{\perp}_{j}$ is closed, then $\sum^{m-1}_{i =1} \alpha_{i} G_{i}$ is a BAM.
Note that
 \begin{align*}
 \sum^{m}_{i =1} \omega_{i} G_{i} = \big(\sum^{m-1}_{j =1} \omega_{j} \big)\left( \sum^{m-1}_{i =1} \frac{\omega_{i} }{\sum^{m-1}_{t =1} \omega_{t}} G_{i} \right) + \omega_{m} G_{m+1}.
 \end{align*}
Because we have the assumption, $(\forall i \in \{1, \ldots,m\})$  $ \sum^{i}_{j=1}   U^{\perp} _{j}$ is closed, by the inductive hypothesis,  $\sum^{m-1}_{i =1} \frac{\omega_{i} }{\sum^{m-1}_{j =1} \omega_{j}} G_{i}$ is a BAM.
By the assumption, $(\forall i \in \I)$ $\sum^{i}_{j=1}   U^{\perp}_{j}$ is closed,
 by  \cref{prop:CompositionBAM}\cref{prop:CompositionBAM:comb} and \cref{lemma:perp:equiv}\cref{lemma:perp:equiv:EQ}, we know that
 $\Fix \left( \sum^{m-1}_{i =1} \frac{\omega_{i} }{\sum^{m-1}_{t =1} \omega_{t}} G_{i} \right) +\Fix G_{m} = (\cap^{m-1}_{i=1} U_{i}) +U_{m} $ is closed.
Hence, apply \cref{theorem:convex:comb:BAM:2} with $G_{1}= \sum^{m-1}_{i =1} \frac{\omega_{i} }{\sum^{m-1}_{j =1} \omega_{j}} G_{i}$, $G_{2} =G_{m}$, $\alpha = \sum^{m-1}_{j =1} \omega_{j} $ to obtain that $\sum^{m}_{i =1} \omega_{i} G_{i}$ is a BAM.
\end{proof}

\subsection*{New method  using the Cartesian product space reformulation}
The main result \cref{theor:averacomb:BAM} in this subsection is almost the same with the \cref{theor:averacomb:BAM:shift} proved in the previous subsection, however, in this subsection, we use a Cartesian product space reformulation.

In the whole subsection, set $\I :=\{1, \ldots, m\}$. Let $(\omega_{i})_{i \in \I}$ be real numbers in $\left]0,1\right] $ such that $\sum_{i \in \I} \omega_{i}=1$.
Let $\mathcal{H}^{m}$ be the real Hilbert space obtained by endowing the Cartesian product $\times_{i \in \I} \mathcal{H}$ with the usual vector space structure and with the weighted inner product
\begin{align*}
(\forall \mathbf{x} =(x_{i})_{i \in \I} \in \mathcal{H}^{m}) (\forall \mathbf{y} =(y_{i})_{i \in \I} \in \mathcal{H}^{m}) \quad \innp{\mathbf{x} , \mathbf{y} } =\sum_{i \in \I} \omega_{i} \innp{x_{i}, y_{i}}.
\end{align*}

Clearly,
\begin{align} \label{eq:norm}
(\forall \mathbf{x} =(x_{i})_{i \in \I} \in \mathcal{H}^{m}) \quad \norm{\mathbf{x}}^{2} =\innp{\mathbf{x} , \mathbf{x} } =\sum_{i \in \I} \omega_{i} \innp{x_{i},x_{i}} =\sum_{i \in \I} \omega_{i} \norm{x_{i}}^{2}.
\end{align}
Denote by
\begin{align*}
\mathbf{D} :=\{ (x)_{i \in \I} \in \mathcal{H}^{m} ~:~ x \in \mathcal{H} \}.
\end{align*}
The following well-known fact is critical in proofs in this subsection.
\begin{fact} \label{fact:projector}
	Let $\mathbf{x} = (x_{i})_{i \in \I} \in \mathcal{H}^{m}$. The following hold:
	\begin{enumerate}
		\item \label{fact:projector:D} $\Pro_{\mathbf{D}  } \mathbf{x} = ( \sum_{j \in \I}\omega_{j} x_{j})_{i\in \I}$.
		\item \label{fact:projector:timesC} Let $(\forall i \in \I)$ $C_{i}$ be nonempty closed and convex subset of $\mathcal{H}$. Then $ \Pro_{\times_{i \in \I}C_{i}}  \mathbf{x}  = ( \Pro_{C_{i}}x_{i})_{i \in \I}$.
	\end{enumerate}
\end{fact}
\begin{proof}
	\cref{fact:projector:D}: This is from \cite[Poposition~29.16]{BC2017}.
	
	\cref{fact:projector:timesC}: This is similar to \cite[Proposition~29.3]{BC2017}. Because the  definition of inner product is different, we show the proof next.
	Clearly, $( \Pro_{C_{i}}x_{i})_{i \in \I} \in \times_{i \in \I}C_{i}$. Moreover, by \cite[Theorem~3.16]{BC2017},
	\begin{align*}
	(\forall (c_{i})_{i\in \I} \in  \times_{i \in \I}C_{i}) \quad  \Innp{(x_{i})_{i \in \I} - ( \Pro_{C_{i}}x_{i})_{i \in \I}, (c_{i})_{i\in \I}- ( \Pro_{C_{i}}x_{i})_{i \in \I}   } =\sum_{i \in \I} \omega_{i} \innp{x_{i} - \Pro_{C_{i}}x_{i}, c_{i} - \Pro_{C_{i}}x_{i}} \leq 0,
	\end{align*}
	which by \cite[Theorem~3.16]{BC2017} again, implies that $ \Pro_{\times_{i \in \I}C_{i}}  \mathbf{x}  = ( \Pro_{C_{i}}x_{i})_{i \in \I}$.
\end{proof}

In the remaining part of this subsection, let $(\forall i \in \I)$ $G_{i} : \mathcal{H} \to \mathcal{H}$. Define $\mathbf{F} : \mathcal{H}^{m} \to \mathcal{H}^{m}$, and $\mathbf{G} : \mathcal{H}^{m} \to \mathcal{H}^{m}$ respectively by
\begin{align}
(\forall \mathbf{x} =(x_{i})_{i \in \I} \in \mathcal{H}^{m}) \quad &\mathbf{F}(\mathbf{x}  ) = ( G_{i}x_{i})_{i \in \I},\label{defn:F}\\
(\forall \mathbf{x} =(x_{i})_{i \in \I} \in \mathcal{H}^{m}) \quad &\mathbf{G}(\mathbf{x}  ) =\big( \sum_{j \in \I} \omega_{j} G_{j}x_{j} \big)_{i \in \I}. \label{defn:G}
\end{align}

\begin{proposition} \label{prop:GiBAM:F}
	\begin{enumerate}
		\item \label{prop:GiBAM:F:FixF} $\Fix \mathbf{F} = \times_{i \in \I}\Fix G_{i} $.
		\item \label{prop:GiBAM:F:BAM} Let $(\forall i\in \I)$ $ \gamma_{i}  \in \left[0, 1\right[\,$.	Suppose that  $(\forall i \in \I)$  $G_{i}$ is a $\gamma_{i}$-BAM. Then $\mathbf{F}$ is a $(\max_{i \in \I} \{\gamma_{i}\})$-BAM.
	\end{enumerate}
\end{proposition}

\begin{proof}
	Let $\mathbf{x} =(x_{i})_{i \in \I} \in \mathcal{H}^{m}$.
	
	\cref{prop:GiBAM:F:FixF}: Now
	\begin{align*}
	\mathbf{x}  \in \Fix \mathbf{F} \stackrel{\cref{defn:F}}{\Leftrightarrow}  (x_{i})_{i \in \I}  = ( G_{i}x_{i})_{i \in \I}
	\Leftrightarrow (\forall i \in \I) x_{i}=G_{i}x_{i}
	\Leftrightarrow (\forall i \in \I) x_{i} \in \Fix G_{i}
	\Leftrightarrow \mathbf{x}  \in \times_{i \in \I}\Fix G_{i} .
	\end{align*}
	
	\cref{prop:GiBAM:F:BAM}: Because $(\forall i \in \I)$  $G_{i}$ is a BAM, we know that $(\forall i \in \I)$  $\Fix G_{i}$ is a	nonempty closed and convex subsets of $\mathcal{H}$ and, by  \cref{def:BAM},  that
	for every $i \in \I$,
	\begin{align}
	&\Pro_{\Fix G_{i}}G_{i}=\Pro_{\Fix G_{i}} \label{eq:prop:GiBAM:F:BAM:Gi:Fix}\\
	(\forall x \in \mathcal{H}) \quad &\norm{G_{i} x -\Pro_{\Fix G_{i}}x} \leq \gamma_{i} \norm{x - \Pro_{\Fix G_{i}}x}. \label{eq:prop:GiBAM:F:BAM:Gi:ineq}
	\end{align}
	By \cref{prop:GiBAM:F:FixF} and \cref{fact:projector}\cref{fact:projector:timesC}, $\Fix \mathbf{F} = \times_{i \in \I}\Fix G_{i} $ is a nonempty closed   convex subset of $\mathcal{H}^{m}$ and
	\begin{align}  \label{eq:prop:GiBAM:F:BAM:PFixG}
 \Pro_{\Fix \mathbf{F}} (\mathbf{x} ) =
	\Pro_{\times_{i \in \I}\Fix G_{i} } (\mathbf{x} ) = (\Pro_{\Fix G_{i}} x_{i} )_{i \in \I}.
	\end{align}
Now
	\begin{align} \label{eq:prop:GiBAM:F:BAM:EQ}
	\Pro_{\Fix \mathbf{F}} \mathbf{F}  (\mathbf{x} )\stackrel{\cref{defn:F}}{=}  \Pro_{\Fix \mathbf{F}} \left( ( G_{i}x_{i})_{i \in \I} \right) \stackrel{\cref{eq:prop:GiBAM:F:BAM:PFixG}}{=} (\Pro_{\Fix G_{i}} G_{i}x_{i} )_{i \in \I} \stackrel{\cref{eq:prop:GiBAM:F:BAM:Gi:Fix} }{=} (\Pro_{\Fix G_{i}} x_{i} )_{i \in \I} \stackrel{\cref{eq:prop:GiBAM:F:BAM:PFixG}}{=}  \Pro_{\Fix \mathbf{F}} (\mathbf{x}).
	\end{align}
	Note that
	by \cref{defn:F} and \cref{eq:prop:GiBAM:F:BAM:PFixG},
	\begin{subequations} \label{eq:prop:GiBAM:F:BAM:Norm}
		\begin{align}
		\norm{\mathbf{F} (\mathbf{x})  - \Pro_{\Fix \mathbf{F}}\mathbf{x}}^{2} &~=~ \norm{ ( G_{i}x_{i})_{i \in \I} - (\Pro_{\Fix G_{i}} x_{i} )_{i \in \I} }^{2}\\
		&\stackrel{\cref{eq:norm}}{=} \sum_{i\in \I} \omega_{i} \norm{G_{i}x_{i} -\Pro_{\Fix G_{i}} x_{i}}^{2}\\
		&\stackrel{\cref{eq:prop:GiBAM:F:BAM:Gi:ineq}}{\leq} \sum_{i\in \I} \omega_{i}  \gamma^{2}_{i} \norm{x_{i} -\Pro_{\Fix G_{i}} x_{i}}^{2}\\
		&~\leq~  \max_{j\in \I} \{ \gamma^{2}_{j}\}  \sum_{i\in \I} \omega_{i}  \norm{x_{i} -\Pro_{\Fix G_{i}} x_{i}}^{2}\\
		&\stackrel{\cref{eq:norm}}{=}   \max_{j\in \I} \{ \gamma^{2}_{j}\}  \Norm{(x_{i})_{i \in \I} -  (\Pro_{\Fix G_{i}} x_{i} )_{i \in \I}}^{2}\\
		&\stackrel{\cref{eq:prop:GiBAM:F:BAM:PFixG}}{=} \max_{j\in \I} \{ \gamma^{2}_{j}\}  \norm{\mathbf{x}  -  \Pro_{\Fix \mathbf{F}}\mathbf{x} }^{2}.
		\end{align}
	\end{subequations}
	Therefore, combine \cref{eq:prop:GiBAM:F:BAM:EQ} and \cref{eq:prop:GiBAM:F:BAM:Norm} with  \cref{def:BAM} to obtain the asserted result.
	
\end{proof}

\begin{proposition} \label{proposition:producF:G}
	Let $\gamma \in \left[0,1\right[\,$ . Then the following hold:
\begin{enumerate}
	\item  \label{proposition:producF:G:F:Gi} If $\mathbf{F}$ is a $\gamma$-BAM, then $(\forall i \in \I)$ $G_{i}$ is a $\gamma$-BAM.
	\item \label{proposition:producF:G:equiv} $\mathbf{F}$ is a  BAM if and only if $(\forall i \in \I)$ $G_{i}$ is a  BAM.
\end{enumerate}
\end{proposition}

\begin{proof}
	\cref{proposition:producF:G:F:Gi}:
	Because $\mathbf{F}$ is a $\gamma$-BAM, using \cref{def:BAM} and \cref{prop:GiBAM:F}\cref{prop:GiBAM:F:FixF}, we know that $\Fix \mathbf{F} = \times_{i \in \I}\Fix G_{i} $ is a nonempty closed and convex subset of $\mathcal{H}^{m}$,  and
	there exists $\gamma \in \left[0, 1\right[\,$ such that
	\begin{align}
	&\Pro_{\times_{i \in \I}  \Fix G_{i} }\mathbf{F}=\Pro_{\times_{i \in \I}  \Fix G_{i}} \label{eq:proposition:producF:G:Fix}\\
	(\forall \mathbf{x}  \in \mathcal{H}^{m}) \quad &\norm{\mathbf{F}\mathbf{x} -\Pro_{\times_{i \in \I}  \Fix G_{i}}\mathbf{x}} \leq \gamma \norm{\mathbf{x}- \Pro_{\times_{i \in \I}  \Fix G_{i}}\mathbf{x}}. \label{eq:proposition:producF:G:ineq}
	\end{align}
Hence, $(\forall i \in \I)$ $ \Fix G_{i}$ is a nonempty closed and convex subset of $\mathcal{H}$.	 Let $x \in \mathcal{H}$.

Set $\mathbf{y} :=(x)_{i \in \I} \in \mathcal{H}^{m}$.
By \cref{fact:projector}\cref{fact:projector:timesC}, \cref{defn:F}, and \cref{eq:proposition:producF:G:Fix},
	\begin{align*}
\big( \Pro_{\Fix G_{i} } ( G_{i}x)\big)_{i \in \I}=	\Pro_{\times_{i \in \I}  \Fix G_{i} } \big( ( G_{i}x)_{i \in \I}\big)=\Pro_{\times_{i \in \I}  \Fix G_{i} }\mathbf{F}\mathbf{y}=\Pro_{\times_{i \in \I}  \Fix G_{i}}\mathbf{y} =\big( \Pro_{\Fix G_{i} } ( x)\big)_{i \in \I},
	\end{align*}
	which yields $(\forall i \in \I)$ $\Pro_{\Fix G_{i} } G_{i}=\Pro_{\Fix G_{i} }  $.
	
	Let $j \in \I$. Set $\mathbf{x} :=(x_{i})_{i \in \I} \in \mathcal{H}^{m}$ such that $x_{j}=x$ and $(\forall i \in \I \smallsetminus \{j\})$ $x_{i} \in \Fix G_{i}$.   Then $(\forall i \in \I \smallsetminus \{j\})$ $x_{i} =G_{i}x_{i} =\Pro_{  \Fix G_{i} }x_{i}$.
	Hence, by \cref{eq:proposition:producF:G:ineq}, we have that
		\begin{align*}
		&~\norm{\mathbf{F}\mathbf{x} -\Pro_{\times_{i \in \I}  \Fix G_{i}}\mathbf{x}}^{2} \leq \gamma^{2}
		\norm{\mathbf{x}- \Pro_{\times_{i \in \I}  \Fix G_{i}}\mathbf{x}}^{2}\\
		\Leftrightarrow~ &~\norm{ ( G_{i} x_{i} )_{i \in \I}-(\Pro_{\Fix G_{i}}  x_{i})_{i\in \I} }^{2} \leq \gamma^{2}
		\norm{ (x_{i})_{i \in \I}  - (\Pro_{\Fix G_{i}}  x_{i})_{i\in \I}  }^{2}\\
		\stackrel{\cref{eq:norm}}{\Leftrightarrow} &~ \sum_{i \in \I} \omega_{i} \norm{G_{i}x_{i} -\Pro_{\Fix G_{i}}  x_{i} }^{2}  \leq \gamma^{2}  \sum_{i \in \I} \omega_{i} \norm{x_{i} -\Pro_{\Fix G_{i}}  x_{i} }^{2}\\
		\Leftrightarrow~ &~ \omega_{j} \norm{G_{j}x -\Pro_{\Fix G_{j}}  x}^{2}  \leq \gamma^{2} \omega_{j} \norm{x-\Pro_{\Fix G_{j}}  x }^{2}\\
		\Leftrightarrow~ &~ \norm{G_{j}x -\Pro_{\Fix G_{j}}  x }^{2}  \leq \gamma^{2}  \norm{x -\Pro_{\Fix G_{j}}  x }^{2}.
		\end{align*}
Hence, by  \cref{def:BAM}, we know that $(\forall i \in \I)$ $G_{i}$ is a $\gamma$-BAM.

	\cref{proposition:producF:G:equiv}:  The equivalence comes from \cref{proposition:producF:G:F:Gi} above  and \cref{prop:GiBAM:F}\cref{prop:GiBAM:F:BAM}.
\end{proof}

The following result is inspired by \cite[Proposition~5.25]{BOyW2019Isometry}. With consideration of \cref{prop:BAM:Properties} and \cref{examp:BAM:Pro}, we note that \cref{theor:averacomb:BAM} is a refinement of \cite[Proposition~5.25]{BOyW2019Isometry}.
\begin{theorem} \label{theor:averacomb:BAM}
Let $(\forall i\in \I)$ $ \gamma_{i}  \in \left[0, 1\right[\,$.	Suppose that  $(\forall i \in \I)$  $G_{i}$ is a $\gamma_{i}$-BAM and that $\Fix G_{i}$ is a closed affine subspace of $\mathcal{H}$ with $\cap_{j \in \I} \Fix G_{j} \neq \varnothing$. Set $c_{F}:=c\left( \mathbf{D}, \times_{j \in \I} (\pa\Fix G_{j} ) \right)$. Suppose that $ \sum_{j \in \I} (\pa \Fix G_{j} )^{\perp}$ is closed.
Denote by $\mu:=\max_{j\in \I} \{\gamma_{j} \}$ and $\gamma:=\sqrt{\mu^{2}+(1-\mu^{2}) \frac{(1+c_{F})^{2} }{4}}$. Then $\Fix \sum_{j \in \I} \omega_{j} G_{j} =\cap_{j \in \I} \Fix G_{j} $ is a closed affine subspace of $\mathcal{H}$ and $\sum_{j \in \I} \omega_{j} G_{j}$ is a $\gamma$-BAM.  Moreover,
\begin{align*}
(\forall x \in \mathcal{H}) \quad \norm{ (\sum_{i \in \I} \omega_{i} G_{i})^{k}x - \Pro_{ \cap_{i \in \I} \Fix G_{i} } x} \leq \gamma^{k} \norm{x - \Pro_{ \cap_{i \in \I} \Fix G_{i} } x}.
\end{align*}
\end{theorem}

\begin{proof}
	By the assumptions and \cref{lem:convcomb:BAM:eq},  $\Fix \sum_{i \in \I} \omega_{i} G_{i}= \cap_{i \in \I} \Fix G_{i}$ is a closed affine subspace of $\mathcal{H}$ and $\Pro_{\cap_{i \in \I} \Fix G_{i}} ( \sum_{i \in \I} \omega_{i} G_{i})= \Pro_{\cap_{i \in \I} \Fix G_{i}} $.
	To show $\sum_{j \in \I} \omega_{j} G_{j}$ is a $\gamma$-BAM, by  \cref{def:BAM}, it suffices to show that
	\begin{align} \label{eq:prop:producF:G:G}
	(\forall x \in \mathcal{H}) \quad 	\Norm{\sum_{i \in \I} \omega_{i} G_{i} x -\Pro_{\cap_{i \in \I} \Fix G_{i}} x} \leq \gamma \norm{x - \Pro_{\cap_{i \in \I} \Fix G_{i}} x}.
	\end{align}
	By \cref{prop:GiBAM:F}\cref{prop:GiBAM:F:FixF}$\&$\cref{prop:GiBAM:F:BAM}, $\Fix \mathbf{F} = \times_{j \in \I}\Fix G_{j} $ is a closed affine subspace of $\mathcal{H}^{m}$ and $\mathbf{F}$ is a $\mu$-BAM.
	By \cref{examp:BAM:Pro},  $\Pro_{\mathbf{D}}$ is a $0$-BAM.
	By \cref{fac:cFLess1}, $\mathbf{D}^{\perp} +\left( \pa \big(\times_{j \in \I}\Fix G_{j} \big) \right)^{\perp}$ is closed if and only if $\mathbf{D}+\left( \pa \big(\times_{j \in \I}\Fix G_{j} \big) \right) $ is closed. Moreover,
by  \cite[Lemma~5.18]{BB1996}, $\mathbf{D}+\left( \pa \big(\times_{j \in \I}\Fix G_{j} \big) \right) $ is closed if and only if $ \sum_{j \in \I} (\pa \Fix G_{j} )^{\perp}$ is closed, which is our assumption. Hence, we obtain that $\mathbf{D}^{\perp} +\left( \pa \big(\times_{j \in \I}\Fix G_{j} \big) \right)^{\perp}$ is closed.
	Then apply  \cref{thm:BAM:COMPO}\cref{thm:BAM:COMPO:BAM:2} with $\mathcal{H} =\mathcal{H}^{m}$, $G_{1} =\mathbf{F} $ and $G_{2}= \Pro_{\mathbf{D}} $ to obtain that $\Pro_{\mathbf{D}}\mathbf{F} $ is a $\gamma$-BAM.
Note that, by \cref{defn:F} and \cref{fact:projector}\cref{fact:projector:D},
	\begin{align*}
(\forall \mathbf{y} =(y_{i})_{i \in \I} \in \mathcal{H}^{m}) \quad &\Pro_{\mathbf{D}} \mathbf{F}(\mathbf{y}  ) =\big( \sum_{j \in \I} \omega_{j} G_{j}y_{j} \big)_{i \in \I} =\mathbf{G}(\mathbf{y}  ),	
	\end{align*}
that is,  $\Pro_{\mathbf{D}} \mathbf{F}=\mathbf{G}$. By  \cite[Corollary~4.5.2]{Cegielski},
\begin{align*}
\Fix \mathbf{G}=\Fix (\Pro_{\mathbf{D}} \mathbf{F})=\mathbf{D} \cap \left( \times_{i \in \I} (\cap_{j \in \I} \Fix G_{j})\right).
\end{align*}
Let $x \in \mathcal{H}$ and set $\mathbf{x} =(x)_{i \in \I} \in \mathcal{H}^{m}$. Similarly with the proof of \cref{fact:projector}\cref{fact:projector:timesC}, by \cite[Theorem~3.16]{BC2017}, we have that
\begin{align} \label{eq:prop:producF:G:G:fixG}
\Pro_{\Fix \mathbf{G} }\mathbf{x}=(\Pro_{\cap_{j \in \I} \Fix G_{j} }x )_{i \in \I}.
\end{align}
Because	 $\mathbf{G}=\Pro_{\mathbf{D}}\mathbf{F} $  is a $\gamma$-BAM, by \cref{def:BAM}\cref{def:BAM:Ineq},
\begin{align*}
& \norm{\mathbf{G} \mathbf{x} -\Pro_{\Fix \mathbf{G}} \mathbf{x} } \leq \gamma \norm{\mathbf{x} - \Pro_{\Fix \mathbf{G}}\mathbf{x}}\\ \stackrel{\cref{eq:prop:producF:G:G:fixG}}{\Leftrightarrow}&~
\Norm{ (\sum_{j \in \I} \omega_{j} G_{j} x)_{i \in \I} -(\Pro_{\cap_{j \in \I} \Fix G_{j} }x )_{i \in \I} }^{2} \leq \gamma^{2} \Norm{(x)_{i\in \I} - (\Pro_{\cap_{j \in \I} \Fix G_{j} }x )_{i \in \I}}^{2}\\
\stackrel{\cref{eq:norm}}{\Leftrightarrow}&~\sum_{i \in \I} \omega_{i} \Norm{ \sum_{j \in \I} \omega_{j} G_{j} x-  \Pro_{\cap_{j \in \I} \Fix G_{j} }x }^{2} \leq \gamma^{2}  \sum_{i \in \I} \omega_{i} \norm{x -  \Pro_{\cap_{j \in \I} \Fix G_{j} }x }^{2}\\
~\Leftrightarrow~&~ \Norm{ \sum_{j \in \I} \omega_{j} G_{j} x-  \Pro_{\cap_{j \in \I} \Fix G_{j} }x } \leq \gamma \norm{x -  \Pro_{\cap_{j \in \I} \Fix G_{j} }x },
\end{align*}
which yields  \cref{eq:prop:producF:G:G}. Hence, the proof is complete.
\end{proof}

\begin{remark}\label{remark:assumpt:constants:different}
	Consider \Cref{theor:averacomb:BAM:shift,theor:averacomb:BAM}. Although
	 the results from these two theorems are the same, but there are different assumptions:   \enquote{$(\forall i \in \I )$  $\sum^{i}_{j=1}   (\pa U_{j} )^{\perp}$ is closed} and \enquote{$ \sum_{i \in \I} (\pa U_{i} )^{\perp}$ is closed} respectively.

	Suppose that $m =3$, that $(\pa U_{2})^{\perp} + (\pa U_{1} )^{\perp}$ is not closed,  and that $(\pa U_{3})^{\perp} =\mathcal{H}$, say, $G_{3}=\Pro_{\{0\}}$. Then clearly,  $ \sum^{3}_{i =1} (\pa U_{i} )^{\perp} =\mathcal{H}$ is closed. Hence, \enquote{$ \sum_{i \in \I} (\pa U_{i} )^{\perp}$ is closed}  $\not \Rightarrow$ \enquote{ $(\forall i \in \I \smallsetminus\{m\})$  $(\pa U_{i+1})^{\perp} +\sum^{i}_{j=1}   (\pa U_{j} )^{\perp}$ is closed}.
	
	Therefore, we know that the assumptions in \cref{theor:averacomb:BAM:shift} are more restrictive than the assumptions in \cref{theor:averacomb:BAM}. However, comparing the constant $\gamma$ in \cref{theorem:convex:comb:BAM:2} and in \cref{theor:averacomb:BAM} for $m=2$, we know that the constants associated with the convex combination of two BAMs are independent in these two theorems. Hence, we keep  \cref{theor:averacomb:BAM:shift,theor:averacomb:BAM} together.
\end{remark}

The following \cref{cor:BAM:COMPO:Pro}\cref{cor:BAM:COMPO:Pro:BAM} is a weak version of \cite[Theorem~9.33]{D2012} which shows clearly the convergence rate of the method of alternating projections.
\begin{corollary} \label{cor:BAM:COMPO:Pro}
	Let $U_{1}, \ldots, U_{m}$  be closed affine subspaces of $\mathcal{H}$ with $\cap^{m}_{i=1} U_{i} \neq \varnothing$.  Then the following statements hold:
	\begin{enumerate}
		\item \label{cor:BAM:COMPO:Pro:BAM} Assume that $(\forall i \in \I)$ $\sum^{i}_{j=1} (\pa U_{j})^{\perp}$ is closed. Then $\Pro_{U_{m}}\cdots \Pro_{U_{2}}\Pro_{U_{1}}$ is a  BAM; moreover, there exists $\gamma \in \left[0,1\right[$ such that
		\begin{align*}
		(\forall x \in \mathcal{H}) \quad \norm{ (\Pro_{U_{m}}\cdots \Pro_{U_{2}}\Pro_{U_{1}})^{k}x - \Pro_{ \cap^{m}_{i=1} U_{i}} x} \leq \gamma^{k} \norm{x - \Pro_{ \cap^{m}_{i=1} U_{i}} x}.
		\end{align*}
		\item \label{cor:BAM:COMPO:Pro:comb} Suppose that $\sum_{i \in \I} (\pa U_{i} )^{\perp} $  is closed.  Let $(\omega_{i})_{1\leq i \leq m}$ be real numbers in $\left]0,1\right] $ such that $\sum^{m}_{i =1} \omega_{i}=1$. Then $\sum^{m}_{i =1}  \omega_{i} \Pro_{U_{i}}$ is a  BAM.  Moreover, there exists $\gamma \in \left[0,1\right[$ such that
		\begin{align*}
		(\forall x \in \mathcal{H}) \quad \Norm{ (\sum^{m}_{i =1}  \omega_{i} \Pro_{U_{i}})^{k}x - \Pro_{ \cap^{m}_{i=1} U_{i}} x} \leq \gamma^{k} \norm{x - \Pro_{ \cap^{m}_{i=1} U_{i}} x}.
		\end{align*}
	\end{enumerate}
\end{corollary}

\begin{proof}
	By \cref{examp:BAM:Pro}, we know that $(\forall i \in \{1, \ldots, m\})$ $\Pro_{U_{i}}$ is a $0$-BAM and $\Fix \Pro_{U_{i}}=U_{i}$ is a closed affine subspace.
	
\cref{cor:BAM:COMPO:Pro:BAM}: This comes from  \cref{thm:BAM:COMPO}\cref{thm:BAM:COMPO:BAM}$\&$\cref{thm:BAM:COMPO:LineaConve}   with $G_{1}=\Pro_{U_{1}}, \ldots,  G_{m}=\Pro_{U_{m}}$.

\cref{cor:BAM:COMPO:Pro:comb}: This follows by  \cref{theor:averacomb:BAM}.
\end{proof}

%%%%%%%%%%%%%%%%%%%%%%%%%%%%%%%%%%%%%%%%%%%%%%%%%%%%%%%%%%%%%%%%%%%%
%%%%%%%%%%%%%%%%%\section{Connections between BAMs and CIMs} %%%%%%%%%%%%%%%
%%%%%%%%%%%%%%%%%%%%%%%%%%%%%%%%%%%%%%%%%%%%%%%%%%%%%%%%%%%%%%%%%%%%
\section{Connections between BAMs and circumcenter mappings} \label{sec:BAMandCIM}
In this section, we present BAMs which are not projections in Hilbert spaces. In particular, we connect the circumcenter mapping with BAM.

\subsection*{Definitions and facts on circumcentered isometry methods}
Before we turn to  the relationship between best approximation mapping and circumcenter mapping, we need the background and facts on the circumcenter mapping and the circumcentered method in this section.

By  \cite[Proposition~3.3]{BOyW2018}, we know that the following definition is well defined.
\begin{definition}[circumcenter operator]  {\rm \cite[Definition~3.4]{BOyW2018}} \label{defn:Circumcenter}
	Let $\mathcal{P}(\mathcal{H})$ be the set of all nonempty subsets of $\mathcal{H}$ containing \emph{finitely many} elements. The \emph{circumcenter operator} is
	\begin{align*}
	\CCO{} \colon \mathcal{P}(\mathcal{H}) \to \mathcal{H} \cup \{ \varnothing \} \colon K \mapsto \begin{cases} p, \quad ~\text{if}~p \in \aff (K)~\text{and}~\{\norm{p-y} ~:~y \in K \}~\text{is a singleton};\\
	\varnothing, \quad~ \text{otherwise}.
	\end{cases}
	\end{align*}
	In particular, when $\CCO(K) \in \mathcal{H}$, that is, $\CCO(K) \neq \varnothing$, we say that the circumcenter of $K$ exists and we call $\CCO(K)$ the \emph{circumcenter} of $K$.
\end{definition}

\begin{definition}[circumcenter mapping] {\rm \cite[Definition~3.1]{BOyW2018Proper}} \label{def:cir:map}
Let 	$F_{1}, \ldots, F_{m}$ be operators from $\mathcal{H}$ to $\mathcal{H}$ such that $\cap^{m}_{j=1} \Fix F_{j} \neq \varnothing$. Set $\mathcal{S}:=\{ F_{1}, \ldots,  F_{m} \}$ and  $(\forall x \in \mathcal{H}) $ $\mathcal{S}(x):=\{ F_{1}x, \ldots,  F_{m}x\} $.
	The \emph{circumcenter mapping} induced by $\mathcal{S}$ is
	\begin{align*}
	\CC{\mathcal{S}} \colon \mathcal{H} \to \mathcal{H} \cup \{ \varnothing \} \colon x \mapsto \CCO(\mathcal{S}(x)),
	\end{align*}
	that is, for every $x \in \mathcal{H}$, if the circumcenter of the set $\mathcal{S}(x)$ defined in \cref{defn:Circumcenter} does not exist, then  $\CC{\mathcal{S}}x= \varnothing $. Otherwise, $\CC{\mathcal{S}}x$ is the unique point satisfying the two conditions below:
	\begin{enumerate}
		\item $\CC{\mathcal{S}}x \in \aff(\mathcal{S}(x))=\aff\{F_{1}(x), \ldots,    F_{m}(x)\}$, and
		\item $\norm{\CC{\mathcal{S}}x -F_{1}(x)}=\cdots =\norm{\CC{\mathcal{S}}x -F_{m}(x)}$.
	\end{enumerate}
	In particular, if for every $x \in \mathcal{H}$, $\CC{\mathcal{S}}x \in \mathcal{H}$, then we say the circumcenter mapping $\CC{\mathcal{S}}$ induced by $\mathcal{S}$ is \emph{proper}. Otherwise, we call  $\CC{\mathcal{S}}$ \emph{improper}.
\end{definition}

\begin{fact} {\rm \cite[Proposition~3.7(ii)]{BOyW2018Proper} } \label{fact:FixCCS:G}
Let 	$F_{1}, \ldots, F_{m}$ be operators from $\mathcal{H}$ to $\mathcal{H}$ with $\cap^{m}_{j=1} \Fix F_{j} \neq \varnothing$. Set $\mathcal{S}:=\{ F_{1}, \ldots,  F_{m} \}$.	Assume that $\CC{\mathcal{S}}$ is proper and that $\Id \in \mathcal{S}$. Then $\Fix \CC{\mathcal{S}} = \cap^{m}_{i=1} \Fix F_{i}$.
\end{fact}

\begin{fact} \label{fact:form:m2:Oper} {\rm  \cite[Proposition~3.3]{BOyW2018Proper}}
	Let 	$F_{1},   F_{2}$ be operators from $\mathcal{H}$ to $\mathcal{H}$ and set $\mathcal{S}:=\{F_{1}, F_{2}\}$. Then
	\begin{align*}
	(\forall x \in \mathcal{H}) \quad  \CC{\mathcal{S}}x = \frac{F_{1}x +F_{2}x}{2}.
	\end{align*}
\end{fact}

Let $x \in \mathcal{H}$ and assume that $\CC{\mathcal{S}}$ is proper. The \emph{circumcenter method} induced by $\mathcal{S}$ is
\begin{align} \label{eq:CIMSequence}
x_{0}:=x, ~\mbox{and}~x_{k}:=\CC{\mathcal{S}}(x_{k-1})=\CC{\mathcal{S}}^{k}x, ~\mbox{where}~k=1,2,\ldots.
\end{align}

\begin{definition} \label{defn:isometry} {\rm \cite[Definition~1.6-1]{Kreyszig1989}}
	A mapping $T: \mathcal{H} \rightarrow \mathcal{H}$ is said to be \emph{isometric} or an \emph{isometry} if
	\begin{align} \label{eq:T:normpreserving}
	(\forall x \in \mathcal{H}) (\forall y \in \mathcal{H}) \quad \norm{Tx -Ty} =\norm{x-y}.
	\end{align}
\end{definition}

\begin{fact} {\rm \cite[Proposition~3.3 and Corollary~3.4]{BOyW2019LinearConvergence} } \label{fact:IsometryAffine}
	Let $T: \mathcal{H} \rightarrow \mathcal{H}$ be isometric. Then $T$ is affine.  Moreover, if $\Fix T$ is nonempty, then $\Fix T$ is a  closed affine  subspace.
\end{fact}

Note that by \cref{fact:IsometryAffine}, every isometry must be affine. In the rest of this section, without otherwise statement,
\begin{empheq}[box = \mybluebox]{equation*}
\big(\forall i \in \{1, \ldots, m\} \big) \quad T_{i} : \mathcal{H} \rightarrow \mathcal{H} ~\text{is affine isometry} \quad \text{with} \quad \bigcap^{m}_{j=1} \Fix T_{j} \neq \varnothing.
\end{empheq}
Denote by
\begin{empheq}[box = \mybluebox]{equation*}
\mathcal{S}:=\{ T_{1}, \ldots, T_{m-1}, T_{m} \}.
\end{empheq}
The associated set-valued operator $\mathcal{S}: \mathcal{H} \rightarrow \mathcal{P}(\mathcal{H})$ is defined by
\begin{align*}
(\forall x \in \mathcal{H}) \quad \mathcal{S}(x):=\{ T_{1}x, \ldots, T_{m-1}x, T_{m}x\}.
\end{align*}

The following \cref{fact:CCS:proper:NormPres:T}\cref{fact:CCS:proper:NormPres:T:prop} makes the circumcentered method induced by $\mathcal{S}$ defined in \cref{eq:CIMSequence} well-defined.  Since every element of  $\mathcal{S}$ is isometry, we call the circumcentered method induced by the $\mathcal{S}$ \emph{circumcentered isometry method} (CIM).

\begin{fact}  \label{fact:CCS:proper:NormPres:T} {\rm \cite[Theorem~3.3  and Proposition~4.2]{BOyW2019Isometry}}
	Let $x \in \mathcal{H}$. Then the following statements hold:
	\begin{enumerate}
		\item \label{fact:CCS:proper:NormPres:T:prop} The circumcenter mapping
		$\CC{\mathcal{S}} : \mathcal{H} \rightarrow \mathcal{H}$ induced by
		$\mathcal{S}$ is proper; moreover,
		$\CC{\mathcal{S}}x$ is the unique point satisfying the two conditions
		below:
		\begin{enumerate}
			\item  \label{thm:CCS:proper:NormPres:T:i} $\CC{\mathcal{S}}x\in  \aff (\mathcal{S}(x))$, and
			\item  \label{thm:CCS:proper:NormPres:T:ii} $
			\left\{  \norm{\CC{\mathcal{S}}x-Tx } ~:~ T \in \mathcal{S} \right\} $ is a singleton.
		\end{enumerate}

		\item \label{fact:CCS:proper:NormPres:T:PaffU}   Let $W$ be nonempty closed  affine subspace of $\cap^{m}_{i=1} \Fix T_{i}$.   Then $(\forall k \in \mathbb{N})$ $\Pro_{W} \CC{\mathcal{S}}^{k}=\Pro_{W}= \CC{\mathcal{S}}^{k}\Pro_{W}$.
	\end{enumerate}
\end{fact}

\begin{fact}  \label{fact:CWP:line:conv} {\rm  \cite[Theorem~4.15(i)]{BOyW2019Isometry}}
	Let $W$ be a nonempty closed  affine subspace of $\cap^{m}_{i=1} \Fix T_{i} $. 	
	Assume that there exist $F: \mathcal{H} \to \mathcal{H}$ and $\gamma \in \left[0,1\right[$ such that $(\forall x \in \mathcal{H})$ $F(x) \in \aff(\mathcal{S}(x) )$ and
	$(\forall x \in \mathcal{H})$ $ \norm{Fx-\Pro_{W} x} \leq \gamma \norm{x -\Pro_{W} x}.
	$
	Then
	\begin{align*}
	(\forall x \in \mathcal{H}) (\forall k \in \mathbb{N}) \quad \norm{\CC{\mathcal{S}}^{k}x-\Pro_{W} x} \leq \gamma^{k} \norm{x -\Pro_{W} x}.
	\end{align*}
\end{fact}
In fact, it is easy to show that $W=\Fix \CC{\mathcal{S}}$ from the last inequality with $k=1$ in \cref{fact:CWP:line:conv}.
\begin{fact}  {\rm  \cite[Theorem~4.16(ii)]{BOyW2019Isometry}}
	\label{fact:CCSLinearConvTSFirmNone}
	Suppose that $\mathcal{H}= \mathbb{R}^{n}$.
	Let $T_{\mathcal{S}} \in \aff \mathcal{S}$  satisfy  $ \Fix T_{\mathcal{S}} \subseteq \cap_{T \in \mathcal{S}} \Fix T$.  Assume that $T_{\mathcal{S}} $ is linear and $\alpha$-averaged with $\alpha \in \, \left]0,1\right[\,$. Then $\norm{T_{\mathcal{S}} \Pro_{(\cap_{T \in \mathcal{S}} \Fix T)^{\perp} }} \in \left[0,1\right[\,$. Moreover,
	\begin{align*}
	(\forall x \in \mathcal{H})  (\forall k \in \mathbb{N}) \quad \norm{\CC{\mathcal{S}}^{k}x-\Pro_{\cap_{T \in \mathcal{S}} \Fix T} x} \leq \norm{T_{\mathcal{S}} \Pro_{(\cap_{T \in \mathcal{S}} \Fix T)^{\perp} }}^{k} \norm{x -\Pro_{\cap_{T \in \mathcal{S}} \Fix T } x}.
	\end{align*}
\end{fact}

\subsection*{Circumcenter mappings that are BAMs }

\begin{theorem} \label{thm:CCSBAM:F}
	Let $W$ be a nonempty closed affine subspace of $\cap^{m}_{i=1} \Fix T_{i}$.  Assume that  $\Id \in \mathcal{S}$ and that there exists $F : \mathcal{H} \to \mathcal{H}$ and $\gamma \in \left[0,1\right[$ such that $(\forall x \in \mathcal{H})$ $Fx \in \aff (\mathcal{S}(x))$ and $(\forall x \in \mathcal{H})$ $\norm{Fx - \Pro_{W}x} \leq \gamma \norm{x - \Pro_{W}x}$. Then $\CC{\mathcal{S}}$ is a  $\gamma$-BAM and $\Fix \CC{\mathcal{S}} =\cap^{m}_{i=1} \Fix T_{i}$.
\end{theorem}	
	
	\begin{proof}
By  \cref{fact:CCS:proper:NormPres:T}\cref{fact:CCS:proper:NormPres:T:prop}, $\CC{\mathcal{S}}$ is proper. Then
	by \Cref{fact:IsometryAffine,fact:FixCCS:G}, $\Fix \CC{\mathcal{S}} =\cap^{m}_{i=1} \Fix T_{i}$ is a nonempty closed affine subspace of $\mathcal{H}$. Apply \cref{fact:CCS:proper:NormPres:T}\cref{fact:CCS:proper:NormPres:T:PaffU} with $W=\cap^{m}_{i=1} \Fix T_{i}$ to obtain that $\Pro_{\Fix \CC{\mathcal{S}}  }\CC{\mathcal{S}} = \Pro_{\Fix \CC{\mathcal{S}}  }$.
	 Moreover, by the assumptions,  \cref{fact:CWP:line:conv} and \cref{def:BAM}, we know   that $\CC{\mathcal{S}}$ is a $\gamma$-BAM.
	\end{proof}

 The following result states that in order to study whether the circumcenter mapping $ \CC{\mathcal{S}} $   is a BAM or not, we are free to assume the related isometries are linear.
 \begin{proposition} \label{prop:S:SF:equivalence}
 	  Let $z \in \cap^{m}_{i=1} \Fix T_{i}$. Define $\left( \forall i \in \{1,\ldots, m\}  \right) $ $(\forall x \in \mathcal{H})$ $F_{i}x:= T_{i} (x+z) -z$.
 	Set   $\mathcal{S}_{F} := \{F_{1}, \ldots, F_{m}\}$.
 	Then the following statements hold:
 	\begin{enumerate}
 		\item \label{prop:S:SF:equivalence:SF} $\mathcal{S}_{F}$ is a set of linear isometries.
 		\item \label{prop:S:SF:equivalence:S:SF} Let $\gamma \in \left[0,1\right[\,$. Assume that $\Id \in \mathcal{S}$. Then $ \CC{\mathcal{S}} $ is a $\gamma$-BAM  if and only if $\CC{\mathcal{S}_{F}}$ is a $\gamma$-BAM.
 	\end{enumerate}
 \end{proposition}

 \begin{proof}
 	\cref{prop:S:SF:equivalence:SF}:  	Because $z \in \cap^{m}_{i=1} \Fix T_{i}$, by \cite[Lemma~3.8]{BOyW2019LinearConvergence}, $F_{1}, \ldots, F_{m}$ are linear isometries.

 	\cref{prop:S:SF:equivalence:S:SF}: 		Because both $\mathcal{S} $ and $\mathcal{S}_{F}$ are sets of isometries, by \cref{fact:CCS:proper:NormPres:T}\cref{fact:CCS:proper:NormPres:T:prop}, both $ \CC{\mathcal{S}} $ and $\CC{\mathcal{S}_{F}}$ are proper.
 	Clearly, $\Id \in \mathcal{S}$ implies that $\Id \in \mathcal{S}_{F}$ as well. So, by \cref{fact:FixCCS:G} ,   $\Fix \CC{\mathcal{S}} = \cap^{m}_{i=1}\Fix T_{i}$ and  $\Fix \CC{\mathcal{S}_{F}} = \cap^{m}_{i=1}\Fix F_{i}$.  In addition, by \cite[Lemma~4.8]{BOyW2019LinearConvergence}, $(\forall x \in \mathcal{H})$ $ \CC{\mathcal{S}} x =  z+ \CC{\mathcal{S}_{F}}(x-z)$. Hence, the desired result comes from \cref{prop:BAMAffinLinea} and \cref{def:BAM}.
 \end{proof}	

  \begin{theorem} \label{thm:CCS:BAM:averaged}
 	Suppose that $\mathcal{H}= \mathbb{R}^{n}$.
 	Let $T_{\mathcal{S}} \in \aff \mathcal{S}$  satisfy that $ \Fix T_{\mathcal{S}} \subseteq \cap_{T \in \mathcal{S}} \Fix T$.  Assume that $T_{\mathcal{S}} $ is linear and $\alpha$-averaged with $\alpha \in \, \left]0,1\right[\,$. Then $\gamma:=\norm{T_{\mathcal{S}} \Pro_{(\cap^{m}_{i=1} \Fix T_{i})^{\perp} }} \in \left[0,1\right[\,$ and $\CC{\mathcal{S}}$ is a $\gamma$-BAM.
 \end{theorem}

 \begin{proof}
 	By \Cref{fact:IsometryAffine,fact:FixCCS:G},  and assumptions, $\Fix \CC{\mathcal{S}}= \cap^{m}_{i=1} \Fix T_{i}  $  is a closed  affine subspace. Apply \cref{fact:CCS:proper:NormPres:T}\cref{fact:CCS:proper:NormPres:T:PaffU} with $W= \cap^{m}_{i=1} \Fix T_{i} $ to obtain that $\Pro_{ \Fix \CC{\mathcal{S}} } \CC{\mathcal{S}} =\Pro_{ \Fix \CC{\mathcal{S}} }$. Moreover, by
 	\cref{fact:CCSLinearConvTSFirmNone}, $\gamma \in \left[0,1\right[$ and $(\forall x \in \mathcal{H})  $ $\norm{\CC{\mathcal{S}}x-\Pro_{\cap^{m}_{i=1} \Fix T_{i}} x} \leq \gamma\norm{x -\Pro_{\cap^{m}_{i=1} \Fix T_{i}} x}$.
 \end{proof}

Let $t \in \mathbb{N} \smallsetminus \{0\}$ and let $(\forall i \in \{1,\ldots, t\})$ $F_{i}: \mathcal{H} \to \mathcal{H}$. From now on, to facilitate the statements later, we denote
\begin{align}  \label{eq:defn:Omega}
\Omega ( F_{1}, \ldots, F_{t}) :=  \Big\{ F_{i_{r}}\cdots F_{i_{2}}F_{i_{1}}  ~\Big|~ r \in \mathbb{N}, ~\mbox{and}~ i_{1}, \ldots,  i_{r} \in \{1, \ldots,t\}    \Big\}
\end{align}
which is the set consisting of all finite composition of operators from $\{F_{1}, \ldots, F_{t}\}$. We use the empty product convention: $F_{i_{0}}\cdots F_{i_{1}} = \Id$.

  \begin{fact} \label{fact:CCS:LineaConve:F1Fn} {\rm \cite[Theorem~5.4]{BOyW2019LinearConvergence}}
  	Suppose that $\mathcal{H} = \mathbb{R}^{n}$. Let $F_{1}, F_{2},  \ldots, F_{t}$ be linear isometries on $\mathcal{H}$. Assume  that $\widetilde{\mathcal{S}}$ is a finite subset of $\Omega(F_{1}, \ldots, F_{t})$, where $\Omega(F_{1}, \ldots, F_{t})$ is defined in \cref{eq:defn:Omega}. Assume that   $\{ \Id,  F_{1}, F_{2},  \ldots, F_{t} \}  \subseteq \widetilde{\mathcal{S}}$.  Let $(\omega_{i})_{i \in \I}$ be real numbers in $]0,1]$ such that $\sum_{i \in \I} \omega_{i} =1$ and let $(\alpha_{i})_{i \in \I}$ be real numbers in $\left]0,1\right[\,$.  Denote $A:= \sum^{t}_{i =1} \omega_{i} A_{i}$ where $(\forall i \in \{1,\ldots,t\})$ $A_{i}:=
  	(1-\alpha_{i}) \Id  +  \alpha_{i}  F_{i}$.  Then the following statements hold:
  	\begin{enumerate}
  		\item \label{fact:CCS:LineaConve:F1Fn:Fix} $\Fix \CC{\widetilde{\mathcal{S}}}= \cap_{T \in \widetilde{\mathcal{S}}} \Fix T=  \cap^{t}_{i=1} \Fix F_{i}  = \Fix A $.
  		\item \label{fact:CCS:LineaConve:F1Fn:LineaConv} $\norm{A \Pro_{(\cap^{t}_{i=1} \Fix F_{i})^{\perp}} } <1$. Moreover,
  		\begin{align*}
  		(\forall x \in \mathcal{H})  (\forall k \in \mathbb{N}) \quad \norm{\CC{\widetilde{\mathcal{S}}}^{k}x - \Pro_{\cap^{t}_{i=1} \Fix F_{i}}x} \leq \norm{A \Pro_{(\cap^{t}_{i=1} \Fix F_{i})^{\perp}} }^{k} \norm{x-\Pro_{\cap^{t}_{i=1} \Fix F_{i}} x }.
  		\end{align*}
  	\end{enumerate}
  \end{fact}

  \begin{fact} \label{fact:CCS:LineaConve:F1t}  {\rm \cite[Theorem~5.6]{BOyW2019LinearConvergence}}
  	Suppose that  $\mathcal{H} = \mathbb{R}^{n}$. Let $F_{1}, F_{2},  \ldots, F_{t}$ be linear isometries. Assume  that $\widetilde{\mathcal{S}}$ is a finite subset of $\Omega(F_{1}, \ldots, F_{t})$, where $\Omega(F_{1}, \ldots, F_{t})$ is defined in \cref{eq:defn:Omega}.
  	Assume that  $ \{ \Id,  F_{1}, F_{2}F_{1},  \ldots, F_{t}\cdots F_{2}F_{1} \} \subseteq \widetilde{\mathcal{S}} $. Let $(\omega_{i})_{i \in \I}$ be real numbers in $]0,1]$ such that $\sum_{i \in \I} \omega_{i} =1$ and let $(\alpha_{i})_{i \in \I}$ and $(\lambda_{i})_{i \in \I}$ be real numbers in $\left]0,1\right[\,$.  Set
  	$ A:= \sum_{i \in \I} \omega_{i} A_{i}$
  	where $A_{1} :=  (1-\alpha_{1}) \Id  + \alpha_{1} F_{1} $ and
  	\begin{align*}
  	(\forall i \in \I \smallsetminus \{1\}) \quad A_{i} := (1-\alpha_{i}) \Id  + \alpha_{i}  \left ( (1-\lambda_{i}) \Id + \lambda_{i}F_{i} \right )F_{i-1}\cdots F_{1}.
  	\end{align*}
  	Then the following assertions hold:
  	\begin{enumerate}
  		\item \label{fact:CCS:LineaConve:F1t:Fix} $\Fix \CC{\widetilde{\mathcal{S}}}= \cap_{T \in \widetilde{\mathcal{S}}} \Fix T=  \cap^{t}_{i=1} \Fix F_{i}  = \Fix A $.
  		\item \label{fact:CCS:LineaConve:F1t:LineaConv} $\norm{A \Pro_{(\cap^{t}_{i=1} \Fix F_{i})^{\perp}} } \in \left[0,1\right[\,$. Moreover,
  		\begin{align*}
  		(\forall x \in \mathcal{H}) (\forall k \in \mathbb{N}) \quad \norm{\CC{\widetilde{\mathcal{S}}}^{k}x - \Pro_{\cap^{t}_{i=1} \Fix F_{i}}x} \leq \norm{A \Pro_{(\cap^{t}_{i=1} \Fix F_{i})^{\perp}} }^{k} \norm{x-\Pro_{\cap^{t}_{i=1} \Fix F_{i}} x }.
  		\end{align*}
  	\end{enumerate}
  \end{fact}

		\begin{proposition} \label{prop:CCSBAM}
		Suppose that $\mathcal{H} =\mathbb{R}^{n}$. Let $F_{1}, \ldots, F_{t}$ be linear isometries from $\mathcal{H} $ to $\mathcal{H}$.  Let $\tilde{ \mathcal{S}}$ be a finite subset of $\Omega ( F_{1}, \ldots, F_{t})$.
		\begin{enumerate}
		\item \label{prop:CCSBAM:T1Tm}  If  $\{ \Id, F_{1},F_{2},  \ldots, F_{t} \}  \subseteq \tilde{ \mathcal{S}}$,
		then $\Fix \CC{\tilde{ \mathcal{S}}} = \cap^{t}_{i=1} \Fix F_{i}$ and $\CC{\tilde{ \mathcal{S}}}$ is a  BAM

		\item \label{prop:CCSBAM:T12mTm} If $ \{ \Id, F_{1}, F_{2}F_{1},  \ldots, F_{t}\cdots F_{2}F_{1} \} \subseteq \tilde{ \mathcal{S}}$, then $\Fix \CC{\tilde{ \mathcal{S}}} = \cap^{t}_{i=1} \Fix F_{i}$ and  $\CC{\tilde{ \mathcal{S}}}$ is a  BAM.
	\end{enumerate}
\end{proposition}	

\begin{proof}
	\cref{prop:CCSBAM:T1Tm}:  By \cref{fact:CCS:LineaConve:F1Fn}\cref{fact:CCS:LineaConve:F1Fn:Fix}, $\Fix \CC{ \tilde{ \mathcal{S}} } = \cap_{T \in \widetilde{\mathcal{S}}} \Fix T= \cap^{t}_{i=1} \Fix F_{i}$ is a nonempty  closed linear subspace of $\mathcal{H}$.
	Apply \cref{fact:CCS:proper:NormPres:T}\cref{fact:CCS:proper:NormPres:T:PaffU} with $W= \Fix \CC{ \tilde{ \mathcal{S}} } = \cap_{T \in \widetilde{\mathcal{S}}} \Fix T$ yields $\Pro_{\Fix \CC{ \tilde{ \mathcal{S}} } } \CC{\tilde{ \mathcal{S}}} = \Pro_{\Fix \CC{ \tilde{ \mathcal{S}} }  }$.
	 Apply \cref{fact:CCS:LineaConve:F1Fn}\cref{fact:CCS:LineaConve:F1Fn:LineaConv}   to obtain that
	there exists $\gamma \in \left[0,1\right[\,$ such that
	\begin{align*}
	(\forall x \in \mathcal{H}) \quad \norm{\CC{\tilde{ \mathcal{S}}}x - \Pro_{\Fix \CC{ \tilde{ \mathcal{S}} } } x  } \leq \gamma \norm{x - \Pro_{\Fix \CC{ \tilde{ \mathcal{S}} } } \Fix Tx}.
	\end{align*}
	Hence, by \cref{def:BAM}, $\CC{\tilde{ \mathcal{S}}}$ is a  BAM.
	
 \cref{prop:CCSBAM:T12mTm}:  The proof is similar to the proof of \cref{prop:CCSBAM:T1Tm}, however, this time we use \cref{fact:CCS:LineaConve:F1t}  instead of \cref{fact:CCS:LineaConve:F1Fn}.
\end{proof}

\begin{theorem} \label{theom:CCS:affine}
	Assume that $\mathcal{H} =\mathbb{R}^{n}$ and that $F_{1}, \ldots, F_{t}$ are affine isometries from $\mathcal{H} $ to $\mathcal{H}$ with $\cap^{t}_{i=1} \Fix F_{i} \neq \varnothing$.  Assume that $\tilde{ \mathcal{S}}$ is a finite subset of $\Omega ( F_{1}, \ldots, F_{t}) $ defined in \cref{eq:defn:Omega} such that
	$\{ \Id,  F_{1},F_{2},  \ldots, F_{t} \}  \subseteq \tilde{ \mathcal{S}}$ or $ \{ \Id,  F_{1}, F_{2}F_{1},  \ldots, F_{t}\cdots F_{2}F_{1} \} \subseteq \tilde{ \mathcal{S}} $.   Then $\Fix \CC{\tilde{ \mathcal{S}}} = \cap^{t}_{i=1} \Fix F_{i}$ and  $\CC{\tilde{ \mathcal{S}}}$ is a  BAM.
\end{theorem}	

\begin{proof}
	This is from  \cref{prop:CCSBAM}\cref{prop:CCSBAM:T1Tm}$\&$\cref{prop:CCSBAM:T12mTm} and \cref{prop:S:SF:equivalence}\cref{prop:S:SF:equivalence:S:SF}.
\end{proof}

The following example shows that BAM is generally neither continuous nor linear.
\begin{example} \label{exam:counterexam}
	Suppose that $\mathcal{H}=\mathbb{R}^2$, set $U_{1}:=\mathbb{R}\cdot (1,0)$,
	and  $U_{2}:=\mathbb{R}\cdot(1,1)$. Suppose that
	$\mathcal{S}=\{\Id, \R_{U_{1}}, \R_{U_{2}}\}$ or that $ \mathcal{S}=\{\Id, \R_{U_{1}}, \R_{U_{2}}\R_{U_{1}}\}$. Then the following statements hold.
	\begin{enumerate}
		\item \label{exam:counterexam:BAM} $\CC{\mathcal{S}}$ is a BAM and $\Fix \CC{\mathcal{S}}=\{(0,0)\}$.
		\item \label{exam:counterexam:discontinuity:nonlinear} $\CC{\mathcal{S}}$ is   neither continuous nor linear.
	\end{enumerate}
\end{example}

\begin{proof}
	\cref{exam:counterexam:BAM}: Because $\R_{U_{1}}$ and $\R_{U_{2}}$ are linear isometries and $\Fix \R_{U_{1}} \cap \Fix \R_{U_{2}} =U_{1} \cap U_{2}=\{(0,0)\}$, by \cref{theom:CCS:affine},  $\CC{\mathcal{S}}$ is a BAM.
	
	\cref{exam:counterexam:discontinuity:nonlinear}: This is from \cite[Examples~4.19 and 4.20]{BOyW2018Proper}.
\end{proof}

The following example illustrates that the composition of three BAMs is a projector does not imply that the individual BAMs are  projectors.
\begin{example} \label{examp:CCS1CCS2}
	Suppose that $\mathcal{H}=\mathbb{R}^2$, set $U_{1}:=\mathbb{R}\cdot (1,0)$, $U_{2}:=\mathbb{R}\cdot(1,1)$ and $U_{3}:=\mathbb{R}\cdot(0, 1)$. Denote by
	$\mathcal{S}_{1} :=\{\Id, \R_{U_{1}}, \R_{U_{2}}\}$ and $\mathcal{S}_{2}:=\{\Id, \R_{U_{2}}, \R_{U_{3}}\}$. Then the following statements hold:
	\begin{enumerate}
		\item \label{examp:CCS1CCS2:BAM} All of  $\CC{\mathcal{S}_{1}}$, $\CC{\mathcal{S}_{2}}$ and $\CC{\mathcal{S}_{2}} \CC{\mathcal{S}_{1}}$ are  BAMs. Moreover, $\Fix \CC{\mathcal{S}_{1}}=\{(0,0)\}$, $\Fix \CC{\mathcal{S}_{2}} =\{(0,0)\}$, and $ \Fix (\CC{\mathcal{S}_{2}}  \CC{\mathcal{S}_{1}})=\{(0,0)\}$.
		\item \label{examp:CCS1CCS2:Projector}
		None of the $\CC{\mathcal{S}_{1}}$, $\CC{\mathcal{S}_{2}}$  or $\CC{\mathcal{S}_{2}}  \CC{\mathcal{S}_{1}}$  is a projector.
		\item \label{examp:CCS1CCS2:0} $\CC{\mathcal{S}_{1}}\CC{\mathcal{S}_{2}}  \CC{\mathcal{S}_{1}} =\Pro_{\{(0,0)\}}$.
	\end{enumerate}
\end{example}

\begin{proof}
	By \cref{fact:form:m2:Oper}, it is easy to see that
	\begin{align} \label{EQ:examp:CCS1CCS2}
	(\forall x \in \mathcal{H}) \quad \CC{\mathcal{S}_{1}}x =\begin{cases}
	\Pro_{ U_{2}}x, \quad &\text{if } x \in U_{1};\\
	\Pro_{ U_{1}}x, \quad &\text{if } x \in U_{2};\\
	0, \quad &\text{otherwise.}
	\end{cases}
	\quad \text{and} \quad
	\CC{\mathcal{S}_{2}}x =\begin{cases}
	\Pro_{ U_{3}}x, \quad &\text{if } x \in U_{2};\\
	\Pro_{ U_{2}}x, \quad &\text{if } x \in U_{3};\\
	0, \quad &\text{otherwise. }
	\end{cases}
	\end{align}
	Hence,
	\begin{align}\label{EQ:examp:CCS1compCCS2}
	\CC{\mathcal{S}_{2}}\CC{\mathcal{S}_{1}}x =\begin{cases}
	\Pro_{ U_{3}}\Pro_{ U_{2}}x, \quad &\text{if } x \in U_{1};\\
	0, \quad &\text{if } x \in U_{2};\\
	0, \quad &\text{otherwise. }
	\end{cases}
	\end{align}
	
	\cref{examp:CCS1CCS2:BAM}: Because $\R_{U_{1}}$, $\R_{U_{2}}$ and $\R_{U_{3}}$ are linear isometries, $\Fix \R_{U_{1}} \cap \Fix \R_{U_{2}} =\{(0,0)\}$, and $\Fix \R_{U_{2}} \cap \Fix \R_{U_{3}} =U_{2} \cap U_{3}=\{(0,0)\}$, by \cref{theom:CCS:affine},
	$\CC{\mathcal{S}_{1}}$ and $\CC{\mathcal{S}_{2}}$ are BAMs and $\Fix \CC{\mathcal{S}_{1}}= \Fix \CC{\mathcal{S}_{2}} =\{(0,0)\}$.
Hence, by \cref{thm:BAM:COMPO}\cref{thm:BAM:COMPO:BAM}, $\CC{\mathcal{S}_{2}}  \CC{\mathcal{S}_{1}}$ is  BAM and $ \Fix (\CC{\mathcal{S}_{2}}  \CC{\mathcal{S}_{1}})=\{(0,0)\}$.

	\cref{examp:CCS1CCS2:Projector}:  Because $U_{1}$ is not  orthogonal with $U_{2}$, and by \cref{EQ:examp:CCS1CCS2},  the range of $\CC{\mathcal{S}_{1}} $ equals $U_{1} \cup U_{2}$, but   $\CC{\mathcal{S}_{1}} \neq \Pro_{ U_{1} \cup U_{2}}$, we know that $\CC{\mathcal{S}_{1}}$ is not a projector. Similarly, neither $\CC{\mathcal{S}_{2}} $ nor $\CC{\mathcal{S}_{2}}  \CC{\mathcal{S}_{1}}$ is   a projector.

	\cref{examp:CCS1CCS2:0}: This is clear  from the definitions of $\CC{\mathcal{S}_{1}}$ and $\CC{\mathcal{S}_{2}}\CC{\mathcal{S}_{1}}$ presented in \cref{EQ:examp:CCS1CCS2} and \cref{EQ:examp:CCS1compCCS2} respectively.

\end{proof}
\subsection*{Circumcenter and best approximation mappings in Hilbert space}
Because reflectors associated with closed affine subspaces are isometries,   we call  the circumcenter method induced by a set of reflectors the \emph{circumcentered reflection method} (CRM).  Clearly, all facts on CIM are applicable to CRM.

In this subsection, we assume that
\begin{subequations}\label{eq:U1Um}
	\begin{align}
	U_{1}, \ldots, U_{m}~\text{are closed linear subspaces in the real Hilbert space}~\mathcal{H},
	\end{align}
	\begin{align}
	\Omega := \Omega(\R_{U_{1}}, \ldots, \R_{U_{m}} ) := \Big\{ \R_{U_{i_{r}}}\cdots \R_{U_{i_{2}}}\R_{U_{i_{1}}}  ~\Big|~ r \in \mathbb{N}, ~\mbox{and}~ i_{1}, \ldots,  i_{r} \in \{1, \ldots,m\}    \Big\},
	\end{align}
	\begin{align}
	\Psi :=\Big\{ \R_{U_{i_{r}}}\cdots \R_{U_{i_{2}}}\R_{U_{i_{1}}}  ~\Big|~ r, i_{1}, i_{2}, \ldots, i_{r} \in \{0,1, \ldots,m\} ~\mbox{and}~0<i_{1}<\cdots<i_{r} \Big\}.
	\end{align}
\end{subequations}

We also assume that
	\begin{align}\label{eq:Defi:S}
\Psi \subseteq \mathcal{S} \subseteq \Omega \quad \text{and} \quad \mathcal{S} ~\text{consists of finitely many elements.}
	\end{align}

For every nonempty closed affine subset $C$ of $\mathcal{H}$, $\R_{C}\R_{C} =(2\Pro_{C} -\Id)(2\Pro_{C} -\Id)=4 \Pro_{C} - 2 \Pro_{C}  -2\Pro_{C} +\Id =\Id$. So, if $m=1$, then $\Omega =\Psi=\{\Id, \R_{U_{1}} \}$. Hence, by the assumption, $\mathcal{S} = \{\Id, \R_{U_{1}} \}$, and, by \cref{fact:form:m2:Oper}, $\CC{\mathcal{S}} =\frac{1}{2} (\Id + \R_{U_{1}}) =\Pro_{U_{1}}$. By \cref{examp:BAM:Pro}, $\CC{\mathcal{S}} = \Pro_{U_{1}}$ is a $0$-BAM. Therefore, $m=1$ is a trivial case and we consider only $m \geq 2$ below.

\begin{lemma} \label{lemma:Fix:CCS}
	$\Fix \CC{\mathcal{S}} =\cap_{T \in \mathcal{S}} \Fix T =  \cap^{m}_{i=1} U_{i}$.
\end{lemma}

\begin{proof}
	By construction of $\Psi$ with $r=0$ and $r=1$, we know that $\{ \Id, \R_{U_{1}}, \ldots, \R_{U_{m}} \} \subseteq \Psi \subseteq \mathcal{S}$, so  by \cref{fact:FixCCS:G},  $\Fix \CC{\mathcal{S}} =\cap_{T \in \mathcal{S}} \Fix T \subseteq \cap^{m}_{i=1} \Fix \R_{U_{i}} =\cap^{m}_{i=1} U_{i} $. On the other hand, because $\mathcal{S} \subseteq \Omega $ and $(\forall T \in \Omega)$ $ \cap^{m}_{i=1} U_{i} \subseteq \Fix T$,  we know that $\cap^{m}_{i=1} U_{i}  \subseteq \cap_{T \in \mathcal{S}} \Fix T = \Fix \CC{\mathcal{S}} $. Altogether,  $\Fix \CC{\mathcal{S}} =\cap_{T \in \mathcal{S}} \Fix T= \cap^{m}_{i=1} U_{i}$.
\end{proof}

\begin{fact} {\rm \cite[Theorem~6.6]{BOyW2019LinearConvergence}}
	\label{fact:MAP:LC}
	Set   $\gamma := \norm{ \Pro_{U_{m}} \Pro_{U_{m-1}} \cdots \Pro_{U_{1}} \Pro_{(\cap^{m}_{i=1} U_{i} )^{\perp}} } $.  Assume that $m \geq 2$ and that $U^{\perp}_{1} + \cdots+ U^{\perp}_{m}$ is closed. Then $\gamma \in \left[0,1\right[$ and
	\begin{align*}
	(\forall x \in \mathcal{H} ) (\forall k \in \mathbb{N}) \quad \norm{\CC{\mathcal{S}}^{k}x- \Pro_{\cap^{m}_{i=1} U_{i}} x} \leq \gamma^{k} \norm{x - \Pro_{\cap^{m}_{i=1} U_{i}} x}.
	\end{align*}
\end{fact}

\begin{theorem}	\label{theorem:MAP:LC}
	Set   $\gamma := \norm{ \Pro_{U_{m}} \Pro_{U_{m-1}} \cdots \Pro_{U_{1}} \Pro_{(\cap^{m}_{i=1} U_{i} )^{\perp}} } $.  Assume that $m \geq 2$ and that $U^{\perp}_{1} + \cdots+ U^{\perp}_{m}$ is closed. Then $\gamma \in \left[0,1\right[\,$,  $ \Fix \CC{\mathcal{S}}  =\cap^{m}_{i=1} U_{i}$, and $\CC{\mathcal{S}}$ is a $\gamma$-BAM.
\end{theorem}

\begin{proof}
	Because $\Psi \subseteq \mathcal{S} \subseteq \Omega$, by \cref{lemma:Fix:CCS}, $ \Fix \CC{\mathcal{S}} =\cap_{T \in \mathcal{S}} \Fix T =\cap^{m}_{i=1} U_{i}$ is a closed linear subspace. Apply \cref{fact:CCS:proper:NormPres:T}\cref{fact:CCS:proper:NormPres:T:PaffU} with $W= \cap_{T \in \mathcal{S}} \Fix T $ to obtain that $ \Pro_{ \Fix \CC{\mathcal{S}} } \CC{\mathcal{S}} =\Pro_{ \Fix \CC{\mathcal{S}} }  $. In addition, by
	\cref{fact:MAP:LC}, $\gamma \in \left[0,1\right[$ and $(\forall x \in \mathcal{H})  $ $\norm{\CC{\mathcal{S}}x-\Pro_{\Fix \CC{\mathcal{S}}} x} \leq \gamma  \norm{x -\Pro_{\Fix \CC{\mathcal{S}} } x}$. Hence, by \cref{def:BAM},   $\CC{\mathcal{S}}$ is a $\gamma$-BAM.
\end{proof}

\begin{corollary}\label{cor:Hilbert:BAM}
	Assume that $m=2$ in \cref{eq:U1Um},  that $\mathcal{S} = \{\Id, \R_{U_{1}}, \R_{U_{2}}, \R_{U_{2}}\R_{U_{1}} \}$, and that $U_{1}+ U_{2}$ is closed. Set $\gamma := \norm{ \Pro_{U_{2}} \Pro_{U_{1}} \Pro_{(U_{1} \cap U_{2})^{\perp}} } $. Then $\gamma \in \left[0,1\right[\,$, $ \Fix \CC{\mathcal{S}}  =U_{1} \cap U_{2}$, and $\CC{\mathcal{S}}$ is a $\gamma$-BAM.
\end{corollary}

\begin{proof}
	By \cref{fac:cFLess1},  $U_{1}+ U_{2}$ is closed if and only if $U^{\perp}_{1}+ U^{\perp}_{2}$ is closed.
	Note that  $m=2$ in \cref{eq:U1Um} implies  $\Psi= \{\Id, \R_{U_{1}}, \R_{U_{2}}, \R_{U_{2}}\R_{U_{1}} \} =\mathcal{S} $. Hence,  the desired result is from \cref{theorem:MAP:LC} with $m=2$.
\end{proof}

\begin{theorem}
	\label{theorem:SymMAP:LC}
Let $n \in \mathbb{N} \smallsetminus \{0\}$.	Assume that $m=2n-1$ and that $U_{1}, \ldots ,U_{n}$ are closed linear subspaces of $\mathcal{H}$ with $U^{\perp}_{1} + \cdots + U^{\perp}_{n}$ being closed. Set $(\forall i \in \{1, \ldots, n-1\})$ $U_{n+i} :=U_{n-i}$.
	Denote $\gamma := \norm{ \Pro_{U_{n}} \Pro_{U_{n-1}} \cdots \Pro_{U_{1}} \Pro_{(\cap^{n}_{i=1} U_{i})^{\perp}} } $. Then $\gamma \in \left[0,1\right[\,$, $\Fix \CC{\mathcal{S}} =\cap^{m}_{i=1}U_{i}$,  and $\CC{\mathcal{S}}$ is a $\gamma^{2}$-BAM.
\end{theorem}

\begin{proof}
	Because $\Psi \subseteq \mathcal{S} \subseteq \Omega$, by \cref{lemma:Fix:CCS}, $ \Fix \CC{\mathcal{S}} =\cap_{T \in \mathcal{S}} \Fix T =\cap^{n}_{i=1} U_{i}$ is a closed linear subspace. Apply \cref{fact:CCS:proper:NormPres:T}\cref{fact:CCS:proper:NormPres:T:PaffU} with $W= \cap_{T \in \mathcal{S}} \Fix T $ to obtain that $  \Pro_{ \Fix \CC{\mathcal{S}} } \CC{\mathcal{S}} =\Pro_{ \Fix \CC{\mathcal{S}} }  $. In addition, by
	\cite[Theorem~6.7]{BOyW2019LinearConvergence}, $\gamma \in \left[0,1\right[$ and $(\forall x \in \mathcal{H})  $ $\norm{\CC{\mathcal{S}}x-\Pro_{\Fix \CC{\mathcal{S}} } x} \leq \gamma^{2}  \norm{x -\Pro_{\Fix \CC{\mathcal{S}} } x}$. Hence, by \cref{def:BAM}, $\CC{\mathcal{S}}$ is a $\gamma^{2}$-BAM.
\end{proof}

\begin{corollary}
	Assume that $m=3$ in \cref{eq:U1Um},  that  $\mathcal{S} :=\{\Id, \R_{U_{1}}, \R_{U_{2}},\R_{U_{1}}\R_{U_{2}}, \R_{U_{2}}\R_{U_{1}}, \R_{U_{1}}\R_{U_{2}}\R_{U_{1}} \}$, and that $U_{1}+ U_{2}$ is closed. Set
	$
	\gamma := \norm{ \Pro_{U_{2}} \Pro_{U_{1}} \Pro_{(U_{1} \cap U_{2})^{\perp}} }.
	$
 Then $\gamma \in \left[0,1\right[\,$, $\Fix \CC{\mathcal{S}} =U_{1} \cap U_{2}$, and $\CC{\mathcal{S}}$ is a $\gamma^{2}$-BAM.
\end{corollary}

\begin{proof}
	Let $U_{3} =U_{1}$ in \cref{eq:U1Um} with $m=3$
to obtain that $\Psi= \{\Id, \R_{U_{1}}, \R_{U_{2}},\R_{U_{1}}\R_{U_{2}}, \R_{U_{2}}\R_{U_{1}}, \R_{U_{1}}\R_{U_{2}}\R_{U_{1}} \} =\mathcal{S} $.  Hence, the required result comes from \cref{theorem:SymMAP:LC} with $n=2$  and $U_{3} =U_{1}$.
\end{proof}

\begin{remark}
	\cite[Theorem~6.8]{BOyW2019LinearConvergence} shows that the sequence of iterations of the $\CC{\mathcal{S}}$ in \cref{theorem:SymMAP:LC}  attains the convergence rate of the accelerated method of alternative projections which is no larger than the $\gamma^{2}$ presented in \cref{theorem:SymMAP:LC}.
	Hence, by \cite[Theorem~6.8]{BOyW2019LinearConvergence}, using the similar proof of \cref{theorem:SymMAP:LC}, one can  show that the constant associated with the BAM, the $\CC{\mathcal{S}}$ in \cref{theorem:SymMAP:LC}, is no larger than the convergence rate of the accelerated method of alternative projections.
\end{remark}

\subsection*{Compositions and convex combinations of circumcenter mapping}
The following \Cref{theom:CCS1:compose:comb,theom:CCS2:compose:comb} with condition \cref{theom:CCS1:compose:BAM} are  generalizations of \cite[Theorem~2]{BCS2019} from one class of circumcenter mapping induced by finite set of reflections to   two classes of more general circumcenter mappings induced by finite set of isometries.  Recall that
\begin{align*}
T_{1}, \ldots, T_{m} \text{ are affine isometries from }  \mathcal{H} \text{ to }  \mathcal{H}  \text{ with }  \cap^{m}_{i=1} \Fix T_{i} \neq \varnothing.
\end{align*}
\begin{theorem} \label{theom:CCS1:compose:comb}
	Suppose that $\mathcal{H} =\mathbb{R}^{n}$. Set $\mathcal{S}_{1} := \{\Id, T_{q_{0}+1}, T_{q_{0}+2},  \ldots, T_{q_{1}}  \}$, $\mathcal{S}_{2} := \{\Id, T_{q_{1}+1}, T_{q_{1}+2},  \ldots, T_{q_{2}}  \}$, $\ldots$, $\mathcal{S}_{t} := \{\Id, T_{q_{t-1}+1}, T_{q_{t-1}+2},  \ldots, T_{q_{t}}  \}$, with $q_{0}=0, q_{t}=m$ and $(\forall i \in \{1, \ldots, t\})$ $q_{i} -q_{i-1} \geq 1$. Suppose that one of the following holds:
	\begin{enumerate}
		\item \label{theom:CCS1:compose:BAM} $\CC{\mathcal{S}}=\CC{\mathcal{S}_{t}}\circ \CC{\mathcal{S}_{t-1}} \circ \cdots \circ \CC{\mathcal{S}_{1}}$.
		\item \label{theom:CCS1:compose:Fix}  $\CC{\mathcal{S}}=\sum^{t}_{i =1}  \omega_{i} \CC{\mathcal{S}_{i}}$, where $\{ \omega_{i} \}_{1\leq i \leq t} \subseteq \left]0,1\right] $ such that $\sum^{t}_{i =1} \omega_{i}=1$.
	\end{enumerate}
Then $\Fix \CC{\mathcal{S}} =\cap^{m}_{i=1} \Fix T_{i}$ and  $\CC{\mathcal{S}}$ is a BAM. Moreover, there exists $\gamma \in \left[0,1\right[$ such that
\begin{align*}
(\forall x \in \mathcal{H})  (\forall k \in \mathbb{N})\quad \norm{\CC{\mathcal{S}}^{k}x - \Pro_{\cap^{m}_{i=1} \Fix T_{I}}x } \leq \gamma^{k} \norm{x - \Pro_{\cap^{m}_{i=1} \Fix T_{I}}x}.
\end{align*}
\end{theorem}	

\begin{proof}
	By \cref{theom:CCS:affine}, $(\forall i \in \{1,\ldots, t\})$ $\CC{\mathcal{S}_{i}}$ is a BAM with $\Fix \CC{\mathcal{S}_{i}}=\cap^{q_{i} -1}_{j=q_{i-1}} \Fix T_{j+1}$.  Using \cref{fact:FixCCS:G,fact:CCS:proper:NormPres:T}  and  \cref{prop:CompositionBAM}\cref{prop:CompositionBAM:comp}$\&$\cref{prop:CompositionBAM:comb}, we know that  $\Fix \CC{\mathcal{S}} =\cap^{m}_{i=1} \Fix T_{i}$.
	Note that every  finite-dimensional linear subspace must be closed.  Hence, by \cref{thm:BAM:COMPO}\cref{thm:BAM:COMPO:BAM} and \cref{theor:averacomb:BAM}, we obtain that $\CC{\mathcal{S}}$ is a BAM.
The last inequality comes from \cref{prop:BAM:Properties}.
\end{proof}

\begin{theorem} \label{theom:CCS2:compose:comb}
	Suppose that $\mathcal{H} =\mathbb{R}^{n}$. 	 Set $\I:=\{1, \ldots, t\}$ and
	\begin{align*}
	(\forall i \in \I) \quad \mathcal{S}_{i} := \{\Id, T_{q_{i-1}+1}, T_{q_{i-1}+2}T_{q_{i-1}+1},  \ldots, T_{q_{i}}\cdots T_{q_{i-1}+2}T_{q_{i-1}+1}  \},
	\end{align*}
%	Set $\mathcal{S}_{1} := \{\Id, T_{q_{0}+1}, T_{q_{0}+2}T_{q_{0}+1},  \ldots, T_{q_{1}}\cdots T_{q_{0}+2}T_{q_{0}+1}  \}$, $\mathcal{S}_{2} := \{\Id, T_{q_{1}+1}, T_{q_{1}+2}T_{q_{1}+1},  \ldots, T_{q_{2}}\cdots T_{q_{1}+2}T_{q_{1}+1}  \}$,  $\ldots$, $\mathcal{S}_{t} := \{\Id, T_{q_{t-1}+1}, T_{q_{t-1}+2} T_{q_{t-1}+1},  \ldots, T_{q_{t}} \cdots   T_{q_{t-1}+2} T_{q_{t-1}+1} \}$,
	with $q_{0}=0, q_{t}=m$ and $(\forall i \in \I)$ $q_{i} -q_{i-1} \geq 1$.  Suppose that  one of the following holds:
	\begin{enumerate}
		\item  \label{theom:CCS2:compose:comb:comp}  $\CC{\mathcal{S}}=\CC{\mathcal{S}_{t}}\circ \CC{\mathcal{S}_{t-1}} \circ \cdots \circ \CC{\mathcal{S}_{1}}$.
		\item  \label{theom:CCS2:compose:comb:comb}  $\CC{\mathcal{S}}=\sum^{t}_{i =1}  \omega_{i} \CC{\mathcal{S}_{i}}$, where $\{ \omega_{i} \}_{1\leq i \leq t} \subseteq \left]0,1\right] $ such that $\sum^{t}_{i =1} \omega_{i}=1$.
	\end{enumerate}
Then $\Fix \CC{\mathcal{S}} =\cap^{m}_{i=1} \Fix T_{i}$ and $\CC{\mathcal{S}}$ is a  BAM. Moreover, there exists $\gamma \in \left[0,1\right[$ such that
\begin{align*}
(\forall x \in \mathcal{H})  (\forall k \in \mathbb{N})\quad \norm{\CC{\mathcal{S}}^{k}x - \Pro_{\cap^{m}_{i=1} \Fix T_{i}}x } \leq \gamma^{k} \norm{x - \Pro_{\cap^{m}_{i=1} \Fix T_{i}}x}.
\end{align*}
\end{theorem}	

\begin{proof}
	The proof is similar to that of \cref{theom:CCS1:compose:comb}.
\end{proof}

We conclude this section by presenting  BAMs from finite composition or convex combination of circumcenter mappings, which is not projections, in Hilbert spaces.
In fact, using \Cref{theorem:MAP:LC,theorem:SymMAP:LC}, one may construct more similar BAMs in  Hilbert space.
\begin{theorem} \label{theorem:Hilbert:BAM}
	Let $U_{1}, \ldots, U_{2m}$ be closed affine subspaces of  $\mathcal{H} $ with $\cap^{2m}_{i=1} U_{i} \neq \varnothing$. Set $\I:=\{1, \ldots, m\}$. Assume that $(\forall i \in \I)$ $\pa U_{2i-1} + \pa U_{2i}$ is closed.  Set
	\begin{subequations}
		\begin{align*}
		 \mathcal{S}_{1} := \{\Id, \R_{U_{q_{0}+1}}, \R_{U_{q_{0}+2}},  \R_{U_{q_{0}+2}}\R_{U_{q_{0}+1}} \}, \ldots,
		\mathcal{S}_{m}  := \{\Id, \R_{U_{q_{m}-1}}, \R_{U_{q_{m}}},  \R_{U_{q_{m}}}\R_{U_{q_{m}-1}}  \},
		\end{align*}
	\end{subequations}
	with $(\forall i \in \{0\} \cup \I)$ $q_{i} =2i$.
	Suppose that  one of the following holds:
	\begin{enumerate}
		\item  \label{theorem:Hilbert:BAM:comp}  $(\forall i \in \I )$ $ \sum^{2i}_{j=1}(\pa U_{j})^{\perp}$ is closed, and  $\CC{\mathcal{S}}=\CC{\mathcal{S}_{t}}\circ \CC{\mathcal{S}_{t-1}} \circ \cdots \circ \CC{\mathcal{S}_{1}}$.
		\item  \label{theorem:Hilbert:BAM:comb:comb}  $ \sum^{2m}_{j=1} (\pa U_{j})^{\perp}$ is closed, and  $\CC{\mathcal{S}}=\sum^{m}_{i =1}  \omega_{i} \CC{\mathcal{S}_{i}}$, where $\{ \omega_{i} \}_{1\leq i \leq m} \subseteq \left]0,1\right] $ such that $\sum^{m}_{i =1} \omega_{i}=1$.
	\end{enumerate}
	Then $\Fix \CC{\mathcal{S}} =\cap^{2m}_{i=1} \Fix T_{i}$ and $\CC{\mathcal{S}}$ is a  BAM.
	Moreover,  there exists $\gamma \in \left[0,1\right[$ such that
		\begin{align*}
		(\forall x \in \mathcal{H})  (\forall k \in \mathbb{N})\quad \norm{\CC{\mathcal{S}}^{k}x - \Pro_{\cap^{2m}_{i=1} U_{i}}x } \leq \gamma^{k} \norm{x - \Pro_{\cap^{2m}_{i=1} U_{i}}x}.
		\end{align*}
\end{theorem}

\begin{proof}
	By \cref{prop:S:SF:equivalence}, we are able to assume that $(\forall i \in \{1,\ldots,2m\})$ $U_{i}$ is closed linear subspace of $\mathcal{H}$.
	For every $( i \in \I)$, because $  U_{2i-1} +  U_{2i}$ is closed,  by \cref{cor:Hilbert:BAM}, $\CC{\mathcal{S}_{i}}$ is a BAM with $\Fix \CC{\mathcal{S}_{i}}= U_{2i-1} \cap U_{2i}$ and by \cref{fac:cFLess1}, $ U^{\perp}_{2i-1} +  U^{\perp}_{2i} =\overline{U^{\perp}_{2i-1} +  U^{\perp}_{2i}}$.
	Hence, for every $ i \in \I $,
	\begin{align*}
	\sum^{i}_{j=1}   (\pa \Fix \CC{\mathcal{S}_{j}})^{\perp} =\sum^{i}_{j=1}   (U_{2j-1} \cap U_{2j})^{\perp} = \sum^{i}_{j=1} \overline{U^{\perp}_{2i-1} +  U^{\perp}_{2i}} =\sum^{2i}_{j=1}(\pa U_{j})^{\perp}.
	\end{align*}
	 Therefore, the asserted results follow by
	 \cref{thm:BAM:COMPO}\cref{thm:BAM:COMPO:BAM} and \cref{theor:averacomb:BAM}.
\end{proof}

%%%%%%%%%%%%%%%%%%%%%%%%%%%%%%%%%%%%%%%%%%%%%%%%%%%%%%%%%%%%%%%%%%%%
%%%%%%%%%%%%%%%%\section{Conclusion and future work}%%%%%%%%%%%%%%%%%%%%%%%%
%%%%%%%%%%%%%%%%%%%%%%%%%%%%%%%%%%%%%%%%%%%%%%%%%%%%%%%%%%%%%%%%%%%%
\section{Conclusion and future work}
We discovered that the iteration sequence of BAM linearly converges to the best approximation onto the fixed point set of the BAM. We compared BAMs with linear convergent mappings, Banach contractions, and linear regular operators.
We also generalized the result proved by Behling, Bello-Cruz and Santos that the finite composition of BAMs with closed affine fixed point sets in $\mathbb{R}^{n}$ is still a BAM from $\mathbb{R}^{n}$ to the general Hilbert space. We  constructed new constant associated with the composition of BAMs.  Moreover, we proved that convex combinations of BAMs with closed affine fixed point sets is still a BAM. In addition, we connected BAMs with circumcenter mappings.

Although \cref{thm:BAM:COMPO} states that the  finite composition of BAMs  with closed affine fixed point sets   is still a BAM,
\cref{examp:BAM:counterexamp:Balls} shows that the composition of BAMs associated with closed Euclidean balls  is generally not a BAM.
Moreover,  \cref{prop:coneball} and  \Cref{exam:LineCone,examp:BAM:counterexamp:Balls}   illustrate that to determine whether the composition of BAMs is a BAM or not, the order of the BAMs does matter.
In addition, although  \Cref{theor:averacomb:BAM:shift,theor:averacomb:BAM} state that the convex combination of BAMs  with closed affine fixed point sets is a BAM, we have a little knowledge for affine  combinations of BAMs with general convex fixed point sets.
It would be interesting to characterize the sufficient conditions for the finite composition of  or affine combination of BAMs  with  general convex fixed point sets.
By \cref{rema:comp:BAM},  the constant  associated with the composition of BAMs in \cref{thm:ComposiBAM:Linear}\cref{thm:ComposiBAM:BAM} is not sharp. Using \cref{exam:alpha:PU}, we know that the constant associated with  the convex combination of BAMs presented \cref{theorem:convex:comb:BAM:2}
is not sharp as well. Hence, we will also try to find  better upper bound for the constant associated with the composition of  or the convex combination of BAMs.  As we mentioned in \cref{remark:assumpt:constants:different}, although the assumption of \cref{theor:averacomb:BAM:shift} is more restrictive than that of \cref{theor:averacomb:BAM}, the constants in these results are independent. We will investigate the relation between the constants associated with the convex combination of BAMs  in \cref{theor:averacomb:BAM:shift,theor:averacomb:BAM}.
Last but not least, we will try to find more BAMs  with general convex fixed point sets and more applications of those BAMs.
\section*{Acknowledgements}
HHB and XW were partially supported by NSERC Discovery Grants.

\addcontentsline{toc}{section}{References}

\bibliographystyle{abbrv}

\end{document}